\documentclass[12pt,a4paper]{article}
\usepackage[latin1]{inputenc}
\usepackage{amsmath,amsfonts,amssymb,amsthm,bm}
\usepackage{mathrsfs}
\usepackage{color}
\usepackage{enumitem}
\usepackage[francais, english]{babel}
\usepackage{graphicx}
\usepackage{graphics}
\usepackage[T1]{fontenc}
\usepackage{float}
\usepackage{placeins}
\graphicspath{{images/}}
\usepackage{fancyhdr}
\usepackage{mathpazo}
\usepackage{xcolor}
\newcommand{\nm}{\noalign{\smallskip}}

\definecolor{green2}{rgb}{0.3,0.6,0.4}

\fancypagestyle{plain}{ 
\fancyhead{} 

}

\newtheorem{lem}{Lemma}[section]
\newtheorem{prop}{Proposition}[section]
\newtheorem{cor}{Corollary}[section]
\newtheorem{thm}{Theorem}[section]
\newtheorem{rem}{Remark}[section]
\newtheorem{defi}{Definition}[section]

\numberwithin{equation}{section} \numberwithin{figure}{section}

\def\R{\mathbb{R}}
\def\C{\mathbb{C}}
\def\Z{\mathbb{Z}}
\def\E{\mathbb{E}}
\def\F{\mathcal{F}}
\def\Pb{\mathbb{P}}

\def\eps{\varepsilon}
\def\ueps{u_\eps}
\def\uepsp{u_\eps^+}
\def\uepsm{u_\eps^-}
\def\xiepsp{\xi_\eps^+}
\def\xiepsm{\xi_\eps^-}
\def\xieps{\xi_\eps}

\def\Deps{\Omega_\eps}
\def\Depsp{\Omega_\eps^+}
\def\Depsm{\Omega_\eps^-}
\def\interface{\Gamma_\eps}

\def\GGamma{{\bm \Gamma}}

\def\Weps{W_\eps}
\def\Hq{H_\C}

\def\mean{\mathcal{M}}

\def\Upsilon{\Phi(\Gamma)}

\title{Spectroscopic imaging  of a dilute cell suspension\thanks{\footnotesize This work was supported  by the
ERC Advanced
Grant Project MULTIMOD--267184.}}

\author{Habib Ammari\thanks{\footnotesize Department of Mathematics and Applications,
Ecole Normale Sup\'erieure, 45 Rue d'Ulm, 75005 Paris, France
(habib.ammari@ens.fr, laure.giovangigli@ens.fr,
wjing@dma.ens.fr).} \and Josselin Garnier\thanks{\footnotesize
Laboratoire de Probabilit\'es et Mod\`eles Al\'eatoires \&
Laboratoire Jacques-Louis Lions, Universit\'e Paris VII, 75205
Paris Cedex 13, France (garnier@math.jussieu.fr).} \and Laure
Giovangigli\footnotemark[2] \and Wenjia Jing\footnotemark[2] \and
Jin-Keun Seo\thanks{\footnotesize Department of Computational
Science and Engineering, Yonsei University 50 Yonsei-Ro,
Seodaemun-Gu, Seoul 120-749, Korea (seoj@yonsei.ac.kr).}}

\begin{document}

\maketitle


\begin{abstract}
A rigorous homogenization theory is derived to describe the
effective admittivity of cell suspensions. A new formula is
reported for dilute cases that gives the frequency-dependent
effective admittivity with respect to the membrane polarization.
Different microstructures are shown to be distinguishable via
spectroscopic measurements of the overall admittivity using the
spectral properties of the membrane polarization. The Debye
 relaxation times associated with the membrane polarization tensor are shown to be able
 to give the microscopic structure of the medium. A natural measure of the
 admittivity
anisotropy is introduced and its dependence on the frequency of
applied current is derived. A Maxwell-Wagner-Fricke formula is
given for concentric circular cells, and the results can be
extended to the random cases. A randomly deformed periodic medium
is also considered and a new formula is derived for the overall
admittivity of a dilute suspension of randomly deformed cells.

\end{abstract}

\bigskip

\noindent {\footnotesize Mathematics Subject Classification
(MSC2000): 35R30, 35B30.}


\tableofcontents


\selectlanguage{english}

\section{Introduction}

The electric behavior of biological tissue under the influence of
an electric field at frequency $\omega$ can be characterized by
its frequency-dependent effective admittivity $k_{ef}:=
\sigma_{ef}(\omega) + i \omega \epsilon_{ef}(\omega)$, where
$\sigma_{ef}$ and $\varepsilon_{ef}$ are respectively its
effective conductivity and permittivity. Electrical impedance
spectroscopy assesses the frequency dependence of the effective
admittivity by measuring it across a range of frequencies from a
few Hz to hundreds of MHz. Effective admittivity of biological
tissues and its frequency dependence vary with tissue composition,
membrane characteristics, intra-and extra-cellular fluids and
other factors. Hence, the admittance spectroscopy provides
information about the microscopic structure of the medium and
physiological and pathological conditions of the tissue.

The determination of the effective, or macroscopic, property of a
suspension is an enduring problem in physics \cite{miltonbook}. It
has been studied by many distinguished scientists, including
Maxwell, Poisson \cite{poisson}, Faraday, Rayleigh
\cite{rayleigh}, Fricke \cite{fricke53}, Lorentz, Debye, and
Einstein \cite{Ein1906}. Many studies have been conducted on
approximate analytic expressions for overall admittivity of a cell
suspension from the knowledge of pointwise conductivity
distribution, and these studies were mostly restricted to the
simplified model of a strongly dilute suspension of spherical or
ellipsoidal cells.


In this paper, we consider a periodic suspension of identical
cells of arbitrary shape. We apply at the boundary of the medium
an electric field of frequency $\omega$. The medium outside the
cells has an admittivity of $k_0:=\sigma_0+ i\omega
\epsilon_0$. Each cell is composed of an isotropic homogeneous
core of admittivity $k_0$ and a thin membrane of constant
thickness $\delta$ and admittivity $k_m := \sigma_m + i \omega
\epsilon_m $. The thickness $\delta$ is considered to be very
small relative to the typical cell size  and the membrane is
considered very resistive, {\it i.e.}, $\sigma_m \ll \sigma_0$. In
this context, the potential in the medium passes an effective
discontinuity over the cell boundary; the jump is proportional to
its normal derivative with a coefficient of the effective
thickness, given by $ \delta k_0 \, / k_m$. The normal derivative
of the potential is continuous across the cell boundaries.

We use homogenization techniques with asymptotic expansions to
derive a homogenized problem and to define an effective
admittivity of the medium. We prove a rigorous convergence of the
initial problem to the homogenized problem via two-scale
convergence.  For dilute cell suspensions, we use layer potential
techniques to expand the effective admittivity in terms of cell
volume fraction. Through the effective thickness, $\delta \, k_0/
k_m$, the first-order term in this expansion can be expressed in
terms of a membrane polarization tensor, $M$, that depends on the
operating frequency $\omega$. We retrieve the
Maxwell-Wagner-Fricke formula for concentric circular-shaped
cells. This explicit formula has been generalized in many
directions: in three dimension for concentric spherical cells; to
include higher power terms of the volume fraction for concentric
circular and spherical cells; and to include various shapes
 such as concentric, confocal ellipses and ellipsoids;
see
\cite{asami,asami2,fricke1,fricke2,fricke3,schwan1,schwan2,schwan3,
jinkeun}.

The imaginary part of $M$ is positive for $\delta$ small enough.
Its two eigenvalues are maximal for frequencies $1/\tau_i, i=1,2,$
of order of a few MHz with physically plausible parameters values.
This dispersion phenomenon well known by the biologists is
referred to as the $\beta$-dispersion. The associated
characteristic times $\tau_i$ correspond to Debye  relaxation
times. Given this, we show that different microscopic
organizations of the medium can be distinguished via $\tau_i,
i=1,2,$ alone. The relaxation times $\tau_i$ are computed
numerically  for different configurations: one circular or
elliptic cell, two or three cells in close proximity. The obtained
results illustrate the viability of imaging cell suspensions using
the spectral properties of the membrane polarization. The Debye
 relaxation times are shown to be able to give the microscopic
structure of the medium.

In the second part of this paper, we show that our results can be
extended to the random case by considering a randomly deformed
periodic medium. We also derive a rigorous homogenization theory
for cells (and hence interfaces) that are randomly deformed from a
periodic structure by random, ergodic, and stationary
deformations. We prove a new formula for the overall conductivity
of a dilute suspension of randomly deformed cells. Again, the
spectral properties of the membrane polarization can be used to
classify different microscopic structures of the medium through
their Debye  relaxation times. For recent works on effective
properties of dilute random media, we refer to \cite{almog,
beryland}.

Our results in this paper have potential  applicability in cancer
imaging, food sciences and biotechnology \cite{biotech, techno2},
and applied and environmental geophysics. They can be used to
model and improve the MarginProbe system for breast cancer
\cite{dune}, which emits an electric field and senses the
returning signal from tissue under evaluation. The greater
vascularization, differently polarized cell membranes, and other
anatomical differences of tumors compared with healthy tissue
cause them to show different electromagnetic signatures. The
ability of the probe to detect signals characteristic of cancer
helps surgeons ensure the removal of all unwanted tissue around
tumor margins.

Another commercial medical system to which our results can be
applied is ZedScan \cite{zilicoweb}. ZedScan is based on
electrical impedance spectroscopy for detecting neoplasias in
cervical disease \cite{zilico,zilico2}. Malignant white blood
cells can be also detected using induced membrane polarization
\cite{blood}. In food quality inspection, spectroscopic
conductivity imaging can be used to detect bacterial cells
\cite{food, food2}. In applied and environmental geophysics,
induced membrane polarization can be used to probe up to subsurface
depths of thousands of meters \cite{geo, geo2}.

The structure of the rest of this paper is as follows. Section
\ref{sec:setting} introduces the problem settings and state the
main results of this work. Section \ref{sect:analysis} is devoted
to the analysis of the problem. We prove existence and uniqueness
results and establish useful {\it a priori} estimates. In section
\ref{sect:homog} we consider a periodic  cell suspension and
derive spectral properties of the overall conductivity. In section
\ref{sec:dilutes} we consider the problem of determining the
effective property of a suspension of cells when the volume
fraction  goes to zero. Section \ref{sec:spectro} is devoted to
spectroscopic imaging of a dilute suspension. We make use of the
asymptotic expansion of the effective admittivity in terms of the
volume fraction to image a permittivity inclusion. We also discuss
selective spectroscopic imaging using a pulsed  approach. Finally,
we introduce a natural measure of the conductivity anisotropy and
derive its dependence on the frequency of applied current. In
section \ref{sec:stoch} we extend our results to the case of
randomly deformed periodic media. In section \ref{sect:numer} we
provide numerical examples that support our findings. A few
concluding remarks are given in the last section.

\section{Problem settings and main results} \label{sec:setting}

The aim of this section is to introduce the problem settings and
state the main results of this paper.

\subsection{Periodic domain}

We consider the probe domain $\Omega$ to be a bounded open set of
$\mathbb{R}^2$ of class $\mathcal{C}^2$. The domain contains a
periodic array of cells whose size is controlled by $\varepsilon$.
Let $C$ be a $\mathcal{C}^{2, \eta}$ domain being contained in the
unit square $Y=\displaystyle [0,1]^2$, see Figure~\ref{figdessin}.
Here, $0 < \eta <1$ and $C$ represents a reference cell. We divide
the domain $\Omega$ periodically in each direction in identical
squares $(Y_{\varepsilon, n})_n$ of size $\varepsilon$, where
\begin{equation*}
Y_{\varepsilon,n} = \varepsilon n + \varepsilon Y.
\end{equation*}
Here, $\displaystyle n \in N_{\varepsilon} :=\left \{ n \in
\mathbb{Z}^2 | Y_{\varepsilon,n} \cap \Omega \neq \emptyset \right
\}$.

We consider that a cell $\vspace{0.1cm} \displaystyle
C_{\varepsilon,n}$ lives in each small square $\displaystyle
Y_{\varepsilon, n}$.  As shown in Figure~\ref{figshematic1}, all
cells are identical, up to a translation and scaling of size
$\varepsilon$, to the reference cell $C$:
\begin{equation*}
\forall n \in N_{\varepsilon}, \hspace{0.3cm} C_{\varepsilon,n} =
\varepsilon n + \varepsilon \,C.
\end{equation*}
So are their boundaries $(\Gamma_{\varepsilon, n})_{n \in
N_{\varepsilon}}$ to the boundary $\Gamma$ of $C$:
\begin{equation*}
\forall n \in N_{\varepsilon}, \hspace{0.3cm}
\Gamma_{\varepsilon,n} = \varepsilon n + \varepsilon \,\Gamma.
\end{equation*}

Let us also assume that all the cells are strictly contained in
$\Omega$, that is for every $n \in N_{\varepsilon}$, the boundary
$\Gamma_{\varepsilon, n}$ of the cell $C_{\varepsilon, n}$ does
not intersect  the boundary $\partial \Omega$:
\begin{equation*}
\partial \Omega  \cap ( \displaystyle \bigcup_{n \in N_{\varepsilon}}\Gamma_{\varepsilon, n}) = \emptyset.
\end{equation*}

\subsection{Electrical model of the cell}

Set for any open set $D$ of $\mathbb{R}^2$ :  $$L^2_0(D) := \left
\{ f \in L^2(D) \Big | \displaystyle \int_{\partial D} f(x) ds(x)
=0 \right \}$$ and  $$H^1(D) := \left \{ f \in L^2(D) \Big |
|\nabla f| \in  L^2(D) \right \}.$$

We consider in this section the reference cell $C$ immersed in a
domain $D$. We apply a sinusoidal electrical current $g \in L_0^2(\partial D)$
with angular frequency $\omega$ at the boundary of $D$.

The medium outside the cell, $D\setminus \overline{C}$, is a
homogeneous isotropic medium with admittivity $k_0:=\sigma_0+
i\omega \epsilon_0$. The cell $C$ is composed of an isotropic
homogeneous core of admittivity $k_0$ and a thin membrane of
constant thickness $\delta$ with admittivity $k_m := \sigma_m + i
\omega \epsilon_m $. We make the following assumptions :
$$
\sigma_0 >0, \sigma_m >0, \epsilon_0 >0, \epsilon_m  \geq 0.
$$

If we apply a sinusoidal current $g(x)\sin(\omega t)$ on the boundary $\partial D$ in the low frequency range below $10$ MHz, the resulting time harmonic potential $\check{u}$ is governed approximately by 
\begin{equation*}
\left \{
\begin{array}{ll}
\vspace{0.3cm}\nabla \cdot (k_0 +(k_m-k_0)\chi_{ \Gamma^\delta })\nabla \check{u})=0 &\textrm{in } D\\
\displaystyle k_0\frac{\partial\check{u}}{\partial n}\Big |_{\partial D} =g, & 
\end{array}
\right.
\end{equation*}
where $\Gamma^\delta:=\{ x\in C~:~ \mbox{dist}(x, \Gamma)<\delta\}$ and 
$\chi_{ \Gamma^\delta }$ is the characteristic function of the set $\Gamma^\delta $.
   
The membrane thickness $\delta$ is considered to be very small compared to
the typical size $\rho$ of the cell  {\it i.e.} $\delta/\rho \ll 1$. According to the transmission condition, the normal component of the current density $\displaystyle k_0 \frac{\partial u}{\partial n}$ can be regarded as continuous across the thin membrane $\Gamma$. 
\vspace{0.3cm}

We set $\beta : = \displaystyle \frac{\delta}{k_m} $.
Since the membrane is very resistive, {\it i.e.} $\sigma_m/\sigma_0\ll 1$, the potential $u$ in $D$ undergoes a jump across
the cell membrane $\Gamma$, which can be approximated at first order by $\beta k_0 \displaystyle \frac{\partial u}{\partial n}$. A rigorous proof of this result, based on asymptotic expansions of layer potentials, can be found in \cite{Khelifi}.\vspace{0.4cm}

More precisely, $u$ is the
solution of the following equations:

\begin{equation} \label{modelcell}
\left \{
\begin{array}{l}
\begin{array}{ll}
\vspace{0.3cm}  \nabla \cdot k_0\nabla u= 0 \,&\, \textrm{in}\, \,D\setminus \overline{C},\\
\vspace{0.3cm}  \nabla \cdot k_0 \nabla u= 0 \,&\, \textrm{in}\, \,C,\\
\vspace{0.3cm} \displaystyle k_0 \frac{\partial u}{\partial n}\Big|_+ =
k_0 \frac{\partial u}{\partial n}\Big|_- \,&\, \textrm{on}\,\, \Gamma,\\
\vspace{0.3cm} u|_+- u|_- -\,\beta k_0 \displaystyle
\frac{\partial u}{\partial n} = 0 \hspace{0.7cm}\,&\,
\textrm{on}\, \,\Gamma,
\end{array} \\
\displaystyle k_0 \frac{\partial u}{\partial n} \Big|_{\partial D}
= g, \hspace{0.3cm}\displaystyle \int_{\partial D} g (x) ds(x)= 0,
\hspace{0.3cm}\displaystyle \int_{D \setminus \overline{C}} u (x)
dx =0.
\end{array}
\right .\end{equation} Here $n$ is the outward unit normal vector
and $u|_{\pm}(x)$ denotes $\lim\limits_{t \rightarrow 0^+} u(x \pm
t n(x))$ for $x$ on the concerned boundary. Likewise,
$\displaystyle\frac{\partial u}{\partial n}\Big|_{\pm} := \lim
\limits_{t \to 0^+} \nabla u(x \pm t n(x)) \cdot n(x)$.

\begin{figure}
\begin{center}
\includegraphics[scale=0.5]{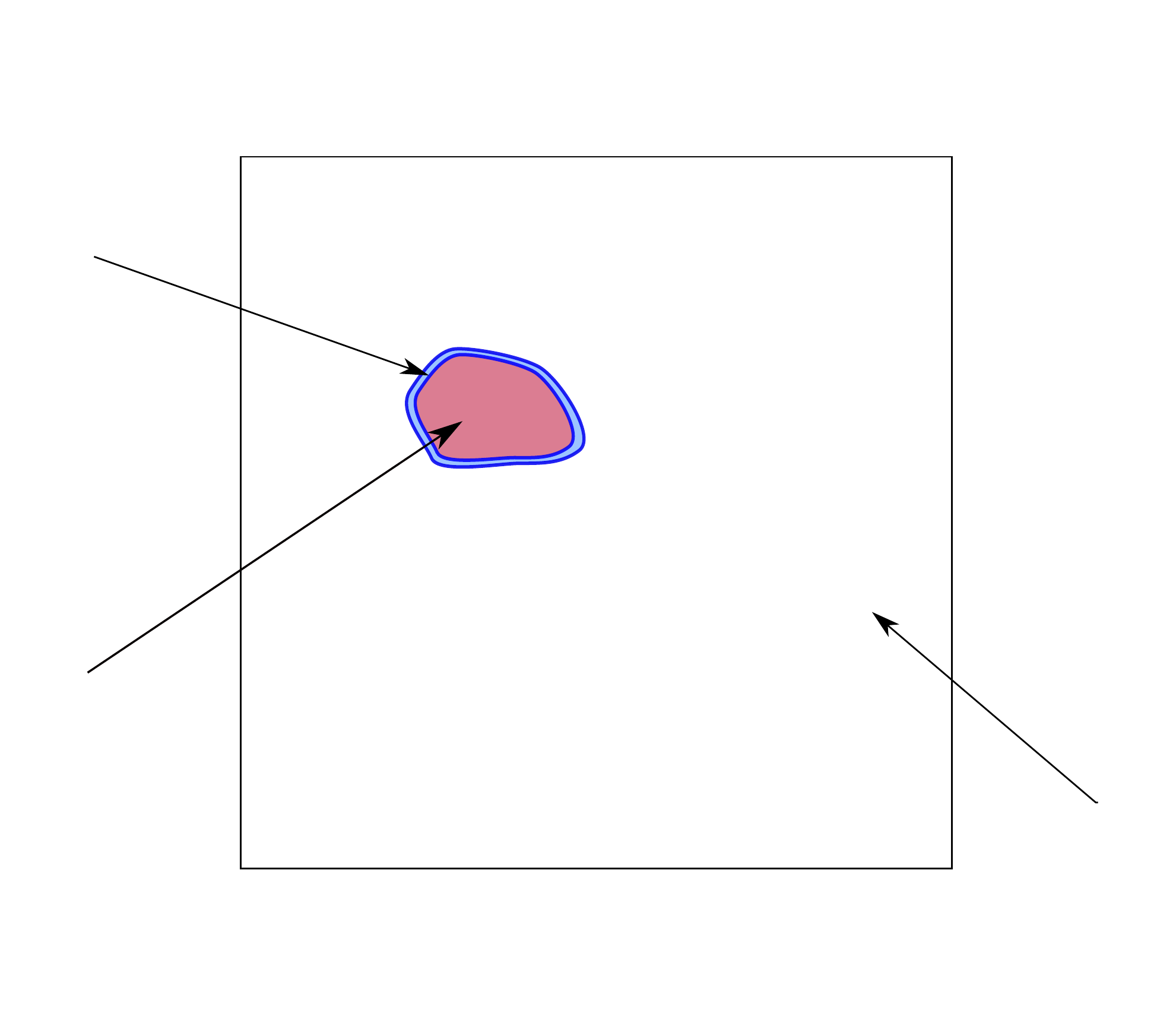}
\begin{picture}(0,0)
\put(-287,187){\textcolor{blue}{$\Gamma$}}
\put(-300,175){\textcolor{blue}{$(\delta, k_m)$}}
\put(-282,84){\textcolor{red}{$Y^-$}}
\put(-289,70){\textcolor{red}{$(k_0)$}} \put(-19,50){$Y^+$}
\put(-25,36){$(k_0)$}
\end{picture}
 \caption{\it{Schematic illustration of a unit period $Y$.}\label{figdessin}}
 \end{center}
\end{figure}
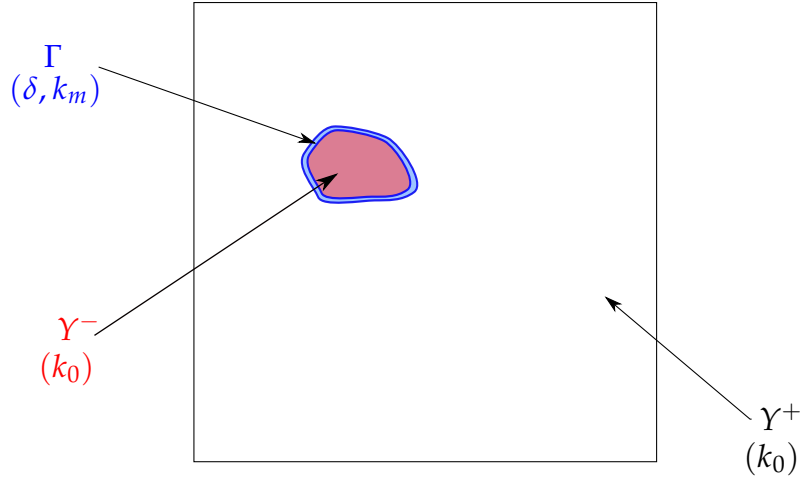

\subsection{Governing equation}

We denote by $\Omega_{\varepsilon}^+$ the medium outside the cells
and $\Omega_{\varepsilon}^-$ the medium inside the cells:
\begin{equation*}
\Omega_{\varepsilon}^+ = \displaystyle  \Omega  \cap ( \bigcup_{n
\in N_{\varepsilon}} Y_{\varepsilon, n} \setminus
\overline{C_{\varepsilon, n}}), \hspace{0.3cm}
\Omega_{\varepsilon}^- = \displaystyle \bigcup_{n \in
N_{\varepsilon}} C_{\varepsilon, n}.
\end{equation*}
Set $\vspace{0.1cm}\Gamma_{\varepsilon} := \displaystyle
\bigcup_{n \in N_{\varepsilon}}\Gamma_{\varepsilon, n}$. By
definition, the boundaries $\partial  \Omega_{\varepsilon}^+$ and
$\partial \Omega_{\varepsilon}^- $ of respectively $
\Omega_{\varepsilon}^+$ and $ \Omega_{\varepsilon}^-$ satisfy:
\begin{equation*}
\partial  \Omega_{\varepsilon}^+ = \partial \Omega \cup \Gamma_{\varepsilon}, \hspace{0.3cm} \partial  \Omega_{\varepsilon}^- = \Gamma_{\varepsilon}.
\end{equation*}

We apply a sinusoidal current $g(x)\sin (\omega t)$ at $x\in
\partial \Omega$, where $g \in L_0^2(\partial \Omega)$. The
induced time-harmonic potential $u_\varepsilon$ in $\Omega$
satisfies \cite{6,poig2,poig3}:

\begin{equation}\label{eq:u_{epsilon}}
\left \{
\begin{array}{l}
\begin{array}{ll}
\vspace{0.3cm}  \nabla \cdot k_0\nabla u_{\varepsilon}^+  = 0
\,&\,
 \textrm{in}\, \,\Omega_{\varepsilon} ^+,\\
\vspace{0.3cm}  \nabla \cdot k_0 \nabla u_{\varepsilon}^-  = 0 \,&\, \textrm{in}\, \,\Omega_{\varepsilon} ^-,\\
\vspace{0.3cm} \displaystyle k_0 \frac{\partial u_{\varepsilon} ^+}{\partial n}
=  k_0 \frac{\partial u_{\varepsilon} ^-}{\partial n} \,&\, \textrm{on}\,\, \Gamma_{\varepsilon},\\
\vspace{0.3cm} u_{\varepsilon} ^+ - u_{\varepsilon} ^- -
\varepsilon \,\beta k_0 \displaystyle \frac{\partial
u_{\varepsilon} ^+}{\partial n} = 0 \hspace{0.7cm}\,&\,
\textrm{on}\, \,\Gamma_{\varepsilon},
\end{array}\\
\vspace{0.3cm} \displaystyle k_0 \frac{\partial u_{\varepsilon}^+
}{\partial n} \Big|_{\partial \Omega} = g,
\hspace{0.3cm}\displaystyle \int_{\partial \Omega} g(x) ds(x)= 0,
\hspace{0.3cm}\displaystyle \int_{\Omega_{\varepsilon} ^+}
u_{\varepsilon} ^+(x) dx= 0,
\end{array}
\right .\end{equation} where $u_{\varepsilon} = \left\{
\begin{array}{ll}
\vspace{0.2 cm} u_{\varepsilon}^+ \, & \, \textrm{in} \, \, \Omega_{\varepsilon} ^+,\\
\vspace{0.2 cm} u_{\varepsilon}^- \, & \, \textrm{in} \, \,
\Omega_{\varepsilon} ^- .
\end{array}
\right.$

Note that the previously introduced constant $\beta$, {\it i.e.},
the ratio between the thickness of the membrane of $C$ and its
admittivity, becomes $\varepsilon \beta$. Because the cells
$\vspace{0.1 cm} (C_{\varepsilon,n})_{n \in N_{\varepsilon}}$ are
in squares of size $\varepsilon$, the thickness of their membranes
is given by $\varepsilon \delta$ and consequently, a factor
$\varepsilon$ appears.\vspace{0.5cm}

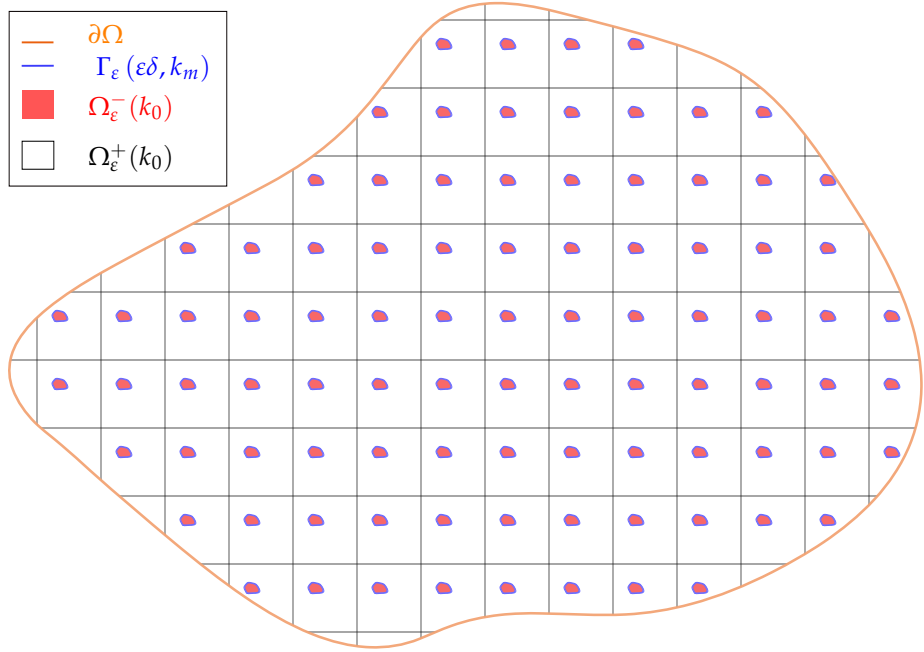
\begin{figure}
\fontsize{9.7pt}{7.2}
\includegraphics[scale=0.9]{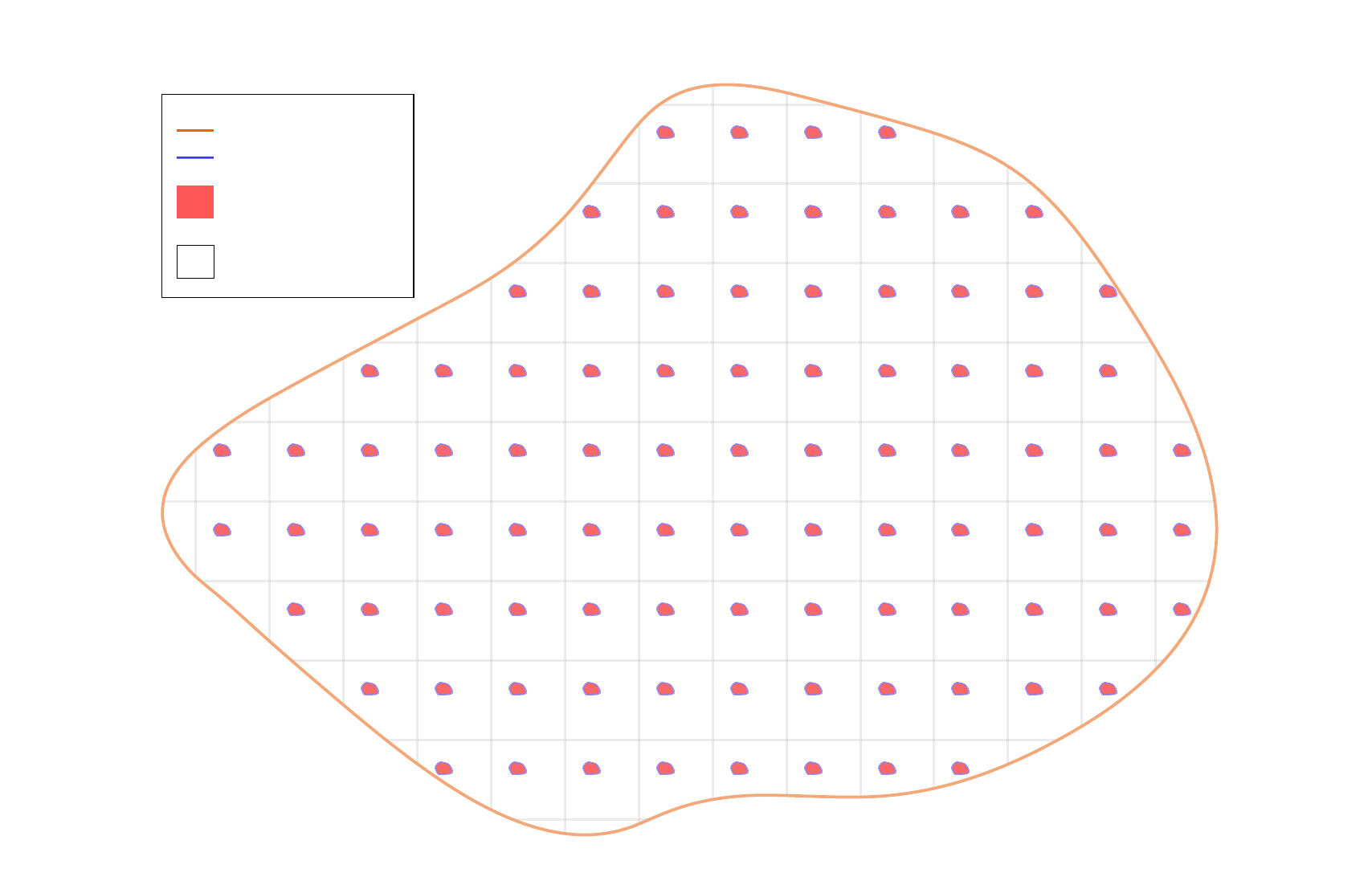}
\begin{picture}(0,0)
\put(-365,244){\textcolor{orange}{$\partial \Omega$}}
\put(-362,232){\textcolor{blue}{$\Gamma_{\varepsilon} \,(\varepsilon
\delta, k_m)$}}
\put(-365,217){\textcolor{red}{$\Omega_{\varepsilon}^- (k_0)$}}
\put(-365,198){$\Omega_{\varepsilon}^+ (k_0)$}
\end{picture}
 \caption{\it{Schematic illustration of the periodic medium $\Omega$.} \label{figshematic1}}
\end{figure}

\subsection{Main results in the periodic case}

We set $Y^+ := Y \setminus \overline{C}$ and $Y^- := C$.
For any open set $D$ in $\mathbb{R}^2$, we denote $H^1_{\mathbb{C}}(D)$ the Sobolev space $H^1(D)/\mathbb{C}$ which can be represented as :
\begin{equation*}
H^1_{\mathbb{C}}(D) = \left\{ u \in H^1(D) ~|~ \int_{D} u(x) dx
= 0\right\}.
\end{equation*}
Throughout this paper, we assume that $ \mathrm{ dist } (Y^-,
\partial Y) =O(1)$. We write the solution $u_{\varepsilon}$ as
\begin{equation} \label{ansatz}
\forall x \in \Omega\, \, \, u_{\varepsilon}(x) = u_0(x) +
\varepsilon u_1(x, \frac{x}{\varepsilon}) + o(\varepsilon),
\end{equation}
with
\begin{equation*}
y \longmapsto u_1(x,y)\, Y\textrm{-periodic} \, \, \textrm{and} \,
\, u_1(x,y) = \left \{
\begin{array}{l}
\vspace{0.2cm} u_1^+(x,y) \, \, \textrm{in} \, \Omega \times Y^+ ,\\
u_1^-(x,y) \, \, \textrm{in} \, \Omega \times Y^- .
\end{array}
\right .
\end{equation*}
The following theorem holds.
\begin{thm} \label{thm:homo}
\begin{enumerate}
\item[{\upshape (i)}] The solution $\ueps$ to
\eqref{eq:u_{epsilon}} two-scale converges to $u_0$ and $\nabla
u_{\varepsilon}(x)$ two-scale converges to $\nabla u_0(x) +
\chi_{Y^+}(y) \nabla_y u_1^+(x,y) + \chi_{Y^-}(y) \nabla_y
u_1^-(x,y)$, where $\chi_{Y^{\pm}}$ are the characteristic
functions of $Y^{\pm}$. \item[{\upshape (ii)}] The function $u_0$
in \eqref{ansatz} is the solution in $H^1_{\mathbb{C}}(\Omega)$ to the following homogenized
problem:
\begin{equation}\label{eq:u_0}
\left \{
\begin{array}{ll}
\vspace{0.3cm} \nabla \cdot K^* \, \nabla u_0(x) = 0 \, \, & \, \, \textrm{in} \,\Omega, \\
\displaystyle n\cdot K^*\nabla u_0  = g \, \, & \, \, \textrm{on}
\,\partial \Omega,
\end{array}
\right .
\end{equation}
where $K^*$, the effective admittivity of the medium, is given by
\begin{equation}\label{eq:K^*}
\displaystyle \forall (i, j) \in \{1,2\}^2, \hspace{1cm}K^* _{i, j} =k_0 \left(\delta_{ij} + \int_{Y}
(\chi_{Y^+} \nabla w_i^+ + \chi_{Y^-} \nabla w_i^-)\cdot
e_j\right),
\end{equation}
and the function $(w_i)_{i=1,2}$ are the solutions of the
following cell problems:
\begin{equation}\label{eq:w_i}
\left \{
\begin{array}{ll}
\vspace{0.3cm} \nabla \cdot k_0 \nabla (w_i^+(y) + y_i) = 0 \,&\, \textrm{in} \,\,Y^+,\\
\vspace{0.3cm} \nabla \cdot k_0 \nabla (w_{i}^- (y) + y_i) = 0 \,&\, \textrm{in} \,\,Y^-,\\
\vspace{0.3cm} \displaystyle k_0 \frac{\partial}{\partial n} ( w_{i}^+(y) + y_i) = k_0
\frac{\partial}{\partial n}( w_{i}^-(y) + y_i) \,&\, \textrm{on}\, \,\Gamma ,\\
\vspace{0.3cm} w_{i}^+ - w_{i}^- - \beta k_0 \displaystyle \frac{\partial}{\partial
n}(w_{i}^+(y) + y_i) = 0 \hspace{0.7cm}\,&\, \textrm{on} \,\,\Gamma ,\\
y \longmapsto w_{i}(y) \, \, Y\textrm{-periodic}. &
\end{array}
\right .
\end{equation}
\item[{\upshape (iii)}] Moreover, $u_1$ can be written as
\begin{equation}\label{eq:u_1}
\displaystyle \forall (x, y) \in \Omega \times Y , \, \, \, u_1(x,
y) = \sum_{i=1}^2 \frac{\partial u_0}{\partial x_i}(x) w_i(y).
\end{equation}
\end{enumerate}
\end{thm}

We define the integral operator $\mathcal{L}_\Gamma: \mathcal{C}^{2,
\eta}(\Gamma) \rightarrow \mathcal{C}^{1, \eta}(\Gamma)$, with $0<\eta<1$ by
\begin{equation} \label{defL}
\mathcal{L}_\Gamma[\varphi](x) = \frac{1}{2\pi} \int_{\Gamma}
\frac{\partial^2 \ln |x-y|}{\partial n(x) \partial n(y)}
\varphi(y) ds(y), \quad x \in \Gamma.
\end{equation}

$\mathcal{L}_\Gamma$ is the normal derivative of the double layer potential $\mathcal{D}_{\Gamma}$.

Since $\mathcal{L}_\Gamma$ is positive, one can prove that the operator $I +
\alpha \mathcal{L}_\Gamma: \mathcal{C}^{2, \eta}(\Gamma) \rightarrow
\mathcal{C}^{1, \eta}(\Gamma)$ is a bounded operator and has a
bounded inverse provided that $\Re \, \alpha >0$ \cite{14,2}.

As the fraction $f$ of the volume occupied by the cells goes to zero,
we derive an expansion of the effective admittivity  for arbitrary
shaped cells in terms of the volume fraction. We refer to the
suspension, as  periodic dilute. The following theorem holds.
\begin{thm} \label{mainhomog}
The effective admittivity of a periodic dilute suspension admits
the following asymptotic expansion:
\begin{equation} \label{dilutethmf}
\displaystyle K^* = k_0 \left(I + f M \left ( I -
\frac{f}{2} M \right )^{-1}\right) + o(f^2),
\end{equation}
where $\rho = \sqrt{|Y^-|}$, $f = \rho^2$,
\begin{equation} \label{defM}
M =\left(M_{ij} = \displaystyle\beta k_0 \int_{\rho^{-1}\Gamma}
n_j \psi_i^*(y) ds(y)\right )_{(i, j) \in \{1,
2\}^2},\end{equation} and $ \psi_i^*$ is defined by
\begin{equation} \label{defpsii}
\psi_i^* =  - \left ( I + \beta k_0 \mathcal{L}_{\rho^{-1}\Gamma} \right
)^{-1}[n_i]. \end{equation}
\end{thm}

\subsection{Description of the random cells and interfaces}

We describe the domains occupied by the cells. As mentioned
earlier, they are formed by randomly deforming a periodic
structure. We transform the aforementioned periodic structure by a
random diffeomorphism $\Phi: \R^2 \to \R^2$. Let
\begin{equation}
\R_2^+ := \bigcup_{n\in \Z^2} (n + Y^+), \quad \R_2^- :=
\bigcup_{n\in \Z^2} (n + Y^-), \quad \GGamma_2 := \bigcup_{n\in
\Z^2} (n + \Gamma).
\end{equation}
The cells, the environment and the interfaces are hence deformed
to $\Phi(\R_2^-)$, $\Phi(\R_2^+)$ and $\Phi(\GGamma_2)$. We
emphasize that the topology of these sets are the same as before.
Finally, the deformed structure is scaled to size $\eps$, where
$0< \eps \ll 1$, by the dilation operator $\eps {\bf I}$ where
${\bf I}$ is the identity operator. The final sets $\eps
\Phi(\R_2^-)$, $\eps \Phi(\GGamma_2)$ and $\eps \Phi(\R^+_2)$ thus
are realistic models for the random cells, membranes and the
environment for the biological problem at hand.

To model the cells inside an arbitrary bounded domain $\Omega$ as
in \eqref{eq:u_{epsilon}}, we would like to set $\Depsp := \Omega
\cap \eps \Phi(\R_2^-)$ and $\interface := \Omega \cap \eps
\Phi(\GGamma_2)$. However, a technicality is encountered,
precisely, the intersection of $\eps \Phi(\GGamma_2)$ with the
boundary $\partial \Omega$ may not be empty. In this case, some
 cells are cut by the boundary of the body, which is not physically
admissible. Moreover, an arbitrary diffeomorphism
$\Phi$ may allow some deformed cells in $\eps \Phi(\R_2^-)$ to get
arbitrarily close to each other. This imposes difficulties for
rigorous mathematical analysis. In order to resolve these issues,
we will impose a few conditions on $\Phi$ and refine the above
construction in the next subsection.

\subsection{Stationary ergodic setting}

Let $(\mathcal{O},\F,\Pb)$ be some probability space on which
$\Phi(x,\gamma): \R^2 \times \mathcal{O} \to \R^2$ is defined. For
a random variable $X \in L^1(\mathcal{O},d\Pb)$, we will denote
its expectation by $$\E X = \int_{\mathcal{O}} X(\gamma)
d\Pb(\gamma).$$

Throughout this paper, we assume that the group $(\Z^2,+)$ acts on
$\mathcal{O}$ by some action $\{\tau_n: \mathcal{O} \to
\mathcal{O}\}_{n \in \Z^2}$, and that for all $n \in \Z^2$,
$\tau_n$ is $\Pb$-preserving, that is, $$\Pb(A) = \Pb(\tau_n A),
\quad \mbox{for all } A \in \F.$$ We assume further that the
action is {\itshape ergodic}, which means that for any $A \in \F$,
if $\tau_n A = A$ for all $n \in \Z^2$, then necessarily $\Pb(A)
\in \{0,1\}$.

Following \cite{BLBL06}, we say that a random process $F \in
L^1_{\rm{loc}}(\R^2,L^1(\mathcal{O}))$ is (discrete) {\itshape
stationary} if
\begin{equation}
\forall n \in \Z^2, \quad F(x+n, \gamma) = F(x,\tau_n \gamma)
\quad \text{for almost every } x \text{ and } \gamma.
\label{eq:stati}
\end{equation}
Clearly, a deterministic periodic function is a special case of
stationary process. However, we precise that the above notion of
stationarity is different from the classical one, see for instance
\cite{PV79} and \cite{Kozlov}. Throughout this paper, we presume
stationarity in the sense of \eqref{eq:stati} if not stated
otherwise. What makes this notion useful is the following version
of ergodic theorem \cite{Dunford, Jikov_book}.

\begin{prop}
Let $F \in L^\infty(\R^2,L^1(\mathcal{O}))$ be a stationary random
process. Equip $\Z^2$ with the norm $|n|_\infty = \max_{1\le i\le
2} |n_i|$ for all $n \in \Z^2$. Then
\begin{equation}
\frac{1}{(2N+1)^2} \sum_{|n|_\infty \le N} F(x,\tau_n \gamma)
\xrightarrow[N \to \infty]{L^\infty} \E F(x,\cdot) \quad \text{
for a.e.}\ \gamma \in \mathcal{O}.
\end{equation}
This implies in particular that if the family
$\{F(\frac{\cdot}{\eps},\gamma)\}$ is bounded in $L^p_{\rm
loc}(\R^2)$, for some $p \in [1,\infty)$, then
\begin{equation}
F\left(\frac{x}{\eps}, \gamma \right) \overset{}{\underset{\eps
\to 0}{\rightharpoonup}} \E\left(\int_Y F(x,\cdot) dx \right)
\text{ in } L^p_{\rm loc}(\R^2) \text{ for a.e.}\ \gamma \in
\mathcal{O}.
\end{equation}
The convergence holds also in the weak-$*$ sense for $p = \infty$.
\end{prop}

We assume that for every $\gamma \in \mathcal{O}$,
$\Phi(\cdot,\gamma)$ is a diffeomorphsim from $\R^2$ to $\R^2$ and
that it satisfies
\begin{equation}
\nabla \Phi(x,\gamma) \text{ is stationary.} \label{eq:Phic1}
\end{equation}
\begin{equation}
\underset{\gamma \in \mathcal{O}, x \in \R^2}{\mathrm{ess\ inf}} \
\mathrm{det} (\nabla \Phi(x,\gamma)) = \kappa > 0,
\label{eq:Phic2}
\end{equation}
\begin{equation}
\underset{\gamma \in \mathcal{O}, x \in \R^2}{\mathrm{ess\ sup}} \
|\nabla \Phi(x,\gamma)|_F = \kappa^\prime > 0, \label{eq:Phic3}
\end{equation}
where  $|\,\cdot\,|_F$ is the Frobenius norm and ${\mathrm{ess\
inf}}$ and ${\mathrm{ess\ sup}}$ are the essential infimum and the
essentiel supremum, respectively. To avoid the intersection of
$\partial \Omega$ and the random cells $\eps \Phi(\R_2^-)$ and the
collision of cells, that is when two connected components of $\eps
\Phi(\R_2^-)$ get as close as $o(\eps)$, we need the further
modification in the construction of cells. To this end, we assume
further that
\begin{equation}
\|\Phi(\cdot,\gamma) - \mathbf{I}(\cdot)\|_{L^\infty(\R^2)} \le
\frac{ \mathrm{ dist } (Y^-, \partial Y)}{2} \text{ for a.e. }
\gamma \in \mathcal{O}. \label{eq:Phic4}
\end{equation}
 Note that this implies also that $\|\Phi^{-1} -
\mathbf{I}\|_{L^\infty} \le  \mathrm{ dist } (Y^-, \partial Y)/2$
a.s. in $\mathcal{O}$. Now, given a bounded and simply connected
open set $\Omega$ with smooth boundary and a small number $\eps
\ll 1$, we denote by $\Omega_{1/\eps}$ the scaled set $\{x \in
\R^2 ~|~ \eps x \in \Omega\}$. Let $\widetilde{\Omega_{1/\eps}}$
be the shrunk set
\begin{equation*}
\widetilde{\Omega_{1/\eps}} := \{x \in \Omega_{1/\eps} ~|~ \mathrm{
dist } (x,
\partial \Omega_{1/\eps}) \ge  \mathrm{ dist } (Y^-, \partial Y)\}.
\end{equation*}
We introduce for $n \in \mathbb{Z}^2,\, \, Y_n$ and $Y_n^\pm$ the
translated cubes, reference cells and reference environments: $Y_n
:= n+Y, Y_n^\pm := n+Y^\pm$. Let $\mathcal{I}_\eps \subset \Z^2$
be the indices of cubes $Y_n$ such that $Y_n \in
\widetilde{\Omega_{1/\eps}}$. Note that
$\mathcal{I}_{\varepsilon}$ corresponds to $N_{\varepsilon}$ in
the periodic case. We set $\Depsm$ to be
\begin{equation}
\Depsm := \sum_{n \in \mathcal{I}_\eps} \eps \Phi(Y_n^-)
\label{eq:Depsmdef}
\end{equation}
and then $\Depsp = \Omega \setminus \overline{\Depsm}$. We also
define the following two notations:
\begin{equation}
E_\eps := \sum_{n \in \mathcal{I}_\eps} \eps \Phi(Y_n) \quad
\text{and } \quad K_\eps := \Omega \setminus \overline{E_\eps}.
\label{eq:KEdef}
\end{equation}
Clearly, $E_\eps$ encloses all the cells in $\eps\Phi(Y_n^-), n
\in \mathcal{I}_\eps$ and their immediate surroundings
$\eps\Phi(Y_n^+)$; $K_\eps$ is a cushion layer near the boundary
that prevents the cells from touching the boundary. From the
construction we see that
\begin{equation}
\inf_{x \in \Depsm} \mathrm{ dist }(x,\partial \Omega) \ge \eps
\mathrm{ dist } (Y^-, \partial Y) \quad \text{ and }\quad \sup_{x
\in K_\eps} \mathrm{ dist }(x,
\partial \Omega) \le (3  \mathrm{ dist } (Y^-, \partial
Y)+\sqrt{2}) \eps.
\end{equation}
Furthermore, we can check that
\begin{equation}
\sup_{n, j \in \mathcal{I}_\eps, n\ne j}\quad \inf_{x \in \eps
\Phi(Y_n^-), y \in \eps \Phi(Y_j^-)} |x - y| \ge  \mathrm{ dist }
(Y^-, \partial Y) \eps.
\end{equation}
This shows that the cells in $\Omega$ are well separated, {\it
i.e.}, with a distance comparable to (if not much larger than) the
size of the cells, see Figure~\ref{figshematic2}.

\begin{figure}
\fontsize{9.7pt}{7.2}
\includegraphics[scale=0.9]{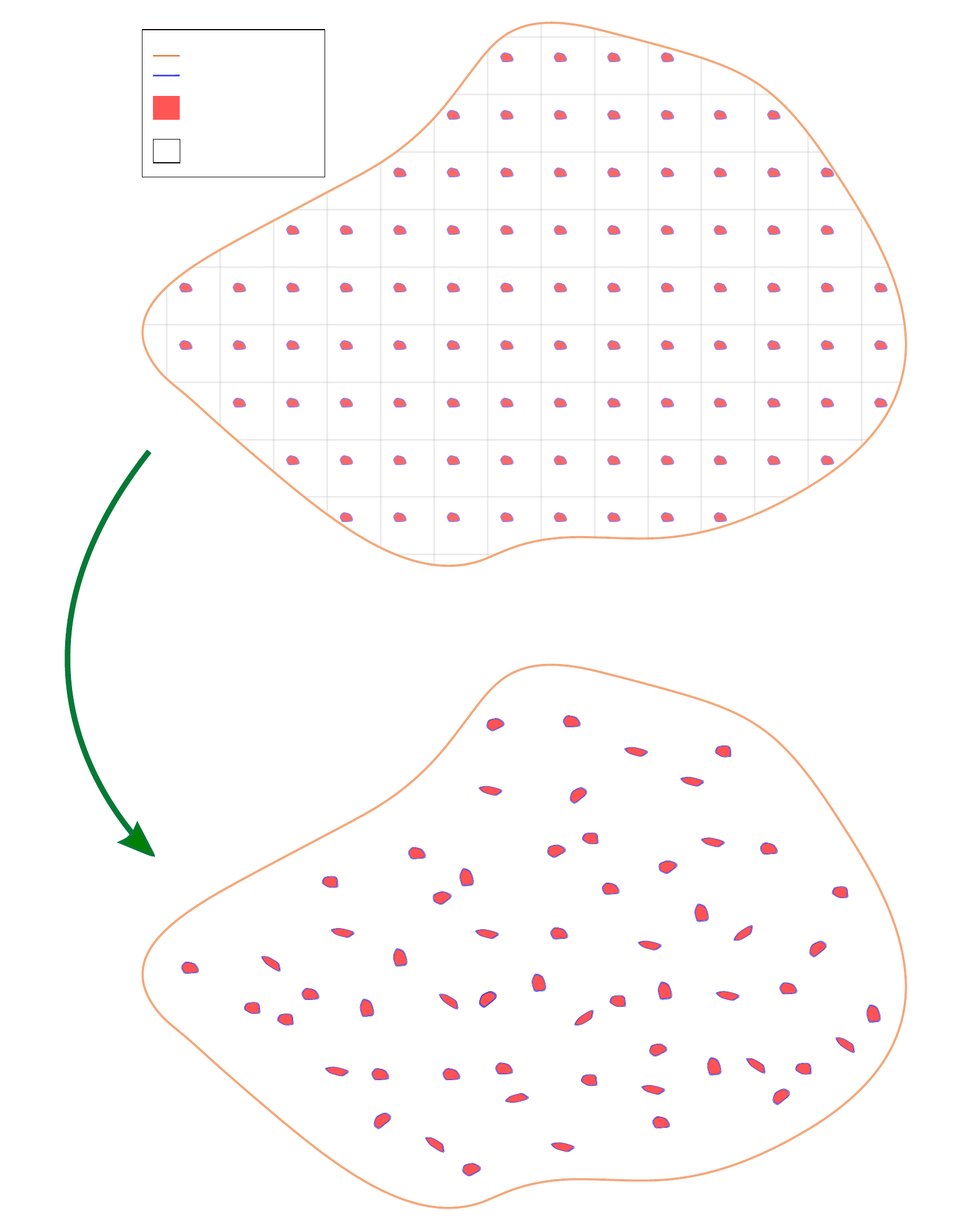}
\begin{picture}(0,0)
\put(-335,524){\textcolor{orange}{$\partial \Omega$}}
\put(-332,512){\textcolor{blue}{$\Gamma_{\varepsilon} \,(\varepsilon
\delta, k_m)$}}
\put(-335,497){\textcolor{red}{$\Omega_{\varepsilon}^- (k_0)$}}
\put(-335,478){$\Omega_{\varepsilon}^+ (k_0)$}
\put(-420,245){\textcolor{green2}{$\Phi$}}
\end{picture}
 \caption{\it{Schematic illustration of the randomly deformed periodic medium $\Omega$.} \label{figshematic2}}
\end{figure}
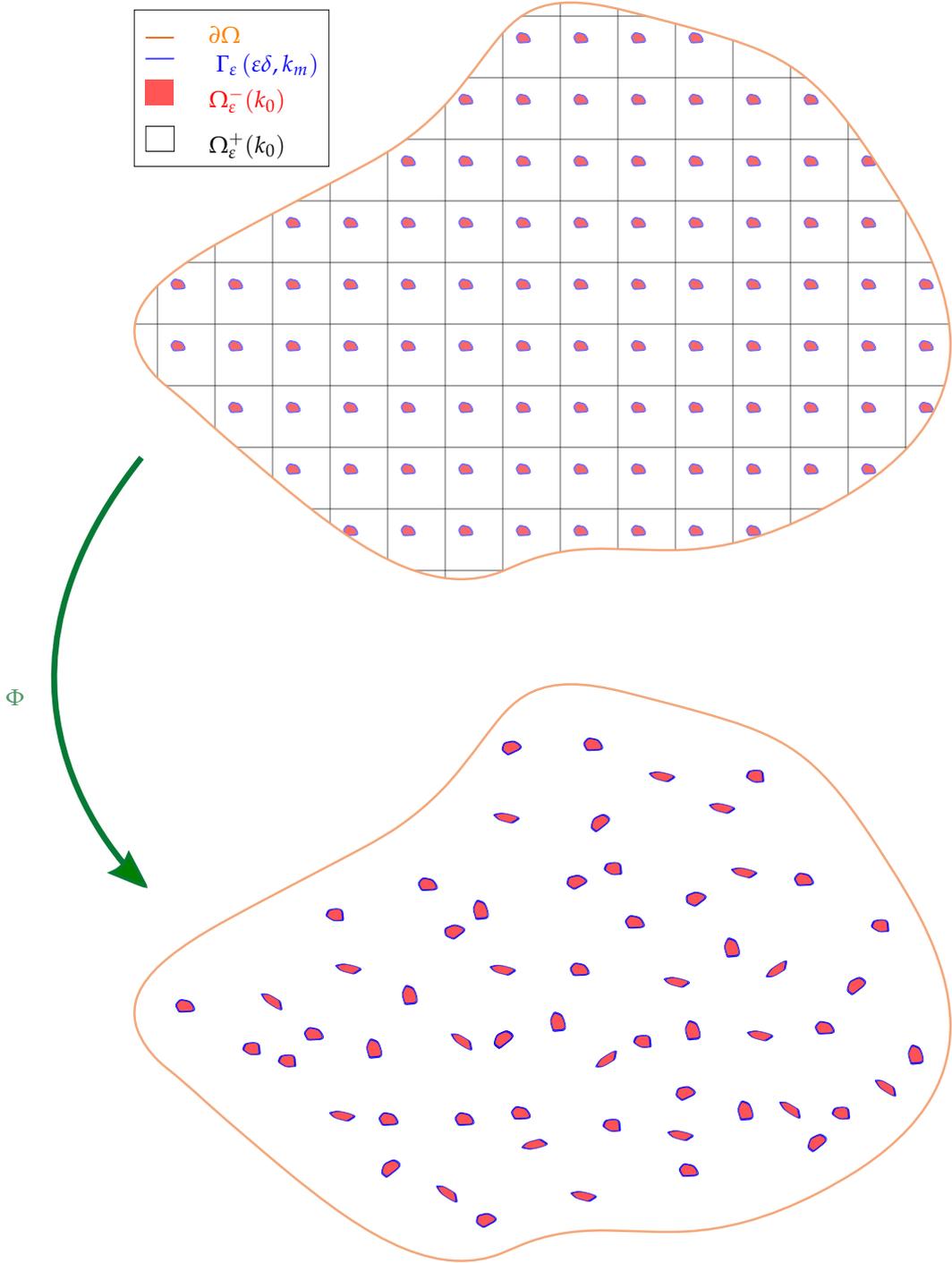

\subsection{Main results in the random case}

The first important result in the random case concerns an
auxiliary problem which produces oscillating test functions that
are used in the stochastic homogenization procedure. In the
following theorem, a function $f^{\rm ext}$ in $W_{\rm
loc}^{1,s}(\R^2)$ is said to be an extension of $f \in W_{\rm
loc}^{1,s}(\R_2^+)$ if $f^{\rm ext} = f$ on $\R_2^+$ and $\|f^{\rm
ext}\|_{W^{1,s}(K)} \le C(K,\R_2^+)\|f\|_{W^{1,s}(\R_2^+ \cap
K)}$, for any compact subset $K$.

The following theorem holds.
\begin{thm}\label{thm:auxiliary} Let $\Phi(\cdot,\gamma)$ be a random process defined on the probability space
 $(\mathcal{O},\F,\Pb)$, valued in the space of diffeomorphisms from $\R^2$ to $\R^2$ satisfying
 \eqref{eq:Phic1}-\eqref{eq:Phic3}.
 Then for a.e. $\gamma \in \mathcal{O}$ and for an arbitrary fixed $p \in \R^2$, the system
\begin{equation}
\left\{
\begin{array}{l}
\begin{array}{ll}
\vspace{0.3cm}\nabla\cdot k_0 (\nabla w^+_p(y) + p) = 0 \hspace{3cm}&\text{ in } \Phi(\R^\pm_2,\gamma), \\
\vspace{0.3cm}\nabla\cdot k_0 (\nabla w^-_p(y) + p) = 0 &\text{ in } \Phi(\R^\pm_2,\gamma),\\
\vspace{0.3cm} k_0 \displaystyle \frac{\partial w^+_p}{\partial n}(y) - \frac{\partial w^-_p}{\partial n}(y)=0& \text{ on
} \Phi(\Gamma_2,\gamma), \\
\vspace{0.3cm} w^+_p
- w^-_p - \beta k_0 \displaystyle \frac{\partial w^+_p}{\partial n}(y) = 0 &\text{ on }
\Phi(\Gamma_2,\gamma),\\
\vspace{0.3cm}w^\pm_p(y,\gamma) = \widetilde{w}^\pm_p(\Phi^{-1}(y,\gamma),\gamma),\\
\vspace{0.3cm}\nabla \widetilde{w}^\pm_p \text{ are stationary},&
\end{array}\\
\exists \tilde{w}^{\rm ext}_p \in H^1_{\rm loc}(\R^2) \text{
extends } \tilde{w}^+_p \quad \text{ s.t. }  \E \left(\int_{Y}
\nabla \tilde{w}^{\rm ext}_p (\tilde{y},\cdot) d\tilde{y} \right)
= 0,
\end{array}
\right. \label{eq:auxiliary}
\end{equation}
admits a unique solution (up to an additive constant) $w_p = w^+_p
\chi_{\Phi(\R_2^+)} + w^-_p \chi_{\Phi(\R_2^-)}$, where $w^\pm_p
\in H^1_{\rm loc}(\Phi(\R_2^\pm))$.
\end{thm}

Clearly, one could add another condition to the above problem,
namely that the integral of $\tilde{w}^+_p$ in $Y^+$ vanishes, to
fix the additional constant. The second main result in the random
case is the following homogenization theorem.

\begin{thm}\label{thm:homog} Let $\Omega$ be a bounded and connected open subset of $\R^2$ with regular boundary.
Let $\Phi$ be a random diffeomorphism on $(\mathcal{O},\F,\Pb)$
satisfying \eqref{eq:Phic1}-\eqref{eq:Phic4}. Assume that the
cells $\Depsm$ are constructed as in the previous section. Then
for a.e. $\gamma \in \mathcal{O}$, the solution
$\ueps(\cdot,\gamma) = (\uepsp, \uepsm)$ of \eqref{eq:u_{epsilon}}
satisfies the following properties:
\begin{enumerate}
\item[{\upshape (i)}] We can extend $\uepsp(\cdot,\gamma)$ to
$\ueps^{\rm ext}(\cdot,\gamma) \in H^1(\Omega)$, where $\ueps^{\rm
ext}(\cdot,\gamma)$ converges weakly, as $\eps \to 0$, to some
deterministic function $u_0 \in H^1(\Omega)$. \item[{\upshape
(ii)}] The function $\ueps(\cdot,\gamma)$ converges strongly in
$L^2$ to $u_0$ above. Further, let $Q$ be the trivial extension
operator setting $Qf = 0$ outside the domain of $f$, and define
$$
\varrho := \det\left(\E \int_{Y} \nabla \Phi(z,\cdot) dz
\right)^{-1}, \quad \theta := \varrho\ \E \int_{Y^-} \det \nabla
\Phi(z,\cdot) dz,
$$
where $\det$ denotes the determinant.  Then, $Q\uepsm$ converges
weakly to $\theta u_0$ in $L^2(\Omega)$ with $\theta < 1$.

\item[{\upshape (iii)}] The function $u_0$ is the unique weak
solution in $H^1_{\mathbb{C}}(\Omega)$ to the homogenized equation
\begin{equation}
\left\{
\begin{aligned}
\nabla \cdot K^* \nabla u_0(x) = 0, &\quad& &x\in \Omega,\\
n(x) \cdot K^*\nabla u_0(x) = g, &\quad& &x \in \partial \Omega,
\end{aligned}
\label{eq:hpde} \right.
\end{equation}
\end{enumerate}
The homogenized admittivity coefficient $K^*$ is given by
\begin{equation}
\forall (i,j) \in \{1,2\}^2, \quad
K^*_{ij} = k_0 \left(\delta_{ij} + \varrho \E\int_{\Phi(Y)} e_j
\cdot (\chi_{Y^+} \nabla w^+_{e_i} + \chi_{Y^-} \nabla w^-_{e_i})
(y,\cdot)\ dy \right), \label{eq:Khomog}
\end{equation}
where $\{e_i\}_{i=1}^2$ is the Euclidean basis of $\R^2$ and for
each $p \in \R^2$, the pair of functions $(w^+_p, w^-_p)$ is the
unique solution to the auxiliary system \eqref{eq:auxiliary}.
\end{thm}

In the dilute limit, we obtain the following approximation of the
effective permittivity for the dilute suspension:
\begin{equation}
K^*_{ij} = k_0(I + \varrho f \E M_{ij}) + o(f),
\label{eq:Kdilute}
\end{equation}
where  $\varrho$ accounts for the averaged change of volume due to
the random diffeomorphism and $f$ is the volume fraction occupied by the cells ; the polarization matrix $M$ is defined
by
\begin{equation}
M_{ij} = \beta k_0 \int_{\rho^{-1} \Phi(\Gamma)} \tilde{\psi}_i
n_j \ ds(\tilde{y}), \label{eq:Mijdef}
\end{equation}
and is associated to the deformed inclusion scaled to the unit
length scale.

\section{Analysis of the problem} \label{sect:analysis}

For a fixed $\eps$, recall that $\Hq^1(\Depsp)$ denotes the Sobolev space
$H^1(\Depsp)/ \C$, which can be represented as
\begin{equation}
\Hq^1(\Depsp) = \left\{ u\in H^1(\Depsp) ~|~ \int_{\Depsp} u(x) dx
= 0\right\}.
\end{equation}
The natural functional space for \eqref{eq:u_{epsilon}} is
\begin{equation}
\Weps := \left\{ u = u^+\chi_\eps^+ + u^-\chi_\eps^- ~|~ u^+ \in
\Hq^1(\Depsp), u^- \in H^1(\Depsm) \right\},
\end{equation}
where $\chi_\eps^{\pm}$ are the characteristic functions of the
sets $\Deps^{\pm}$. We can verify that
\begin{equation}
\|u \|_{\Weps} = \left(\|\nabla u^+ \|_{L^2(\Depsp)}^2 + \|\nabla
u^- \|_{L^2(\Depsm)}^2 + \eps \|u^+ -
u^-\|_{L^2(\interface)}^2\right)^{\frac 1 2}
\end{equation}
defines a norm on $\Weps$. In fact, as it will be seen in
Proposition \ref{prop:equiv}, this norm is equivalent to the
standard norm on $\Weps$ which is
\begin{equation}
\|u\|_{\Hq^1(\Depsp)\times H^1(\Depsm)} = \left(\|\nabla
u^+\|_{L^2(\Depsp)}^2 + \|\nabla u^-\|_{L^2(\Depsm)}^2 +
\|u^-\|_{L^2(\Depsm)}^2 \right)^{\frac 1 2}.
\end{equation}

\subsection{Existence and uniqueness of a solution}

Problem \eqref{eq:u_{epsilon}} should be understood through its
weak formulation as follows: For a fixed $\eps > 0$, find $\ueps
\in \Weps$ such that
\begin{equation}
\begin{aligned}
\int_{\Depsp} k_0 \nabla \uepsp(x) \cdot \nabla \overline{v^+}(x) dx + &\int_{\Depsm} k_0
\nabla \uepsm(x) \cdot \nabla \overline{v^-}(x)ds(x) \\
+  &\frac{1}{\eps \beta} \int_{\interface}
(\uepsp-\uepsm)(x)\overline{(v^+ - v^-)}(x) ds(x) = \int_{\partial
\Omega} g(x) \overline{v^+}(x) ds(x),
\end{aligned}
\label{eq:rpdeweak}
\end{equation}
for any function $v \in \Weps$.

Define the sesquilinear form $a_\eps(\cdot,\cdot)$ on $W_\eps
\times W_\eps$ by
\begin{equation}
\label{eq:aepsdef} a_\eps(u,v) := \int_{\Depsp} k_0 \nabla u^+
\cdot \nabla \overline{v^+} dx + \int_{\Depsm} k_0 \nabla u^-
\cdot \nabla \overline{v^-} dx + \frac{1}{\eps \beta}
\int_{\interface} (u^+ - u^-)\overline{(v^+ - v^-)} ds.
\end{equation}
Associate the following anti-linear form on $W_\eps$ to the
boundary data $g$:
\begin{equation}
\ell(u) := \int_{\partial \Omega} g \overline{u^+} ds.
\end{equation}
The forms $a_\eps$ and $\ell$ are bounded. Moreover, $a_\eps$ is
coercive in the following sense
\begin{equation}
\Re \ k_0^{-1} a_\eps(u,u) = \left(\int_{\Depsp}  |\nabla u^+|^2
dx + \int_{\Depsm}  |\nabla u^-|^2 dx\right) +
\frac{1}{\eps\beta^\prime} \int_{\interface} |u^+ - u^-|^2 ds \ge
C \|u\|_{\Weps}^2,
\end{equation}
where $\beta^\prime := \delta (\sigma_0 \sigma_m + \omega^2
\epsilon_0 \varepsilon_m)/(\sigma_m^2 + \omega^2
\epsilon_m^2)$. Consequently, due to the Lax--Milgram theorem
we have existence and uniqueness for \eqref{eq:u_{epsilon}} for
each fixed $\eps$ and for every $\gamma \in \mathcal{O}$. Note
that $C$ can be chosen independent of $\varepsilon$.

\begin{prop} Let $g \in H^{-1/2}(\partial \Omega)$. There exists a unique $u_\eps \in \Weps$ so that
\begin{equation}
a_\eps(u_\eps, \varphi) = \ell(\varphi), \quad \forall \varphi \in
\Weps. \label{eq:aEqg}
\end{equation}
\end{prop}

To end this subsection we remark that the two norms on $W_\eps$
are equivalent.

\begin{prop}\label{prop:equiv}
The norm $\|\cdot\|_{\Weps}$ is equivalent with the standard norm
on $\Hq^1(\Depsp)\times H^1(\Depsm)$. Moreover, we can find two
positive constants $C_1 < C_2$, independent of $\eps$, so that
\begin{equation}
\label{eq:equiv} C_1 \|u\|_{\Weps} \le \|u\|_{\Hq^1 \times H^1}
\le C_2 \|u\|_{\Weps},
\end{equation}
for any $u \in \Hq^1(\Depsp)\times H^1(\Depsm)$.
\end{prop}

Similar equivalence relation was established by Monsurr\`{o}
\cite{Monsur}, whose method can be adapted easily to the current
case. For the sake of completeness, we present the details in
Appendix \ref{appendixc}.

\subsection{Energy estimate} \label{sec:apriori}

For any fixed $\gamma \in \mathcal{O}$ and a sequence of $\eps \to
0$, by solving \eqref{eq:u_{epsilon}} we obtain the sequence
$\ueps = \uepsp \chi_{\eps}^+ + \uepsm \chi_{\eps}^-$. We obtain
some {\it a priori} estimates for $\ueps$.

We first recall that the extension theorem (Theorem
\ref{thm:Dext}) yields a Poincar\'e--Wirtinger inequality in
$H^1_{\mathbb{C}}(\Omega_{\varepsilon}^+)$ with a constant
independent of $\varepsilon$. Indeed, Corollary \ref{cor:poincare}
shows that for all $v^+ \in
H^1_{\mathbb{C}}(\Omega^+_{\varepsilon})$, there exists a constant
$C$, independent of $\varepsilon$,  such that
\begin{equation*}
 \|v^+ \|_{L^2(\Omega_{\varepsilon}^+)} \leq C \| \nabla v^+ \|_{L^2(\Omega_{\varepsilon}^+)}.
 \end{equation*}

Similarly, we can find a constant, independent of $\varepsilon$,
by applying the trace theorem  in $H^1(\Omega_{\varepsilon}^+)$.
Using Corollary \ref{cor:trace}, the following result holds.

\begin{prop}
\label{prop:apriori} Let $g \in H^{-\frac{1}{2}}(\partial
\Omega)$. For any $\gamma \in \mathcal{O}$, let $\Omega = \Depsp
\cup \interface\cup \Depsm$. Then there exist constants $C$'s,
independent of $\eps$ and $\gamma$, such that the solution $\ueps$
to \eqref{eq:u_{epsilon}} satisfies the following estimates:
\begin{equation}
\label{eq:GradientEst} \|\nabla \uepsp\|_{L^2(\Depsp)} + \|\nabla
\uepsm\|_{L^2(\Depsm)} \le C |k_0|^{-1}
\|g\|_{H^{-\frac{1}{2}}(\partial \Omega)},
\end{equation}
\begin{equation}
\|\uepsp - \uepsm\|_{L^2(\interface)} \le C |k_0|^{-1} \sqrt{\eps
\beta^\prime} \|g\|_{H^{-\frac 1 2}(\partial \Omega)}.
\label{eq:L2GammaEst}
\end{equation}
\end{prop}
\begin{proof} By taking $\varphi = \ueps$ in \eqref{eq:aEqg}, and taking the real part of resultant equality, we get
\begin{equation}
\label{eq:prop:apri1}  \|\nabla \uepsp\|^2_{L^2(\Depsp)} +
\|\nabla \uepsm\|^2_{L^2(\Depsm)} + (\eps \beta^\prime)^{-1}
\|\uepsp - \uepsm\|_{L^2(\interface)}^2 = \Re k_0^{-1} \langle g,
\uepsp\rangle.
\end{equation}
Here $\displaystyle \langle g, \uepsp\rangle = \int_{\partial \Omega}
g\overline{\uepsp} ds$ is the pairing on $H^{-\frac 1 2}(\partial
\Omega) \times H^{\frac 1 2}(\partial \Omega)$, for which we have
the estimate
\begin{equation*}
\lvert \langle g, \uepsp\rangle \rvert \le \|g\|_{H^{-\frac 1 2}(\partial \Omega)}
\|\uepsp\|_{H^{\frac 1 2}(\partial \Omega)} \le C_1\|g\|_{H^{-\frac 1 2}(\partial
\Omega)} \|\uepsp\|_{H^1(\Depsp)}.
\end{equation*}
thanks to the Cauchy - Schwartz inequality and Corollary (\ref{cor:trace}). $C_1$ is here a constant which does not depend on $\varepsilon$.

Applying Proposition (\ref{prop:equiv}) yields
\begin{equation*}
\lvert \langle g, \uepsp\rangle \rvert \le C_2 \|g\|_{H^{-\frac 1 2}(\partial
\Omega)} \|\ueps\|_{W_{\varepsilon}},
\end{equation*} with a constant $C_2$ independent of $\varepsilon$.

Using this in \eqref{eq:prop:apri1} along with the coercivity of $a$ we get 
$$ \|\ueps\|_{W_{\varepsilon}}\le C_3 |k_0|^{-1}\|g\|_{H^{-\frac 1
2}(\partial \Omega)},$$where $C_3$ is still independent of $\varepsilon$.

It follows also that $$\lvert \langle g,
\uepsp\rangle \rvert \le C_2 C_3 |k_0|^{-1} \|g\|_{H^{-\frac 1
2}(\partial \Omega)}.$$
Substitute this estimate into the
right-hand side of \eqref{eq:prop:apri1}, we get the desired
estimates.
\end{proof}

Next, we apply the extension theorem (Theorem \ref{thm:Dext}) to
obtain a bounded sequence in $H^1(\Omega)$ for which we can
extract a converging subsequence.

\begin{prop}\label{prop:extlim}
Suppose that the same conditions of the previous proposition hold.
Let $P^\eps_\gamma: H^1(\Depsp) \to H^1(\Omega)$ be the extension
operator of Theorem \ref{thm:Dext}. Then we have
\begin{equation}
\|P^\eps_\gamma \uepsp\|_{H^1(\Omega)} \le C |k_0|^{-1}
\|g\|_{H^{-\frac 1 2}(\partial \Omega)}, \label{eq:BddExt}
\end{equation}
and
\begin{equation}
\|P^\eps_\gamma \uepsp - \ueps \|_{L^2(\Omega)} \le C \eps
|k_0|^{-1}(1+  \sqrt{\beta^\prime}) \|g\|_{H^{-\frac 1 2}(\partial
\Omega)}.
\end{equation}
\end{prop}
\begin{proof} The first inequality is a direct result of \eqref{eq:lem:Dext}, \eqref{eq:lem:Dext}, \eqref{eq:PW} and \eqref{eq:GradientEst}. For the second inequality, we have
\begin{equation*}\begin{array}{l}
\|P^\eps_\gamma \uepsp - \ueps \|_{L^2(\Omega)} = \|P^\eps_\gamma
\uepsp - \uepsm \|_{L^2(\Depsm)} \le C\sqrt{\eps} \|P^\eps_\gamma
\uepsp - \uepsm \|_{L^2(\interface)} \\ \nm \qquad + C\eps
\|\nabla(P^\eps_\gamma \uepsp - \uepsm)
\|_{L^2(\Depsm)}.\end{array}
\end{equation*}
Here, we have used  estimate \eqref{eq:L2fW}. Now,
$\|P^\eps_\gamma \uepsp - \uepsm \|_{L^2(\interface)} = \|\uepsp -
\uepsm \|_{L^2(\interface)}$ is bounded in \eqref{eq:L2GammaEst}.
The second term is bounded from above by
\begin{equation*}
C\eps \|\nabla P^\eps_\gamma \uepsp\|_{L^2(\Depsm)} +
C\eps\|\nabla \uepsm\|_{L^2(\Depsm)} \le C\eps(\|\nabla
\uepsp\|_{L^2(\Depsp)} + \|\nabla \uepsm\|_{L^2(\Depsm)}),
\end{equation*}
where we have used again \eqref{eq:lem:Dext}. This gives the
desired estimates.
\end{proof}

\begin{rem}\label{rem:subseq} As a consequence of the previous proposition, we get a sequence in $H^1(\Omega)$,
namely $P^\eps_\gamma \uepsp$, which is a good estimate of $\ueps$
in $L^2(\Omega)$ and from which we can extract a subsequence
weakly converging in $H^1(\Omega)$ and strongly in $L^2(\Omega)$.
\end{rem}

\section{Homogenization} \label{sect:homog}

We follow \cite{habibi1,habibi2} to derive a homogenized problem
for the model with two-scale asymptotic expansions and to prove a
rigorous two-scale convergence. In \cite{Monsur}, the
homogenization of an analogue problem is developed and proved with
another method.

\subsection{Two-scale asymptotic expansions}

We assume that the solution $u_{\varepsilon}$ admits the following two-scale asymptotic expansion
$$
\forall x \in \Omega\, \, \, u_{\varepsilon}(x) = u_0(x) +
\varepsilon u_1(x, \frac{x}{\varepsilon}) + o(\varepsilon),
$$
with
\begin{equation*}
y \longmapsto u_1(x,y)\, Y\textrm{-periodic} \, \, \textrm{and} \,
\, u_1(x,y) = \left \{
\begin{array}{l}
\vspace{0.2cm} u_1^+(x,y) \, \, \textrm{in} \, \Omega \times Y^+ ,\\
u_1^-(x,y) \, \, \textrm{in} \, \Omega \times Y^- .
\end{array}
\right .
\end{equation*}
We choose a test function $\varphi_{\varepsilon}$ of the same form
as $u_{\varepsilon}$:
\begin{equation*}
\forall x \in \Omega,\, \, \, \varphi_{\varepsilon}(x) =
\varphi_0(x) + \varepsilon \varphi_1(x, \frac{x}{\varepsilon}),
\end{equation*}
with $\varphi_0$ smooth in $\Omega$, $\varphi_1(x,.)$ $Y$-periodic,
\begin{equation*} 
\varphi_1(x,y) = \left \{
\begin{array}{l}
\vspace{0.2cm} \varphi_1^+(x,y) \, \, \textrm{in} \, \Omega \times Y^+ ,\\
\varphi_1^-(x,y) \, \, \textrm{in} \, \Omega \times Y^-,
\end{array}
\right .
\end{equation*}
and $\vspace{0.2cm}\varphi_1^{\pm} $ smooth in $\Omega \times Y^{\pm}$.

In order to prove items {\upshape (ii)} and {\upshape (iii)} in
Theorem \ref{thm:homo}, we perform an asymptotic expansion of the
variational formulation \eqref{eq:aEqg}. We thus inject these
ansatz in the variational formulation  and only consider the order
$0$ of the different integrals.

 At order $0$, $$\vspace{0.2cm}
\nabla u_{\varepsilon}(x) = \nabla u_0(x) + \nabla_y u_1(x,
\frac{x}{\varepsilon}) + o(\varepsilon).$$ Thanks to Lemma
\ref{lem:lim}, we then have for the two first integrals:
\begin{equation*}
\begin{array}{lr}
\vspace{0.3 cm} \displaystyle \int_{\Omega_{\varepsilon}^+}k_0 \left (\nabla u_0(x) + \nabla_y u_1^+(x, \frac{x}{\varepsilon}) \right ) \cdot \left (\nabla \overline{\varphi}_0(x) + \nabla_y \overline{\varphi}_1^+(x, \frac{x}{\varepsilon}) \right ) dx \\
\vspace{0.3 cm} \hspace{2.6cm} = \displaystyle \int_{\Omega} \int_{Y^+}k_0 \left (\nabla u_0(x) +
\nabla_y u_1^+(x, y) \right ) \cdot \left (\nabla \overline{\varphi}_0(x) + \nabla_y
\overline{\varphi}_1^+(x, y) \right ) dx dy + o(\varepsilon) \\
\end{array}
\end{equation*}
and
\begin{equation*}
\begin{array}{lr}
 \vspace{0.3 cm} \displaystyle\int_{\Omega_{\varepsilon}^-}k_0 \left (\nabla u_0(x) + \nabla_y u_1^-(x, \frac{x}{\varepsilon}) \right ) \cdot \left (\nabla \overline{\varphi}_0(x) + \nabla_y \overline{\varphi}_1^-(x, \frac{x}{\varepsilon}) \right ) dx \\
 \vspace{0.3cm} \hspace{2.6cm} = \displaystyle \int_{\Omega} \int_{Y^-} k_0\left
 (\nabla u_0(x) + \nabla_y u_1^-(x, y) \right ) \cdot \left (\nabla
 \overline{\varphi}_0(x) + \nabla_y \overline{\varphi}_1^-(x, y) \right ) dx dy + o(\varepsilon). \\
 \end{array}
\end{equation*}

We write the third integral in (\ref{eq:aepsdef}) as the sum, over all squares $Y_{\varepsilon,n}$,
of integrals on the boundaries $\Gamma_{\varepsilon,n}$. We have
\begin{equation*}
\begin{array}{l}
\vspace{0.3 cm} \displaystyle\frac{1} {\beta \varepsilon} \int_{\Gamma_{\varepsilon}} \left (
u_{\varepsilon}^+ (x, \frac{x}{\varepsilon}) -
 u_{\varepsilon}^- (x, \frac{x}{\varepsilon}) \right ) \left (
 \overline{\varphi}_{\varepsilon}^{+} (x, \frac{x}{\varepsilon})
  - \overline{\varphi}_{\varepsilon}^- (x, \frac{x}{\varepsilon}) \right )
 ds(x)  \\
\hspace{3cm}\displaystyle = \frac{1} {\beta \varepsilon} \sum_{n
\in N_{\varepsilon}}  \int_{\Gamma_{\varepsilon,n}} \left (
u_{\varepsilon}^+ (x, \frac{x}{\varepsilon}) - u_{\varepsilon}^-
(x, \frac{x}{\varepsilon}) \right ) \left (
\overline{\varphi}_{\varepsilon}^{+} (x, \frac{x}{\varepsilon})
-\overline{\varphi}_{\varepsilon}^- (x, \frac{x}{\varepsilon})
\right )
 ds(x).
 \end{array}
\end{equation*}

Let $x_{0,n}$ be the center of $Y_{\varepsilon, n}$ for each
$n \in N_{\varepsilon}$. We perform Taylor expansions with respect
to the variable $x$ around $x_{0,n}$ for all functions $(u_i)_{i\in\{1,2\}}$ and
$(\varphi_i)_{i\in\{1,2\}}$ on $\vspace{0.1cm}Y_{\varepsilon, n}$. After the
change of variables $\varepsilon(y - y_{0,n}) = x - x_{0, n}$, we
obtain that

\begin{equation*}
\begin{array}{l}
\vspace{0.2cm} u_{\varepsilon} (x) = u_0(x_{0,n}) + \varepsilon
u_1(x, y) + \varepsilon
(y - y_{0,n}) \cdot\nabla u_0(x_{0,n}) + o(\varepsilon), \\
\varphi_{\varepsilon} (x) = \varphi_0(x_{0,n}) + \varepsilon
\varphi_1(x, y) + \varepsilon (y - y_{0,n})\cdot \nabla
\varphi_0(x_{0,n}) + o(\varepsilon).
\end{array}
\end{equation*}

 Consequently, the third integral in the variational formulation \eqref{eq:aEqg} becomes
\begin{equation*}
\displaystyle\frac{\varepsilon^2} {\beta} \sum_{n\in
N_{\varepsilon}}  \int_{\Gamma_{n}} \left ( u_1^+ (x_{0, n}, y) -
u_1^-(x_{0, n}, y) \right ) \left ( \overline{\varphi}_1^+(x_{0,
n}, y)- \overline{\varphi}_1^-(x_{0, n}, y) \right ) ds(y).
\end{equation*}

Finally, Lemma \ref{lem:lim} gives us that
\begin{equation*}
\begin{array}{l}
\vspace{0.3cm}  \displaystyle \frac{1} {\varepsilon \beta} \int_{\Gamma_{\varepsilon}}
\left ( u_{\varepsilon}^+ - u_{\varepsilon}^- \right ) \left (  \overline{\varphi}_{\varepsilon}^+
- \overline{\varphi}_{\varepsilon}^-\right )
ds \\
\hspace{3cm} = \displaystyle\frac{1} {\beta} \int_{\Omega}
\int_{\Gamma} \left ( u_1^+ (x, y) - u_1^- (x, y) \right ) \left (
\overline{\varphi}_1^{+} (x, y) - \overline{\varphi}_1^- (x, y)
\right ) dx ds(y) + o(\varepsilon).
\end{array}
\end{equation*}

Moreover, we can easily see that
\begin{equation*}
  \displaystyle  \int_{\partial \Omega} g \overline{\varphi}_{\varepsilon}^+ ds =
  \displaystyle  \int_{\partial \Omega} g \overline{\varphi}_0
  ds + o(\varepsilon).
\end{equation*}

The order $0$ of the variational formula is thus given by
\begin{equation*}
\begin{array}{ll}
\vspace{0.3 cm} &\displaystyle \int_{\Omega} \int_{Y^+} k_0 \left (\nabla u_0(x) + \nabla_y u_1^+(x, y) \right ) \cdot \left (\nabla \overline{\varphi}_0(x) + \nabla_y \overline{\varphi}_1^+(x, y) \right ) dx dy \\
\vspace{0.3 cm} +&\displaystyle\int_{\Omega} \int_{Y^-}k_0  \left (\nabla u_0(x) + \nabla_y u_1^-(x, y) \right ) \cdot \left (\nabla \overline{\varphi}_0(x) + \nabla_y \overline{\varphi}_1^-(x, y) \right ) dx dy  \\
\vspace{0.3 cm} +&\displaystyle\frac{1} {\beta} \int_{\Omega} \int_{\Gamma} \left ( u_1^+ (x, y) - u_1^- (x, y) \right ) \left ( \overline{\varphi}_1^{+} (x, y) - \overline{\varphi}_1^- (x, y) \right ) dx
ds(y) \\
- &\displaystyle  \int_{\partial \Omega}  g(x)
\overline{\varphi}_0(x) ds(x) = 0.
\end{array}
\end{equation*}

By taking $\varphi_0 =0$, it follows that
\begin{equation*}
\begin{array}{ll}
\vspace{0.3 cm} &\displaystyle \int_{\Omega} \int_{Y^+} k_0\left (\nabla u_0(x) + \nabla_y u_1^+(x, y) \right ) \cdot  \nabla_y \overline{\varphi}_1^+(x, y) dx dy \\
\vspace{0.3 cm} +&\displaystyle\int_{\Omega} \int_{Y^-} k_0\left (\nabla u_0(x) + \nabla_y u_1^-(x, y) \right ) \cdot \nabla_y\overline{\varphi}_1^-(x, y) dx dy  \\
+&\displaystyle\frac{1} {\beta} \int_{\Omega} \int_{\Gamma} \left
( u_1^+ (x, y) - u_1^- (x, y) \right ) \left (
\overline{\varphi}_1^{+} (x, y) - \overline{\varphi}_1^- (x, y)
\right ) dx ds(y) =0,
\end{array}
\end{equation*}
which is exactly the variational formulation of the cell problem
(\ref{eq:w_i}) and  definition (\ref{eq:u_1}) of $u_1$.

By taking $\varphi_1 = 0$, we recover the variational formulation
of the homogenized problem (\ref{eq:u_0}):
\begin{equation*}
\begin{array}{ll}
\vspace{0.3 cm} &\displaystyle \int_{\Omega} \int_{Y^+}k_0 \left (\nabla u_0(x) + \nabla_y u_1^+(x, y) \right ) \cdot  \nabla \overline{\varphi}_0(x) dx dy  \\
\vspace{0.3 cm}+& \displaystyle\int_{\Omega} \int_{Y^-}k_0 \left (\nabla u_0(x) + \nabla_y u_1^-(x, y) \right ) \cdot \nabla\overline{\varphi}_0(x) dx dy \\
-& \displaystyle \int_{\partial \Omega}  g(x)
\overline{\varphi}_0(x) ds(x) = 0.
\end{array}
\end{equation*}

We introduce some function spaces, which will be very useful in
the following:
\begin{itemize}
\item $C^{\infty}_{\sharp} (D)$ is the space of functions, which
are $Y$ - periodic and in $C^{\infty}(D)$, \item $L^2_{\sharp}(D)$
is the completion of $C^{\infty}_{\sharp}(D)$ in the $L^2$-norm,
\item $H^1_{\sharp}(D)$ is the completion of
$C^{\infty}_{\sharp}(D)$ in the $H^1$-norm, \item $L^2(\Omega,
H^1_{\sharp}(D))$ is the space of square integrable functions on
$\Omega$ with values in the space $H^1_{\sharp}(D)$, \item
$\mathcal{D}(\Omega)$ is the space of infinitely smooth functions
with compact support in $\Omega$, \item $\mathcal{D}(\Omega,
C^{\infty}_{\sharp} (D))$ is the space of infinitely smooth
functions with compact support in $\Omega$ and with values in the
space $C^{\infty}_{\sharp}$,
\end{itemize}
where $D$ is $Y, Y^{+}, Y^-$ or $\Gamma$.

The following lemma was used in the preceding proof.  It follows
from \cite[Lemma 3.1]{habibi1}.
\begin{lem} \label{lem:lim}Let $f$ be a smooth function. We have
\begin{equation*}
\begin{array}{ll}
\vspace{0.3 cm} (i)& \displaystyle\varepsilon^2 \sum_{n\in
N_{\varepsilon}} \int_{\Gamma_{\varepsilon, n}}f(x_{0, n},y)ds(y)
 = \int_{\Omega}\int_{\Gamma} f(x,y) dxds(y) + o(\varepsilon);\\
\vspace{0.3cm}(ii)&\displaystyle \int_{\Omega_{\varepsilon}^+} f(x, \frac{x}{\varepsilon}) \,dx =
 \int_{\Omega} \int_{Y^+} f(x, y) \,dx dy + o(\varepsilon)\\
&\displaystyle \mbox{and} \quad \int_{\Omega_{\varepsilon}^-} f(x,
\frac{x}{\varepsilon})\, dx = \int_{\Omega} \int_{Y^-} f(x, y)
\,dx dy + o(\varepsilon).
\end{array}
\end{equation*}
\end{lem}

We prove that the following lemmas hold.

\begin{lem} The homogenized problem admits a unique solution in
$H^1_\mathbb{C}(\Omega)$. \end{lem}
\begin{proof}
 The effective admittivity can be rewritten as $$ \begin{array}{lll} \displaystyle K^* _{i, j} &=&\displaystyle k_0  \int_{Y^+} (\nabla w_i^+
+ e_i) \cdot (\nabla \overline{w_j^+} + e_j) dx + k_0 \int_{Y^-}
(\nabla w_i^- + e_i) \cdot (\nabla \overline{w_j^-} + e_j) dx \\
\nm && \displaystyle + \frac{1}{\beta} \int_\Gamma (w_i^+ - w_i^-)
(\overline{w_j^+} - \overline{w_j^-}) ds, \quad i,j=1,2.
\end{array}$$ Therefore, it follows that, for $\xi =
\left( \xi_1, \xi_2 \right) \in \mathbb{R}^2$,
$$
K^* \xi \cdot \xi = k_0  \int_{Y^+} |\nabla w^+ + \xi|^2 dx + k_0
\int_{Y^-} |\nabla w^- + \xi|^2 dx + \frac{1}{\beta} \int_\Gamma
|w^+ - w^-|^2 ds, $$ where $w = \sum_i \xi_i w_i$. Since $\Re e
\beta \geq 0$,
$$
K^* \xi \cdot \xi \geq k_0  \int_{Y^+} |\nabla w^+ + \xi|^2 dx +
k_0 \int_{Y^-} |\nabla w^- + \xi|^2 dx.
$$
Consequently, it follows from \cite{allairebook} that there exist
two positive constants $C_1$ and $C_2$ such that
$$
C_1 |\xi|^2 \leq \Re e K^* \xi \cdot \xi \leq C_2 |\xi|^2.$$
Standard elliptic theory yields existence and uniqueness of a
solution to the homogenized problem in $H^1_\mathbb{C}(\Omega)$.
\end{proof}

\begin{lem}
The cell problem (\ref{eq:w_i}) admits a unique solution in $H^1_{\sharp}(Y^+)/\mathbb{C} \times H^1_{\sharp}(Y^-)$.
\end{lem}

\begin{proof}
Let us introduce the Hilbert space
\begin{equation*}
W:=\left \{v:=v^+ \chi_{Y^+} + v^- \chi_{Y^-} | (v^+, v^-) \in
H^1_{\mathbb{C}}(Y^+) \times H^1(Y^-) \right \},
\end{equation*}
equipped with the norm
\begin{equation*}
\| v\|_W^2 = \|\nabla v^+\|_{L^2(Y^+)}^2 +  \| \nabla v^-\|_{L^2(Y^-)}^2 +\| v^+ - v^-\|_{L^2(\Gamma)}^2.
\end{equation*}
We consider the following problem:
\begin{equation}\label{eq:vfw_i}
\left \{
\begin{array}{l}
\vspace{0.3 cm} \textrm{Find} \, w_i \in  W_{\sharp} \, \textrm{such that for all} \,\varphi \in W_{\sharp}\\
\vspace{0.3 cm} \displaystyle \int_{Y^+} k_0 \nabla w_i^+(y) \cdot \nabla \overline{\varphi} ^+(y) dy+ \int_{Y^-}k_0  \nabla w_i^-(y) \cdot \nabla \overline{\varphi}^-(y) dy  \\
\vspace{0.3 cm} \hspace{1cm}\displaystyle + \frac{1} {\beta}
\int_{\Gamma} \left ( w_i^+ - w_i^- \right )(y) \left (
\overline{\varphi}^+ - \overline{\varphi}^-\right )(y)
ds(y)=\\
 \hspace{1.5cm} - \displaystyle \int_{Y^+} k_0 \nabla y_i \cdot \nabla \overline{\varphi} ^+(y) dy - \int_{Y^-}k_0  \nabla y_i \cdot \nabla \overline{\varphi}^-(y) dy.  \\
\end{array}
\right .
\end{equation}

Lax--Milgram theorem  gives us existence and uniqueness of a
solution. Moreover, one can show that this ensures the existence
of a unique solution in $H^1_{\sharp}(Y^+)/\mathbb{C} \times H^1_{\sharp}(Y^-)$ for the cell
problem (\ref{eq:w_i}).
\end{proof}

\subsection{Convergence}

We present in this section a rigorous proof of the convergence of
the initial problem to the homogenized one. We use for this
purpose the two-scale convergence technique and hence need first
of all some bounds on $u_{\varepsilon}$ to ensure the convergence.

\subsubsection{A priori estimates}

\begin{thm}\label{thm:bound+}
The function $u_{\varepsilon}^+$ is uniformly bounded with respect
to $\varepsilon$ in $H^1(\Omega_{\varepsilon}^+)$, {\it i.e.},
there exists a constant $C$, independent of $\varepsilon$, such
that
\begin{equation*}
\| u_{\varepsilon}^+ \|_{H^1(\Omega_{\varepsilon}^+)} \le C.
\end{equation*}
\end{thm}

\begin{proof}
Combining (\ref{eq:GradientEst}) and Poincaré - Wirtinger inequality, we obtain immediately the wanted result.
\end{proof}

The proof of the following result follows the one of Lemma 2.8 in
\cite{Monsur}.

\begin{lem}\label{lem:normineq}There exists a constant $C$, which does not depend on $\varepsilon$, such that
 for all $v \in \Weps$:
\begin{equation*}
\|v^-\|_{L^2(\Omega^-_{\varepsilon})} \leq C \|v\|_{\Weps}.
\end{equation*}
\end{lem}

\begin{proof}
We write the norm $\|v^-\|_{L^2(\Omega_{\varepsilon}^-)}$ as a sum over all the cells.
\begin{equation*}
\|v^-\|_{L^2(\Omega_{\varepsilon}^-)}^2 = \sum_{n \in
N_{\varepsilon}} \|v^-\|_{L^2(Y_{\varepsilon,n}^-)}^2 =  \sum_{n
\in N_{\varepsilon}} \int_{Y_{\varepsilon,n}^-} |v^-(x)|^2 dx.
\end{equation*}

We perform the change of variable $y=\displaystyle
\frac{x}{\varepsilon}$ and get
\begin{equation}\label{eq:decomposition}
\|v^-\|_{L^2(\Omega_{\varepsilon}^-)}^2 =  \varepsilon^2 \sum_{n
\in N_{\varepsilon}} \int_{Y_{n}^-} |v_{\varepsilon}^-(y)|^2 dy,
\end{equation}
where $v_{\varepsilon}^- (y) := v^-(\varepsilon y)$ for all $y \in
Y^-$ and $Y_n^- = n + Y^-$ with $n \in N_{\varepsilon}$.

Recall that $W$ denotes the following Hilbert space:
\begin{equation*}
W:=\left \{v:=v^+ \chi_{Y^+} + v^- \chi_{Y^-}| (v^+, v^-) \in
H^1_{\mathbb{C}}(Y^+) \times H^1(Y^-) \right \},
\end{equation*}
equipped with the norm:
\begin{equation*}
\| v\|_W^2 = \|\nabla v^+\|_{L^2(Y^+)}^2 +  \| \nabla v^-\|_{L^2(Y^-)}^2 +\| v^+ - v^-\|_{L^2(\Gamma)}^2.
\end{equation*}

We first prove that there exists a constant $C_1$, independent of $\varepsilon$, such that for every $v \in W$:
\begin{equation} \label{eq:ineq1x}
\|v^-\|_{L^2(Y^-)} \leq C_1 \|v\|_{W}.
\end{equation}

We proceed by contradiction. Suppose that for any $n \in
\mathbb{N^*}$, there exists $v_n \in \Weps$ such that
\begin{equation*}
\|v_n^-\|_{L^2(Y^-)} =1\hspace{0.5cm} \textrm{and} \hspace{0.5cm} \|v_n\|_{W} \leq \displaystyle \frac{1}{n}.
\end{equation*}

Since $\|v_n^-\|_{L^2(Y^-)} =1$ and $\| \nabla v_n^-\|_{L^2(Y^-)} \leq \| v_n\|_W \leq \displaystyle \frac{1}{n}$, $v_n^-$ is bounded in $H^1(Y^-)$. Therefore it converges weakly in $H^1(Y^-)$. By compactness, we can extract a subsequence, still denoted $v_n^-$, such that $v_n^-$ converges strongly in $L^2(Y^-)$.
We denote by $\vspace{0.1cm} l$ its limit.

Besides, $\nabla v_n^-$ converges strongly to $0$ in $L^2(Y^-)$. We thus have $\nabla l =0$ and $l$ constant in $Y^-$.

By applying  in $Y^+$ the trace theorem and Poincar\'e--Wirtinger
inequality to $v_n^+$, one also gets  that
\begin{equation*}
\|v_n^-\|_{L^2(\Gamma)} \leq \|v_n^+ - v_n^-\|_{L^2(\Gamma)} + \|v_n^+\|_{L^2(\Gamma)} \leq \|v_n^+ - v_n^-\|_{L^2(\Gamma)} + C  \|v_n^+\|_{H^1(Y^+)} \leq \displaystyle \frac{C'}{n}.
\end{equation*}

Consequently, $v_n^-$ converges strongly to $0$ in $L^2(\Gamma)$ and $l=0$ on $\Gamma$.

We have then $l=0$ in $Y^-$, which leads to a contradiction. This
proves (\ref{eq:ineq1x}).

We can now find an upper bound to (\ref{eq:decomposition}):
\begin{equation*}
\|v^-\|_{L^2(\Omega_{\varepsilon}^-)}^2 \leq \varepsilon^2 C_1
\sum_{n \in N_{\varepsilon}} \int_{Y_{n}^+} |\nabla
v_{\varepsilon}^+ (y)|^2 dy + \int_{Y_{n}^-} |\nabla
v_{\varepsilon}^- (y)|^2 dy + \int_{\Gamma_{n}} |
v_{\varepsilon}^+ (y) - v_{\varepsilon}^- (y)|^2 ds(y).
\end{equation*}

After the change of variable $x=\varepsilon y$, one gets
\begin{equation*}
\|v^-\|_{L^2(\Omega_{\varepsilon}^-)}^2 \leq \varepsilon\,
C_1 \left(\|\nabla v^+ \|_{L^2(\Omega_{\varepsilon}^+)}^2 +
\|\nabla v^-\|_{L^2(\Omega_{\varepsilon}^-)}^2 + \varepsilon\| v^+-
v^-\|^2_{L^2(\Gamma_{\varepsilon})}  \right).
\end{equation*}

Since $\varepsilon<1$, there exists a constant $C_2$, which does
not depend on $\varepsilon$ such that for every $v \in \Weps$,
\begin{equation*}
\|v^-\|_{L^2(\Omega^-_{\varepsilon})} \leq C_2 \|v\|_{\Weps},
\end{equation*}
which completes the proof.
 \end{proof}

\begin{thm}
$u_{\varepsilon}^-$ is uniformly bounded in $\varepsilon$ in
$H^1(\Omega_{\varepsilon}^-)$, {\it i.e.}, there exists a constant
$C$ independent of $\varepsilon$, such that
\begin{equation*}
\| u_{\varepsilon}^- \|_{H^1(\Omega_{\varepsilon}^-)} \le C.
\end{equation*}
\end{thm}

\begin{proof}
By definition of the norm on $\Weps$ , $\| \nabla
u_{\varepsilon}^-\|^2_{L^2(\Omega_{\varepsilon}^-)} \le \|
u_{\varepsilon} \|^2_{\Weps} $.

We thus have with the result of Lemma \ref{lem:normineq}:
\begin{equation}\label{eq:part1}
\| u_{\varepsilon}^-\|^2_{H^1(\Omega_{\varepsilon}^-)} \le C_1 \|
u_{\varepsilon} \|^2_{\Weps} ,
\end{equation}
with a constant $C_1$ which does not depend on $\varepsilon$.

Furthermore, by combining (\ref{eq:ineq2}) and the result of
Theorem \ref{thm:bound+}, there exists a constant $C_2$
independent of $\varepsilon$ such that
 \begin{equation*}
 |a(u_{\varepsilon}, u_{\varepsilon})| \leq C_2.
 \end{equation*}
 We use the coercivity of $a$ and get a uniform bound in $\varepsilon$ of $u_{\varepsilon}$
  in $\vspace{0.2cm}\Weps$. This bound and (\ref{eq:part1}) complete the proof.
\end{proof}

\subsubsection{Two-scale convergence}

We first recall the definition of two-scale convergence and a few
results of this theory \cite{allaire}.
\begin{defi}
A sequence of functions $u_{\varepsilon}$ in $L^2(\Omega)$ is said
to \emph{two-scale converge} to a limit $u_0$ belonging to
$L^2(\Omega \times Y)$ if,
 for any function $\psi$ in $L^2(\Omega, C_{\sharp}(Y))$, we have
\begin{equation*}
\displaystyle \lim_{\varepsilon \to 0} \int_{\Omega}
u_{\varepsilon}(x) \psi(x, \frac{x}{\varepsilon}) dx =
\int_{\Omega} \int_{Y} u_0(x, y) \psi(x, y) dx dy.
\end{equation*}
\end{defi}

This notion of two-scale convergence makes sense because of the
next compactness theorem.

\begin{thm}
From each bounded sequence $u_{\varepsilon}$ in $L^2(\Omega)$, we
can extract a subsequence, and there exists a limit $u_0 \in
L^2(\Omega \times Y)$ such that this subsequence two-scale
converges to $u_0$.
\end{thm}

Two-scale convergence can be extended to sequences defined on
periodic surfaces.

\begin{prop}
For any sequence $u_{\varepsilon}$ in $L^2(\Gamma_{\varepsilon})$
such that
\begin{equation}\label{eq:condition}
\displaystyle \varepsilon \int_{\Gamma_{\varepsilon}} |
u_{\varepsilon}|^2 dx \leq C,
\end{equation}
there exists a subsequence, still denoted $u_{\varepsilon}$, and a
limit function $u_0 \in L^2(\Omega, L^2(\Gamma))$ such that
$u_{\varepsilon}$ two-scale converges to $u_0$ in the sense
\begin{equation*}
\displaystyle \lim_{\varepsilon \to 0}\varepsilon
\int_{\Gamma_{\varepsilon}} u_{\varepsilon}(x) \psi(x,
\frac{x}{\varepsilon}) ds(x) = \int_{\Omega}\int_{\Gamma}
u_0(x, y) \psi(x, y) dx ds(y),
\end{equation*}
for any function $\psi \in L^2(\Omega, C_{\sharp}(Y))$.
\end{prop}

\begin{rem} \label{remhabibi}
If $u_{\varepsilon}$ and $\nabla u_{\varepsilon}$ are bounded in
$L^2(\Omega)$, one can prove by using for example \cite[Lemma
2.4.9]{ganaoui} that $u_{\varepsilon}$ verifies the uniform bound
(\ref{eq:condition}). The two-scale limit on the surface is then
the trace on $\Gamma$ of the two-scale limit of $u_{\varepsilon}$
in $\Omega$.
\end{rem}

In order to prove item {\upshape (i)} in Theorem \ref{thm:homo},
we need the following results.

\begin{lem}\label{lem:tsc}
Let the functions $(u_{\varepsilon})_\varepsilon$ be the sequence
of solutions of (\ref{eq:u_{epsilon}}). There exist functions $u(x) \in H^1(\Omega)$, $v^+(x,y) \in
L^2(\Omega, H^1_{\sharp}(Y^+))$ and $v^-(x,y) \in L^2(\Omega,
H^1_{\sharp}(Y^-))$ such that, up to a subsequence,
\begin{equation*}
\begin{array}{lll}
\left (
\begin{array}{c}
\vspace{0.2cm} u_{\varepsilon}\\
\vspace{0.2cm}\displaystyle\chi_{\eps}^+(\frac{x}{\varepsilon})\nabla u_{\varepsilon}^+\\
\displaystyle\chi_{\eps}^-(\frac{x}{\varepsilon})\nabla
u_{\varepsilon}^-
\end{array}
\right)&
\textrm{\hspace{0.2cm}two-scale converge to\hspace{0.2cm}}&
\left (
\begin{array}{c}
\vspace{0.3cm} u(x)\\
\vspace{0.3cm}\chi_{Y^+}(y)\Big( \nabla u(x) + \nabla_y v^+(x,y)\Big)\\
\chi_{Y^-}(y)\Big( \nabla u(x) + \nabla_y v^-(x,y)\Big)
\end{array}
\right ).
\end{array}
\end{equation*}
\end{lem}

\begin{proof}We denote by $\tilde{\cdot}$ the extension by zero of functions on $\Omega_{\varepsilon}^+$
and $\Omega_{\varepsilon}^-$ in the respective domains $\Omega_{\varepsilon}^-$ and $\Omega_{\varepsilon}^+$.

From the previous estimates, $\widetilde{u}_{\varepsilon}^{\pm}$
and $\widetilde{\nabla u}^{\pm}_{\varepsilon}$ are bounded
sequences  in $L^2(\Omega)$. Up to a subsequence, they two-scale
converge to $\tau^{\pm}(x, y)$ and $\xi^{\pm}(x, y)$. Since
$\widetilde{u}_{\varepsilon}^{\pm}$ and $\widetilde{\nabla
u}^{\pm}_{\varepsilon}$ vanish in $\Omega_{\varepsilon}^{\mp}$, so
do $\tau^{\pm}$ and $\xi^{\pm}$.

Consider $\varphi \in \mathcal{D}(\Omega,
C_{\sharp}^{\infty}(Y))^2$ such that $\varphi =0$ for $y \in
\overline{Y^-}$. By integrating by parts, it follows that
\begin{equation*}
\displaystyle \varepsilon \int_{\Omega_{\varepsilon}^+} \nabla u_{\varepsilon}^+(x)\cdot \overline{\varphi}(x, \frac{x}{\varepsilon})dx  = - \int_{\Omega_{\varepsilon}^+} u_{\varepsilon}^+(x) \left(\textrm{div}_y \, \overline{\varphi}(x, \frac{x}{\varepsilon}) + \varepsilon \,\textrm{div}_x \, \overline{\varphi}(x, \frac{x}{\varepsilon})\right)dx.
\end{equation*}
We take the limit of this equality as $\varepsilon \rightarrow 0$:
\begin{equation*}
 \displaystyle 0  = - \int_{\Omega} \int_{Y^+} \tau^+(x,y)\textrm{div}_y \, \overline{\varphi}(x, y)dx dy.
\end{equation*}

Therefore, $\tau^+$ does  not depend on $y$ in $Y^+$, {\it i.e.},
there exists a function $u^+ \in L^2(\Omega)$ such that $\tau^+(x,
y) = \chi_{Y^+}(y) u^+(x)$ for all $(x, y) \in \Omega \times Y$.

Take now $\varphi \in  \mathcal{D}(\Omega,
C_{\sharp}^{\infty}(Y))^2$ such that $\varphi =0$ for $y \in
\overline{Y^-}$ and $\textrm{div}_y \varphi =0$. Similarly, we
have
\begin{equation*}
\displaystyle\int_{\Omega_{\varepsilon}^+} \nabla u_{\varepsilon}^+(x) \cdot
\overline{\varphi}(x, \frac{x}{\varepsilon})dx  = - \int_{\Omega_{\varepsilon}^+}
u_{\varepsilon}^+(x) \textrm{div}_x \, \overline{\varphi}(x, \frac{x}{\varepsilon})dx,
\end{equation*}
and thus
\begin{equation} \label{partint}
 \displaystyle \int_{\Omega} \int_{Y^+} \xi^+(x,y)\cdot \overline{\varphi}(x, y)dx dy  = - \int_{\Omega} \int_{Y^+} u^+(x)\textrm{div}_x \, \overline{\varphi}(x, y)dx dy.
\end{equation}

For $\varphi$ independent of $y$, this implies that $u^+ \in
H^1(\Omega)$. Furthermore, if we integrate by parts the right-hand
side of (\ref{partint}), we get
 \begin{equation*}
 \displaystyle \int_{\Omega} \int_{Y^+} \xi^+(x,y)\cdot \overline{\varphi}(x, y)dx dy  = \int_{\Omega} \int_{Y^+} \nabla u^+(x) \cdot\overline{\varphi}(x, y)dx dy,
\end{equation*}
for all $\varphi \in  \mathcal{D}(\Omega, C_{\sharp}^{\infty}(Y^+))^2$ such that $\textrm{div}_y
\varphi =0$ and $\varphi(x, y)\cdot n(y) = 0$ for $y$ on $\Gamma$.

Since the orthogonal of the divergence-free functions are exactly
the gradients, there exists a function $v^+ \in L^2(\Omega,
H^1_{\sharp}(Y^+))$ such that
\begin{equation*}
\xi^+(x, y) =\chi_{Y^+}(y) \left ( \nabla u^+(x) + \nabla_y v^+(x,
y) \right ),
\end{equation*}
for all $(x, y) \in \Omega \times Y$.

Likewise, there exist functions $u^- \in H^1(\Omega)$ and $v^- \in
L^2(\Omega, H^1_{\sharp}(Y^-))$ such that
\begin{equation*}
\tau^-(x, y) = \chi_{Y^-}(y) u^-(x), \hspace{0.2cm} \textrm{and}
\hspace{0.2cm} \xi^-(x, y) = \chi_{Y^-}(y) \left (\nabla u^-(x) +
\nabla_y v^-(x, y) \right ),
\end{equation*}
for all $(x, y) \in \Omega \times Y$.

Furthermore, thanks to Remark \ref{remhabibi}, we have also
\begin{equation*}
\displaystyle \varepsilon \int_{\Gamma_{\varepsilon}} u_{\varepsilon}^{\pm}(x) \overline{\varphi}(x, \frac{x}{\varepsilon}) dx \xrightarrow[\varepsilon \to 0] {}\int_{\Omega}\int_{\Gamma} u^{\pm}(x, y) \overline{\varphi}(x, y) dxdy,
\end{equation*}
for all $\varphi \in L^2(\Omega, C^{\infty}_{\sharp}(\Gamma))$.

Recall that $u_{\varepsilon}$ is a solution to the following
variational form:
\begin{equation*}
\begin{array}{l}
\vspace{0.3 cm} \displaystyle \int_{\Omega_{\varepsilon}^+} k_0 \nabla u_{\varepsilon}^+(x) \cdot \nabla \overline{\varphi}_{\varepsilon} ^+(x) dx+ \int_{\Omega_{\varepsilon}^-}k_0  \nabla u_{\varepsilon} ^-(x) \cdot \nabla \overline{\varphi}_{\varepsilon}^-(x) dx  \\
\hspace{4.5cm} \displaystyle + \frac{1} {\varepsilon \beta}
\int_{\Gamma_{\varepsilon}} \left ( u_{\varepsilon}^+ -
u_{\varepsilon}^- \right ) \left (
\overline{\varphi}_{\varepsilon}^+ -
\overline{\varphi}_{\varepsilon}^-\right ) ds - k_0 \int_{\partial
\Omega} g \overline{\varphi}_{\varepsilon}^+ ds = 0,
\end{array}
\end{equation*}
\vspace{0.3cm}
for all $(\varphi_{\varepsilon}^+, \varphi_{\varepsilon}^-) \in (H^1(\Omega^+_{\varepsilon}),
H^1(\Omega^-_{\varepsilon}))$.

We multiply this equality by $\varepsilon^2$ and take the limit when $\varepsilon$ goes to $0$.
The first two terms disappear and we obtain, for all $(\varphi^+, \varphi^-)
\in \mathcal{D}(\Omega, C^{\infty}_{\sharp}(Y^+))\times \mathcal{D}(\Omega, C^{\infty}_{\sharp}(Y^-))$:
\begin{equation*}
\displaystyle\frac{1}{\beta} \int_{\Omega}\int_{\Gamma} (u^+(x)
 - u^-(x))(\overline{\varphi}^+(x,y) - \overline{\varphi}^-(x, y))dx dy
 =0.
\end{equation*}

Thus $u^+(x) =u^-(x)$ for all $x \in \Omega$, and
$u_{\varepsilon}$ two-scale converges to $u = u^+=u^- \in
H^1(\Omega)$. This completes the proof.
\end{proof}

Now, we are ready to prove Theorem \ref{thm:homo}. For this, we need
 to show that $u$, $v^{+}$ and $v^-$ are respectively $u_0$,
solution of the homogenized problem (up to a constant),  $u_1^{+}$
defined in (\ref{eq:u_1}) (up to a constant) and $u_1^{-}$ defined
in (\ref{eq:u_1}). The uniqueness of a solution for the
homogenized problem and the cell problems will then allow us to
conclude the convergence, not only up to a subsequence.

\begin{proof} We first want to retrieve the expression of $u_1$ as a test function of
the derivatives of $u_0$ and the cell problem solutions $w_i$.

We choose in the variational formulation (\ref{eq:rpdeweak}) a
function $\varphi_{\varepsilon}$ of the form
$$\varphi_{\varepsilon}(x) = \varepsilon\varphi_1(x, \frac{x}{\varepsilon}),$$
where $\vspace{0.2cm}\varphi_1 \in \mathcal{D}(\Omega,
C^{\infty}_{\sharp}(Y^+))\times \mathcal{D}(\Omega,
C^{\infty}_{\sharp}(Y^-))$.

Lemma \ref{lem:tsc} shows the two-scale convergence of the
following three terms:
\begin{equation*}\fontsize{11pt}{7.2}
\begin{array}{clc}
\vspace{0.3cm} \displaystyle \int_{\Omega_{\varepsilon}^+} k_0 \nabla u_{\varepsilon}^+(x) \cdot \nabla \overline{\varphi}_{\varepsilon}^+(x) dx &\xrightarrow[\varepsilon \to 0]{}&\displaystyle \int_{\Omega}\! \int_{Y^+}\! k_0 \left (\nabla u(x) + \nabla_y v^+(x, y) \right)\cdot\nabla_y \overline{\varphi}_1^+(x, y) dxdy \\
\vspace{0.3cm} \displaystyle  \int_{\Omega_{\varepsilon}^-} k_0 \nabla u_{\varepsilon}^-(x) \cdot \nabla \overline{\varphi}_{\varepsilon}^-(x) dx &\xrightarrow[\varepsilon \to 0]{}&\displaystyle  \int_{\Omega}\! \int_{Y^-} \! k_0 \left (\nabla u(x) + \nabla_y v^-(x, y) \right)\cdot  \nabla_y \overline{\varphi}_1^-(x, y)dxdy\\
\displaystyle\int_{\partial \Omega} g(x)
\overline{\varphi}_{\varepsilon}^+(x)
ds(x)&\xrightarrow[\varepsilon \to 0]{}& 0.
\end{array}
\end{equation*}

We can not take directly the limit as $\varepsilon \to 0$ in the
last term:
\begin{equation*}
\begin{array}{l}
\vspace{0.3cm}\displaystyle \frac{1}{\varepsilon \beta}\int_{\Gamma_{\varepsilon}} (u_{\varepsilon}^+(x) -u_{\varepsilon}^-(x))(\overline{\varphi}_{\varepsilon}^+(x) - \overline{\phi}_{\varepsilon}^-(x))ds(x) \\
\hspace{4cm}\displaystyle=
\frac{1}{\beta}\int_{\Gamma_{\varepsilon}} (u_{\varepsilon}^+(x)
-u_{\varepsilon}^-(x))\left( \overline{\varphi}_1^+(x,
\frac{x}{\varepsilon}) - \overline{\varphi}_1^-(x,
\frac{x}{\varepsilon})\right )ds(x).
\end{array}
\end{equation*}

Lemma \ref{lem:theta} ensures the existence of a function $\theta
\in (\mathcal{D}(\Omega, H^1_{\sharp}(Y^+)) \times
\mathcal{D}(\Omega, H^1_{\sharp}(Y^-)))^2$ such that for all $\psi
\in H^1_{\sharp}(Y^+)/\mathbb{C} \times H^1_{\sharp}(Y^-)$ :
\begin{equation}\label{eq:vftheta}
\begin{array}{l}
\vspace{0.3cm}\displaystyle \int_{Y^+} \nabla \psi^+(y) \cdot \overline{\theta}^+(x, y) dy + \int_{Y^-} \nabla \psi^-(y) \cdot \overline{\theta}^-(x, y) dy \\
\hspace{3.5cm}\displaystyle+ \int_{\Gamma} \left(\psi^+(y)
-\psi^-(y)\right)\left( \overline{\varphi}_1^+(x, y) -
\overline{\varphi}_1^-(x, y)\right)ds(y)=0.
\end{array}
\end{equation}

We make the change of variables $y = \displaystyle
\frac{x}{\varepsilon}$, sum over all $(Y_{\varepsilon, n})_{n \in
N_{\varepsilon}}$, and choose $\psi = u_{\varepsilon}$ to get
\begin{equation*}
\begin{array}{l}
\vspace{0.3cm}\displaystyle\int_{\Gamma_{\varepsilon}} (u_{\varepsilon}^+(x)
-u_{\varepsilon}^-(x))\left( \overline{\varphi}_1^+(x, \frac{x}{\varepsilon})
- \overline{\varphi}_1^-(x, \frac{x}{\varepsilon})\right )ds(x) = \\
\hspace{4.2cm}\displaystyle - \int_{\Omega_{\varepsilon}^+} \nabla
u_{\varepsilon}^+(x, \frac{x}{\varepsilon}) \cdot \theta^+(x,
\frac{x}{\varepsilon}) dx -
\displaystyle\int_{\Omega_{\varepsilon}^-} \nabla
u_{\varepsilon}^-(x, \frac{x}{\varepsilon}) \cdot \theta^-(x,
\frac{x}{\varepsilon}) dx.
\end{array}
\end{equation*}

We can now take the limit as $\varepsilon$ goes to $0$:
\begin{equation*}
\begin{array}{c}
\vspace{0.3cm}\displaystyle\lim_{\varepsilon \to 0} \int_{\Gamma_{\varepsilon}} (u_{\varepsilon}^+(x) -u_{\varepsilon}^-(x))\left( \overline{\varphi}_1^+(x, \frac{x}{\varepsilon}) - \overline{\varphi}_1^-(x, \frac{x}{\varepsilon})\right )ds(x)=\\
\displaystyle- \int_{Y^+} \left(\nabla u(x) + \nabla_y v^+(x,y)
\right) \cdot \overline{\theta}^+(x, y) dxdy -
\displaystyle\int_{Y^-} \left(\nabla u(x) + \nabla_y v^-(x,y)
\right) \cdot \overline{\theta}^-(x, y) dxdy.
\end{array}
\end{equation*}

Finally, the variational formula (\ref{eq:vftheta}) gives us
\begin{equation*}
\begin{array}{l}
\vspace{0.3cm}\displaystyle\lim_{\varepsilon \to 0} \int_{\Gamma_{\varepsilon}} (u_{\varepsilon}^+(x) -u_{\varepsilon}^-(x))\left( \overline{\varphi}_1^+(x, \frac{x}{\varepsilon}) - \overline{\varphi}_1^-(x, \frac{x}{\varepsilon})\right )ds(x) =\\
\hspace{3cm} \displaystyle \int_{\Omega} \int_{\Gamma}\left(v^+(y)
-v^-(y)\right)\left( \overline{\varphi}_1^+(x, y) -
\overline{\varphi}_1^-(x, y)\right)ds(y).
\end{array}
\end{equation*}

For $\vspace{0.1cm}\displaystyle \varphi_{\varepsilon}(x) =
\varepsilon\varphi_1(x, \frac{x}{\varepsilon}),$ with $\varphi_1
 \in \mathcal{D}(\Omega, C^{\infty}_{\sharp}(Y^+))\times
 \mathcal{D}(\Omega, C^{\infty}_{\sharp}(Y^-))$,
 the two-scale limit of the variational formula is
\begin{equation*}
\begin{array}{c}
\vspace{0.3cm}\displaystyle  \int_{\Omega}\! \int_{Y^+}\! k_0 \left (\nabla u(x) + \nabla_y v^+(x, y) \right)\cdot\nabla_y \overline{\varphi}_1^+(x, y) dxdy \\
\vspace{0.3cm}\displaystyle + \int_{\Omega}\! \int_{Y^-} \! k_0
\left
(\nabla u(x) + \nabla_y v^-(x, y) \right)\cdot  \nabla_y \overline{\varphi}_1^-(x, y)dxdy\\
 \displaystyle+ \frac{1}{\beta} \int_{\Omega} \int_{\Gamma}
 \left(v^+(y) -v^-(y)\right)\left( \overline{\varphi}_1^+(x, y)
  - \overline{\varphi}_1^-(x, y)\right)ds(y) =0.
\end{array}
\end{equation*}

By density, this formula hold true for $\varphi_1
 \in L^2(\Omega, H^1_{\sharp}(Y^+))\times
 L^2(\Omega, H^1_{\sharp}(Y^-))$.
One can recognize the formula verified by $u_1^{\pm}$ and the
definition of the cell problems. Hence, separation of variables
and uniqueness of the solutions of the cell problems in $W$ give
\begin{equation*}
\displaystyle v^-(x, y) = u_1^- = \sum_{i=1, 2} \frac{\partial
u_0}{\partial x_i}(x) w_i^{-}(y)
\end{equation*}
and, up to a constant:
\begin{equation*}
\displaystyle v^+(x, y) = u_1^+ = \sum_{i=1, 2} \frac{\partial
u_0}{\partial x_i}(x) w_i^{+}(y).
\end{equation*}

We now choose in the variational formula verified by
$u_{\varepsilon}$ a test function $\vspace{0.1cm} \displaystyle
\varphi_{\varepsilon}(x) = \varphi(x),$ with $\varphi \in
C^{\infty}_c(\overline{\Omega})$.

The limit of (\ref{eq:rpdeweak}) as $\varepsilon$ goes to $0$ is
then given by
\begin{equation*}
\begin{array}{c}
\vspace{0.3cm}\displaystyle  \int_{\Omega}\! \int_{Y^+}\! k_0 \left (\nabla u(x) + \nabla_y v^+(x, y) \right)\cdot\nabla \overline{\varphi}(x) dxdy \\
\vspace{0.3cm}\displaystyle + \int_{\Omega}\! \int_{Y^-} \! k_0 \left (\nabla u(x) + \nabla_y v^-(x, y) \right)\cdot  \nabla\overline{\varphi}(x)dxdy\\
 \displaystyle+  \int_{\partial \Omega} g(x) \overline{\phi}(x) ds(x) =0.
\end{array}
\end{equation*}
By density, this formula hold true for $\varphi \in
H^1(\Omega)$, which leads exactly to the variational formula of the homogenized
problem (\ref{eq:u_0}). Since the solution of this problem is
unique in $H^1_{\mathbb{C}}(\Omega)$, $u_{\varepsilon}$ converges
to $u_0$, not only up to a subsequence. Likewise, $\nabla
u_{\varepsilon}$ two-scale converges to $\nabla u_0 +
\chi_{Y^+}\nabla_y u_1^+ + \chi_{Y^-} \nabla_y u_1^-$.
\end{proof}

\section{Effective admittivity for a dilute suspension}
\label{sec:dilutes}

In general, the effective admittivity given by formula
\eqref{eq:K^*} can not be computed exactly except for a few
configurations. In this section, we consider the problem of
determining the effective property of a suspension of cells when
the volume fraction $|Y^-|$ goes to zero. In other words, the
cells have much less volume than the medium surrounding them. This kind of
suspension is called dilute. Many approximations for the effective
properties of composites are based on the solution for dilute
suspension.

\subsection{Computation of the effective admittivity}

We investigate  the periodic double-layer potential used in
calculating effective permittivity of a suspension of cells. We
introduce the periodic Green function $G_{\sharp}$, for the
Laplace equation in $Y$, given by
\begin{equation*}
\forall x\in Y, \qquad G_{\sharp}(x) = \displaystyle \sum_{n \in
\mathbb{Z}^2\setminus\{0\}} \frac{e^{i2\pi n\cdot
x}}{4\pi^2|n|^2}.
\end{equation*}

The following lemma  from \cite{AKT, book2} plays an essential
role in deriving the effective properties of a suspension in the
dilute limit.
\begin{lem}
The periodic Green function  $G_{\sharp}$ admits the following
decomposition:
\begin{equation}\label{eq:Gperiodic}
\forall x \in Y, \qquad G_{\sharp}(x) = \displaystyle
\frac{1}{2\pi} \ln{|x|} + R_2(x),
\end{equation}
where $R_2$ is a smooth function with the following Taylor expansion at $0$:
\begin{equation}\label{eq:R_2}
R_2(x) = R_2(0) - \displaystyle \frac{1}{4} (x_1^2 + x_2^2) + O(|x|^4).
\end{equation}
\end{lem}

Let $L^2_0(\Gamma) := \left\{\varphi \in L^2(\Gamma)\Big|
\displaystyle \int_{\Gamma} \varphi(x) ds(x) =0 \right\}$.

We define the periodic double-layer potential
$\widetilde{\mathcal{D}}_{\Gamma}$ of the density function
$\varphi \in L^2_0(\Gamma)$:
\begin{equation*}
\begin{array}{c}
\displaystyle \widetilde{\mathcal{D}}_{\Gamma}[\varphi](x) =
\int_{\Gamma} \frac{\partial} {\partial n_y}G_{\sharp}(x-y)
\varphi(y) ds(y).
\end{array}
\end{equation*}

The double-layer potential has the following properties
\cite{book2}.
\begin{lem}
Let $\varphi \in L^2_0(\Gamma)$.
$\widetilde{\mathcal{D}}_{\Gamma}[\varphi]$ verifies:
\begin{equation*}
\begin{array}{lcl}
\vspace{0.3cm} \textrm{(i)} \,& \Delta \widetilde{\mathcal{D}}_{\Gamma}[\varphi] = 0 \,&\,\textrm{in} \,Y^+ ,\\
\vspace{0.3cm} &\Delta \widetilde{\mathcal{D}}_{\Gamma}[\varphi] = 0 \,&\,\textrm{in}\, Y^-,\\
\vspace{0.3cm} \textrm{(ii)} \,& \displaystyle \frac{\partial}{\partial n} \widetilde{\mathcal{D}}_{\Gamma}[\varphi]
\Big|_+ =  \frac{\partial}{\partial n} \widetilde{\mathcal{D}}_{\Gamma}[\varphi]\Big|_-\,&\,\textrm{on}\, \Gamma, \\
\textrm{(iii)} \,& \displaystyle
\widetilde{\mathcal{D}}_{\Gamma}[\varphi]\Big|_{\pm} = \left ( \mp
\frac{1}{2} I + \widetilde{\mathcal{K}}_{\Gamma} \right
)[\varphi]\,&\,\textrm{on}\, \Gamma,
\end{array}
\end{equation*}
where $\widetilde{\mathcal{K}}_{\Gamma} : L^2_0(\Gamma) \mapsto
L^2_0(\Gamma)$ is the Neumann--Poincar\'e operator defined by
\begin{equation*}
\displaystyle \forall x \in \Gamma, \qquad
\widetilde{\mathcal{K}}_{\Gamma}[\varphi](x) = \,\int_{\Gamma}
\frac{\partial} {\partial n_y}G_{\sharp}(x-y) \varphi(y) ds(y).
\end{equation*}
\end{lem}

The following integral representation formula holds.
\begin{thm}
Let $w_i$ be the unique solution in $W$ of
(\ref{eq:w_i}) for $i=1,2$. $w_i$ admits the following integral
representation in $Y$:
\begin{equation}\label{eq:w_i2}
w_i = - \beta k_0 \,\widetilde{\mathcal{D}}_{\Gamma} \left ( I +
\beta k_0 \widetilde{\mathcal{L}}_{\Gamma} \right )^{-1}[n_i],
\end{equation}
where $\widetilde{\mathcal{L}}_{\Gamma} = \displaystyle \frac{\partial
\widetilde{D}_{\Gamma}}{\partial n}$ and $n = (n_i)_{i=1,2}$ is the outward unit normal to $\Gamma$.
\end{thm}

\begin{proof}
Let $\varphi := - \beta k_0 \left ( I + \beta k_0
\widetilde{\mathcal{L}} \right )^{-1}[n_i]$.
$\varphi$ verifies :
\begin{equation*}
 \displaystyle \int_{\Gamma} \varphi(y) ds(y) = -
 \beta k_0 \int_{\Gamma} \frac{\partial}{\partial n} (\widetilde{\mathcal{D}}_{\Gamma}[\varphi](y) +
 y_i) ds(y) = 0.
\end{equation*}
The first equality comes from the definition of $\varphi$ and the second from an integration
 by parts and the fact that $\widetilde{\mathcal{D}}_{\Gamma}[\varphi]$ and $I$ are harmonic. Consequently,
 $\varphi \in L^2_0(\Gamma)$.

We now introduce $\vspace{0.3cm}V_i := \widetilde{\mathcal{D}}_{\Gamma} [\varphi].$
$V_i$ is solution to the following problem:
\begin{equation*}
\left \{
\begin{array}{ll}
\vspace{0.3cm} \nabla \cdot k_0 \nabla V_i = 0 \,&\, \textrm{in} \,\,Y^+,\\
\vspace{0.3cm}  \nabla \cdot k_0 \nabla V_i = 0 \,&\, \textrm{in} \,\,Y^-,\\
\vspace{0.3cm} \displaystyle k_0 \frac{\partial V_i}{\partial n}\Big|_+  =
k_0 \frac{\partial V_i}{\partial n}\Big|_-\,&\, \textrm{on}\, \,\Gamma, \\
\vspace{0.3cm} V_{i}|_+- V_{i}|_- = \varphi \hspace{0.7cm}\,&\, \textrm{on} \,\,\Gamma, \\
y \longmapsto V_{i}(y)  \, \, Y\textrm{-periodic}. &
\end{array}
\right .
\end{equation*}
We use the definitions of $\varphi$ and $V_i$ and recognize that the last problem is exactly problem (\ref{eq:w_i}). The uniqueness of the solution in $W$ gives us the wanted result.
\end{proof}

From Theorem \ref{thm:homo}, the effective admittivity of the
medium $K^*$ is given by
\begin{equation*}
\forall (i, j) \in \{1, 2\}^2, \qquad \displaystyle K^* _{i, j} =
k_0 \left (\delta_{ij} + \int_{Y} \nabla w_i \cdot e_j\right ).
\end{equation*}

After an integration by parts, we get
\begin{equation*}
\forall (i, j) \in \{1, 2\}^2, \qquad \displaystyle K^* _{i, j} =
k_0 \left(\delta_{ij} + \int_{\partial Y} w_i(y) n_j(y) \,ds(y) -
\int_{\Gamma} \left ( w_i^+ - w_i^- \right) n_j(y) \,ds(y) \right).
\end{equation*}

Because of the $Y$-periodicity of $w_i$, we have: $\displaystyle
\int_{\partial Y} w_i(y) n_j \,ds(y) = 0$.

Finally, the integral representation \ref{eq:w_i2} gives us  that
\begin{equation*}
\forall (i, j) \in \{1, 2\}^2, \qquad \displaystyle
 K^* _{i, j} = k_0 \left(\delta_{ij} - (\beta k_0) \int_{\Gamma}
  \left ( I + \beta k_0 \widetilde{\mathcal{L}}_{\Gamma} \right )^{-1}[n_i] n_j \,ds(y) \right).
\end{equation*}

We consider that we are in the context of a dilute suspension,
{\it i.e.}, the size of the cell is small compared to the square:
$\big|Y^-\big| \ll |Y| =1$. We perform the change of variable:
$\displaystyle z=\rho^{-1}y$ with $\rho =|Y^-|^{\frac{1}{2}}$ and
obtain that
\begin{equation*}
\forall (i, j) \in \{1, 2\}^2, \qquad \displaystyle K^* _{i, j}
=k_0 \left( \delta_{ij} - \rho^2 (\beta k_0)\int_{\rho^{-1}\Gamma}
\left ( I + \rho \beta k_0
 \widetilde{\mathcal{L}}_{\Gamma} \right )^{-1}[n_i](\rho z) n_j(z) \,ds(z)\right),
\end{equation*}
where $n$ is the outward unit normal to $\Gamma$. Note that, in
the same way as before, $\beta$ becomes $\rho \beta$ when we
rescale the cell.

Let us introduce $\varphi_i = - \left ( I +\rho \beta k_0
\widetilde{\mathcal{L}}_{\Gamma} \right )^{-1}[n_i]$ and
$\psi_i(z) = \varphi_i(\rho z)$ for all $z \in \rho^{-1}\Gamma$.
From (\ref{eq:Gperiodic}), we get, for any $z\in \rho^{-1}\Gamma$,
after changes of variable in the integrals:
\begin{equation*}
\widetilde{\mathcal{L}}_{\Gamma}[\varphi_i](\rho z) =
\displaystyle \frac{\partial}{\partial
n}\widetilde{\mathcal{D}}_{\Gamma}[\varphi_i](\rho z) = \rho^{-1}
\frac{\partial}{\partial
n}\mathcal{D}_{\rho^{-1}\Gamma}[\psi_i](z) +
\frac{\partial}{\partial n(z)} \int_{\rho^{-1}\Gamma}
\frac{\partial}{\partial n(y)}R_2(\rho z - \rho y) \varphi(\rho y)
ds(y).
\end{equation*}

Besides, the expansion (\ref{eq:R_2}) gives us that the estimate
\begin{equation*}
\nabla R_2(\rho(z-y)) \cdot n(y) = - \frac{\rho}{2} (z-y) \cdot
n(y) + O(\rho^3),
\end{equation*}
holds uniformly in $z,y \in \rho^{-1}\Gamma$.

We thus get the following expansion:
\begin{equation*}
\displaystyle \widetilde{\mathcal{L}}_{\Gamma}[\varphi_i](\rho z)
= \rho^{-1} \mathcal{L}_{\rho^{-1}\Gamma}[\psi_i](z)  -
\frac{\rho}{2}\sum_{j=1, 2} n_j \int_{\rho^{-1}\Gamma} n_j
\psi_i(y) ds(y) + O(\rho^4).
\end{equation*}

Using $\psi_i^*$ defined by \eqref{defpsii} we  get  on
$\rho^{-1}\Gamma$:
\begin{equation}\label{eq:psi_i}
\displaystyle \psi_i = \psi_i^* + \beta k_0
\frac{\rho^2}{2}\sum_{j=1, 2} \psi_j^* \int_{\rho^{-1}\Gamma} n_j(y)
\psi_i(y) ds(y) + O(\rho^4).
\end{equation}

By iterating the formula (\ref{eq:psi_i}), we obtain on
$\rho^{-1}\Gamma$ that
\begin{equation*}
\displaystyle \psi_i = \psi_i^* + \beta k_0
\frac{\rho^2}{2}\sum_{j=1, 2} \psi_j^* \int_{\rho^{-1}\Gamma} n_j(y)
\psi_i^*(y) ds(y) + O(\rho^4).
\end{equation*}

Therefore, one can easily see that Theorem \ref{mainhomog} holds.

\subsection[Maxwell-Wagner-Fricke formula]{Case of concentric circular-shaped cells: the Maxwell-Wagner-Fricke formula}

We consider in this section that the cells are disks of radius
$r_0$. $\rho^{-1}\Gamma$ becomes a circle of radius $r_0$.

For all $g \in L^2(]0, 2\pi[)$, we introduce the Fourier coefficients:
\begin{equation*}
 \forall m \in \mathbb Z,  \quad  \hat{g}(m) = \displaystyle \frac{1}{2\pi} \int_0^{2\pi} g(\varphi) e^{-im\varphi} d\varphi,
\end{equation*}
and have then for all $\varphi \in ]0, 2\pi[$ :
\begin{equation*}
 g(\varphi) = \sum_{m=-\infty}^{\infty} \hat{g}(m) e^{im\varphi}.
\end{equation*}

For $f \in \mathcal{C}^{2, \eta}(\rho^{-1}\Gamma)$, we obtain
after a few computations:
\begin{equation*}
\forall \theta \in  ]0, 2\pi[, \quad (I + \beta k_0
\mathcal{L}_{\rho^{-1}\Gamma})^{ -1}[f](\theta) =  \displaystyle
\sum_{n \in \mathbb{Z}^*}\left( 1 + \beta k_0\, \frac{|n|}{2 r_0}
\right ) ^{-1} \hat{f}(n)\, e^{in\theta}.
\end{equation*}

For $p=1,2$, $\psi_p^* = - (I + \beta k_0
\mathcal{L}_{\rho^{-1}\Gamma})^{ -1}[n_p]$ then have the following
expression:
\begin{equation*}
\forall \theta \in  ]0, 2\pi[, \quad \psi_p^* = - \displaystyle
\left( 1 + \, \frac{\beta k_0}{2 r_0}
 \right ) ^{-1} n_p .
\end{equation*}

Consequently, we get for $(p, q) \in \{1, 2\}^2$ :
\begin{equation*}
M_{p, q} = - \delta_{pq} \displaystyle \frac{\beta k_0 \pi r_0}{1
+ \, \displaystyle \frac{\beta k_0}{2 r_0}},
\end{equation*}
and hence,
\begin{equation} \label{im_m}
\Im M_{p, q} = \displaystyle \delta_{p,q} \frac{\pi r_0 \delta
\omega (\epsilon_m \sigma_0 - \epsilon_0
\sigma_m)}{\displaystyle(\sigma_m +  \sigma_0 \,\frac{\delta}{2 r_0})^2 + \omega^2
(\epsilon_m + \epsilon_0 \,\frac{\delta}{2 r_0})^2}.
\end{equation}

Formula (\ref{im_m}) is the two-dimensional version of the
Maxwell-Wagner-Fricke formula, which gives the effective
admittivity of a dilute suspension of spherical cells covered by a
thin membrane.

An explicit formula for the case of elliptic  cells can be derived
by using the spectrum of the integral operator
$\mathcal{L}_{\rho^{-1}\Gamma}$, which can be identified by
standard Fourier methods \cite{KPS}.

\subsection{Debye relaxation times}
From
\eqref{im_m}, it follows that the imaginary part of the membrane
polarization attains its maximum with respect to the frequency  at
$$\displaystyle \frac{1}{\tau} = \frac{\displaystyle\sigma_m + \sigma_0\, \frac{\delta
}{2 r_0}}{\displaystyle\epsilon_m +  \epsilon_0\, \frac{\delta}{2
r_0}}.
$$ This dispersion phenomenon due to the membrane polarization is well known and referred to as the
 $\beta$-dispersion. The associated characteristic time $\tau$ corresponds to a Debye relaxation time.

For arbitrary-shaped cells, we define the first and second Debye
 relaxation times, $\tau_i, i=1,2$, by
\begin{equation} \label{deftaui} \frac{1}{\tau_i}:= \mathrm{arg}\max_{\omega}
| \lambda_i(\omega)|, \end{equation} where $\lambda_1 \leq
\lambda_2$ are the eigenvalues of the imaginary part of the
membrane polarization tensor $M(\omega)$. Note that if the cell is
of circular shape, $\lambda_1=\lambda_2$.

As it will be shown later, the Debye  relaxation times can be used
for identifying the microstructure.

\subsection{Properties of the membrane polarization tensor and the Debye  relaxation times}

In this subsection, we derive important properties of the membrane
polarization tensor and the Debye  relaxation times defined
respectively by \eqref{defM} and \eqref{deftaui}. In particular,
we prove that the Debye relaxation times are invariant with
respect to translation, scaling, and rotation of the cell.

First, since the kernel of $\mathcal{L}_{\rho^{-1}\Gamma}$ is
invariant with respect to translation, it follows that $M(C, \beta
k_0)$ is invariant with respect to translation of the cell $C$.

Next, from the scaling properties of the kernel of
$\mathcal{L}_{\rho^{-1}\Gamma}$ we have
$$
M(s C, \beta k_0) = s^2 M(C, \frac{\beta k_0}{s})$$ for any
scaling parameter $s>0$.

Finally, we have $$ M(\mathcal{R} C, \beta k_0) = \mathcal{R}
M(C,\beta k_0) \mathcal{R}^t \quad \mbox{for any rotation }
\mathcal{R},$$ where $t$ denotes the transpose.

Therefore, the Debye  relaxation times are translation and
rotation invariant. Moreover, for scaling, we have
$$
\tau_i(h C, \beta k_0) =  \tau_i(C, \frac{\beta k_0}{h}), \qquad
i=1,2, \quad h>0.$$ Since $\beta$ is proportional to the thickness
of the cell membrane, $\beta/h$ is nothing else than the real
rescaled coefficient $\beta$ for the cell $C$. The Debye
 relaxation times $(\tau_i)$ are therefore invariant by scaling.

Since $\mathcal{L}_{\rho^{-1}\Gamma}$ is self-adjoint, it follows
that $M$ is symmetric. Finally, we show positivity of the
imaginary part of the matrix $M$ for $\delta$ small enough. 

We consider that the cell contour $\Gamma$ can be parametrized by polar coordinates. We have, up to $O(\delta^3)$,
\begin{equation} \label{mbetadelta}
M +  \beta \rho^{-1} |\Gamma| =  - \beta^2 \int_{\rho^{-1} \Gamma}
n \mathcal{L}_{\rho^{-1}\Gamma}[n] \, ds,\end{equation} where
again we have assumed that $\sigma_0=1$ and $\epsilon_0=0$.

Recall that $$ \beta = \frac{\delta \sigma_m}{\sigma_m^2 +
\omega^2 \varepsilon_m^2} - i \frac{\delta \omega
\varepsilon_m}{\sigma_m^2 + \omega^2 \varepsilon_m^2}.$$ Hence,
the positivity of $\mathcal{L}_{\rho^{-1}\Gamma}$ yields $$ \Im\,
M \geq \frac{\delta \omega \varepsilon_m}{ 2\rho (\sigma_m^2 +
\omega^2 \varepsilon_m^2)}  |\Gamma| I
$$ for $\delta$ small enough, where $I$ is the identity matrix.

Finally, by using \eqref{mbetadelta} one can see that the
eigenvalues of $\Im \, M$ have one maximum each with respect to
the frequency. Let $l_i, i=1,2$, $l_1\geq l_2$, be the eigenvalues
of $\int_{\rho^{-1}\Gamma} n \mathcal{L}_{\rho^{-1}\Gamma}[n] ds$.
 We have
\begin{equation} \label{formli}
\lambda_i = \frac{\delta \omega \varepsilon_m}{ \rho (\sigma_m^2 +
\omega^2 \varepsilon_m^2)}  |\Gamma|  - \frac{2\delta^2 \omega
\varepsilon_m \sigma_m}{(\sigma_m^2 + \omega^2 \varepsilon_m^2)^2}
l_i, \quad i=1,2.\end{equation} Therefore, $\tau_i$ is the inverse
of the positive root of the following polynomial in $\omega$:
$$
- \varepsilon_m^4 |\Gamma| \omega^4 + 6 \delta \varepsilon_m^2
\sigma_m l_i \rho \omega^2 + \sigma_m^4 |\Gamma|.
$$

\subsection{Anisotropy measure} \label{subsectionanistropy}
Anisotropic electrical properties can be found in biological
tissues such as muscles and nerves. In this subsection, based on
formula (\ref{dilutethmf}), we introduce a natural measure of the
conductivity anisotropy and derive its dependence on the frequency
of applied current. Assessment of electrical anisotropy of muscle
may have useful clinical application. Because neuromuscular
diseases produce substantial pathological changes, the anisotropic
pattern of the muscle is likely to be highly disturbed
\cite{rutkov1,rutkov2}. Neuromuscular diseases could lead to a
reduction in anisotropy for a range of frequencies as the muscle
fibers are replaced by isotropic tissue.

Let $\lambda_1 \leq \lambda_2$ be the eigenvalues of the imaginary
part of the membrane polarization tensor $M(\omega)$. The function
$$\omega \mapsto \frac{\lambda_1(\omega)}{\lambda_2(\omega)}$$
can be used as a measure of the anisotropy of the conductivity of
a dilute suspension. Assume $\epsilon_0=0$. As frequency
$\omega$ increases, the factor $\beta k_0$ decreases. Therefore,
for large $\omega$, using the expansions in (\ref{formli}) we
obtain that
\begin{equation} \label{fanisotropic}
\frac{\lambda_1(\omega)}{\lambda_2(\omega)} = 1 + (l_1-l_2)
\frac{2 \delta \sigma_m \rho}{(\sigma_m^2 + \omega^2
\varepsilon_m^2) |\Gamma| } + O(\delta^2),
\end{equation}
where $l_1 \leq l_2$ are the eigenvalues of
$\int_{\rho^{-1}\Gamma} n \mathcal{L}_{\rho^{-1}\Gamma}[n] ds$.

Formula (\ref{fanisotropic}) shows that as the frequency
increases, the conductivity anisotropy decreases. The anisotropic
information can not be captured for
$$
\omega \gg \frac{1}{\varepsilon_m} ((l_1-l_2) \frac{2 \delta
\sigma_m \rho}{|\Gamma| } - \sigma_m^2)^{1/2}.
$$

\section{Spectroscopic imaging of a dilute suspension}
\label{sec:spectro}

\subsection{Spectroscopic conductivity imaging}
We now make use of the asymptotic expansion of the effective
admittivity in terms of the volume fraction $f=\rho^2$ to image a
permittivity inclusion. Consider $D$ to be a bounded domain in
$\Omega$ with admittivity $1+ f M(\omega)$, where $M(\omega)$ is a
membrane polarization tensor and $f$ is the volume fraction of the
suspension in $D$. The inclusion $D$ models a suspension of cells
in the background $\Omega$. For simplicity, we neglect the
permittivity $\epsilon_0$ of $\Omega$ and assume that its
conductivity $\sigma_0=1$. We also assume that $M(\omega)$ is
isotropic. At the macroscopic scale, if we inject a current $g$ on
$\partial \Omega$, then the electric potential satisfies:
\begin{equation} \label{eqspectro} \left \{
\begin{array}{ll}
\vspace{0.3cm} \nabla \cdot ( 1 + f M(\omega) \chi_{D}) \nabla u
= 0 \,&\,
 \textrm{in}\, \,\Omega,\\
\vspace{0.3cm} \displaystyle \frac{\partial u }{\partial n}
\Big|_{\partial \Omega} = g, \hspace{0.3cm}\displaystyle
\int_{\partial \Omega} g(x) ds(x)= 0, \hspace{0.3cm}\displaystyle
\int_{\Omega} u(x) dx= 0.
\end{array}
\right .
\end{equation}
The imaging problem is to detect and characterize $D$ from
measurements of $u$ on $\partial \Omega$.

Integrating by parts and using the trace theorem for the
double-layer potential \cite{14,2}, we obtain,  $\forall \, x\in
\partial \Omega$,
\begin{equation} \label{handside}
\begin{array}{l}
\displaystyle \frac{1}{2} u(x) + \frac{1}{2\pi} \int_{\partial
\Omega} \frac{(x-y)\cdot n(x)}{|x-y|^2} u(y)ds(y) +
\frac{1}{2\pi} \int_{\partial \Omega} g(y) \ln |x-y| ds(y) \\
\nm \qquad \displaystyle = \frac{f}{2\pi} M(\omega) \int_D \nabla
u(y)\cdot \frac{(x-y)}{|x-y|^2} dy. \end{array}
\end{equation} Since $f$ is small,
$$
\int_D \nabla u(y)\cdot \frac{(x-y)}{|x-y|^2} dy \simeq \int_D
\nabla U(y)\cdot \frac{(x-y)}{|x-y|^2} dy$$ holds uniformly for
$x\in
\partial \Omega$, where $U$ is the background solution, that is,
$$\left \{
\begin{array}{ll}
\vspace{0.3cm} \Delta U = 0 \,&\,
 \textrm{in}\, \,\Omega,\\
\vspace{0.3cm} \displaystyle \frac{\partial U }{\partial n}
\Big|_{\partial \Omega} = g, \hspace{0.3cm}\displaystyle
\int_{\Omega} U(x) dx= 0,
\end{array}
\right .
$$
Therefore, taking the imaginary part of \eqref{handside} yields
\begin{equation} \label{handside2}
\begin{array}{l}
\displaystyle  \frac{1}{2} \Im u(x) + \frac{1}{2\pi}
\int_{\partial \Omega} \frac{(x-y)\cdot n(x)}{|x-y|^2} \Im
u(y)ds(y)  \simeq \frac{f}{2\pi} \Im M(\omega)
\int_D \nabla U(y)\cdot \frac{(x-y)}{|x-y|^2} dy,
\end{array}
\end{equation}
uniformly for  $x\in
\partial \Omega$, provided that $g$ is real.  Finally, taking the argument of the
maximum of the right-hand side in \eqref{handside2} with respect
to the frequency $\omega$ gives the Debye  relaxation time of the
suspension in $D$.

\subsection{Selective spectroscopic imaging}

A challenging applied problem is to design a selective
spectroscopic imaging approach for suspensions of cells. Using a
pulsed imaging approach \cite{pulse2, pulse}, we propose a simple
way to selectively image dilute suspensions. Again, we assume for
the sake of simplicity that $\epsilon_0=0$ and $\sigma_0=1$.

In the time-dependant regime, the electrical model for the cell
\eqref{modelcell} is replaced with
$$
u(x,t)=\int \hat{h}(\omega) \hat{u}(x,\omega) e^{i \omega t} d\omega,
$$
where $\hat{u}(x,\omega)$ is the solution to
\begin{equation} \label{modelcellpusle}
\left \{
\begin{array}{l}
\begin{array}{ll}
\vspace{0.3cm}  \Delta \hat{u}(\cdot,\omega) = 0 \,&\, \textrm{in}\, \,D\setminus \overline{C},\\
\vspace{0.3cm}  \Delta  \hat{u}(\cdot,\omega) = 0 \,&\, \textrm{in}\, \,C,\\
\vspace{0.3cm} \displaystyle  \frac{\partial
\hat{u}(\cdot,\omega)}{\partial n}\Big|_+ =
 \frac{\partial  \hat{u}(\cdot,\omega)}{\partial n}\Big|_- \,&\, \textrm{on}\,\, \Gamma,\\
\vspace{0.3cm}  \hat{u}(\cdot,\omega)|_+- \hat{u}(\cdot,\omega)|_-
-\,\beta(\omega)  \displaystyle \frac{\partial
\hat{u}(\cdot,\omega)}{\partial n} = 0 \hspace{0.7cm}\,&\,
\textrm{on}\, \,\Gamma,
\end{array} \\
\displaystyle \frac{\partial \hat{u}(\cdot,\omega)}{\partial n}
\Big|_{\partial D} = f, \hspace{0.3cm}\displaystyle \int_{\partial
D} \hat{u}(\cdot,\omega) ds =0,
\end{array}
\right .\end{equation} and $$h(t) = \int \hat{h}(\omega)
e^{i\omega t} d\omega$$ is the pulse shape. The support of $h$ is
assumed to be compact.

At the macroscopic scale, if we inject a pulsed current,
$g(x)h(t)$, on $\partial \Omega$, then the electric potential
$u(x,t)$ in the presence of a suspension occupying $D$ is given by
$$
u(x,t)=\int \hat{h}(\omega) \hat{u}(x,\omega) e^{i \omega t} d\omega,
$$
where
$$\left \{
\begin{array}{ll}
\vspace{0.3cm} \nabla \cdot ( 1 + f M(\omega) \chi_{D}) \nabla
\hat{u}(\cdot, \omega) = 0 \,&\,
 \textrm{in}\, \,\Omega,\\
\vspace{0.3cm} \displaystyle \frac{\partial \hat{u}(\cdot, \omega)
}{\partial n} \Big|_{\partial \Omega} = g,
\hspace{0.3cm}\displaystyle \int_{\partial \Omega} \hat{u}(\cdot,
\omega) ds= 0.
\end{array}
\right .
$$
Assume that we are in the presence of two suspensions occupying
the domains $D_1$ and $D_2$ inside $\Omega$. From \eqref{handside}
it follows that
\begin{equation} \label{handside3}
\begin{array}{l}
\displaystyle \frac{1}{2} \hat{u}(x,\omega) + \frac{1}{2\pi}
\int_{\partial \Omega} \frac{(x-y)\cdot n(x)}{|x-y|^2}
\hat{u}(y,\omega)ds(y) +
\frac{1}{2\pi} \int_{\partial \Omega} g(y) \ln |x-y| ds(y) \\
\nm \qquad \displaystyle  \simeq \frac{f_1}{2\pi}  M_1(\omega)
\int_{D_1} \nabla U(y)\cdot \frac{(x-y)}{|x-y|^2} dy +
\frac{f_2}{2\pi} M_2(\omega) \int_{D_2} \nabla U(y)\cdot
\frac{(x-y)}{|x-y|^2} dy,
\end{array}
\end{equation}
and therefore,
\begin{equation} \label{handside4}
\begin{array}{l}
\displaystyle \frac{1}{2} u(x,t) + \frac{1}{2\pi} \int_{\partial
\Omega} \frac{(x-y)\cdot n(x)}{|x-y|^2} u(y,t)ds(y) +
\frac{1}{2\pi} h(t) \int_{\partial \Omega} g(y) \ln |x-y| ds(y) \\
\nm \qquad \displaystyle   \simeq \frac{f_1}{2\pi}
\mathcal{M}_1(t) \int_{D_1} \nabla U(y)\cdot \frac{(x-y)}{|x-y|^2}
dy + \frac{f_2}{2\pi} \mathcal{M}_2(t) \int_{D_2} \nabla U(y)\cdot
\frac{(x-y)}{|x-y|^2} dy,
\end{array}
\end{equation}
uniformly in $x \in \partial \Omega$ and $t \in \mathrm{supp}\,
h$, where
$$
\mathcal{M}_i(t) := \int \hat{h}(\omega) M_i(\omega) e^{i \omega
t} d\omega, \quad i=1,2.
$$
As it will be shown in section \ref{sect:numer}, by comparing the
Debye relaxation times associated to $M_1$ and $M_2$, one can
design the pulse shape $h$ in order to image selectively $D_1$ or
$D_2$.  For example, one can selectively image $D_1$ by taking
$\hat{h}(\omega)$ close to zero around the Debye  relaxation time
of $M_2$ and close to one around the Debye relaxation time of
$M_1$.

\subsection{Spectroscopic measurement of anisotropy}
In this subsection we assume that $M$ is anisotropic and consider
the solution $u$ to (\ref{eqspectro}). We want to assess the
anisotropy of the inclusion $D$ of admittivity $1 + f M(\omega)$
from measurements of $u$ on the boundary $\partial \Omega$.

From (\ref{handside2}) it follows that
\begin{equation} \label{handside22}
\begin{array}{l}
\displaystyle \int_{\partial \Omega} g(x) \bigg[\frac{1}{2} \Im
u(x) + \frac{1}{2\pi} \int_{\partial \Omega} \frac{(x-y)\cdot
n(x)}{|x-y|^2} \Im u(y)ds(y) \bigg] ds(x) \\ \nm \qquad
\displaystyle  \simeq \frac{f}{2\pi} \int_D \Im M(\omega) \nabla
U(y)\cdot  \nabla U(y) dy,
\end{array}
\end{equation}
provided that $g$ is real. Now, taking constant current sources
corresponding to $g= a \cdot n$, where $a\in \R^2$ is a unit
vector, yields
$$
\displaystyle \mathcal{S}[a]:=\int_{\partial \Omega} g(x)
\bigg[\frac{1}{2} \Im u(x) + \frac{1}{2\pi} \int_{\partial \Omega}
\frac{(x-y)\cdot n(x)}{|x-y|^2} \Im u(y)ds(y) \bigg] ds(x) \simeq
\frac{f}{2\pi} \Im M(\omega) |a|^2|D|.$$ Since
$$
\frac{\min_a \mathcal{S}[a]}{\max_a \mathcal{S}[a]} \simeq
\frac{\lambda_1(\omega)}{\lambda_2(\omega)},$$ where  $\lambda_1$
and $\lambda_2$ (with $\lambda_1\leq \lambda_2$) are the
eigenvalues of $\Im M$, it follows from subsection
\ref{subsectionanistropy} that
$$
\omega \mapsto \frac{\min_a \mathcal{S}[a]}{\max_a
\mathcal{S}[a]}$$ is a natural measure of conductivity anisotropy.
This measure may be used for the detection and classification of
neuromuscular diseases via measurement of muscle anisotropy
\cite{rutkov1,rutkov2}.

\section{Stochastic homogenization of randomly deformed
conductivity resistant membranes} \label{sec:stoch}

The first main result of this section is to show that a rigorous
homogenization theory can be derived when the cells (and hence
interfaces) are randomly deformed from a periodic structure, and
the random deformation is ergodic and stationary in the sense of
\eqref{eq:stati}.



\subsection{Auxiliary problem: proof of Theorem
\ref{thm:auxiliary}}

In this subsection, we prove Theorem \ref{thm:auxiliary}, that is
the existence and uniqueness of the auxiliary problem. As in many
stochastic homogenization problems, this is the key step. The main
difficulty, as usual, lies in the loss of compactness.

Our strategy is as follows: First, an absorption term is added to
regularize the problem which gains back some compactness: the
sequence of regularized solutions, which correspond to a sequence
of vanishing regularization, have a converging gradient. Secondly,
the potential field corresponds to the limiting gradient is shown
to be a solution to the auxiliary problem. Finally, using
regularity results and sub-linear growth of potential field with
stationary gradient, we prove that the solution to the auxiliary
problem is unique.
\smallskip

\begin{proof}[Proof of Theorem \ref{thm:auxiliary}] {\it Step 1: The regularized auxiliary problem.}
Fix $p \in \R^2$. Consider the following regularized problem where
an absorption $\alpha > 0$ is added.
\begin{equation}
\left\{
\begin{aligned}
\nabla\cdot k_0 (\nabla w^\pm_{p,\alpha} (y) + p) + \alpha w^\pm_{p,\alpha}= 0 &\quad
\text{ in } &&\Phi(\R^\pm_2,\gamma),\\
n \cdot k_0 \nabla w^-_{p,\alpha}(y) = n \cdot k_0\nabla
w^-_{p,\alpha}(y), &\quad \text{and }&& w^+_{p,\alpha} -
w^-_{p,\alpha} = \beta k_0 n \cdot \nabla w^-_{p,\alpha} \text{ in
}
 \Phi(\Gamma_2,\gamma),\\
w^\pm_{p,\alpha} (y,\gamma) = \widetilde{w}^\pm_{p,\alpha}
(\Phi^{-1}(y,\gamma),\gamma), &\quad \text{and }&&
\widetilde{w}^\pm_{p,\alpha} \text{ are stationary}.
\end{aligned}
\right. \label{eq:regaux}
\end{equation}
Define the space $\mathcal{H} = \{w = \widetilde{w} \circ
\Phi^{-1} ~|~  \widetilde{w} \in H^1_{\rm loc}(\R_2^+) \times
H^1_{\rm loc}(\R_2^+), \widetilde{w} \text{ is stationary}\}$.
More precisely, this means the space of functions $w =
\widetilde{w}\circ \Phi^{-1}$ where $\widetilde{w}$ restricted in
$\R_2^+$ (respectively $\R_2^-$) is in $\in H^1_{\rm loc}(\R_2^+)$
(respectively $H^1_{\rm loc}(\R_2^-)$), and in addition,
$\widetilde{w}^\pm$ are stationary.

Equip $\mathcal{H}$ with the inner product
\begin{equation}
(u,v)_{\mathcal{H}} = \E \left(\int_{Y^+} \nabla \tilde{u} \cdot
\nabla \overline{\tilde{v}} dx + \int_{Y^-} \nabla \tilde{u}\cdot
\nabla \overline{\tilde{v}} dx + \int_Y \tilde{u}
\overline{\tilde{v}} dx \right).
\end{equation}
Then $\mathcal{H}$ is a Hilbert space. Define the bilinear form
\begin{equation*}
\begin{aligned}
A_\alpha(u,v) = \ &\E \left(\int_{\Phi(Y^+)} k_0 \nabla u^+ \cdot \overline{\nabla v^+} dx + \int_{\Phi(Y^-)}k_0 \nabla u^- \cdot \overline{\nabla v^-} dx\right.\\
& \quad\quad\quad + \left.\alpha \int_{\Phi(Y)} u \tilde{v} dx +
\frac{1}{\beta} \int_{\Phi(\Gamma_0)} (u^+ - u^-) \overline{(w^+ -
w^-)} ds \right),
\end{aligned}
\end{equation*}
and the linear functional
\begin{equation*}
b_p(v) = -k_0\E\left( \int_{\Phi(Y^+)} p\cdot \overline{\nabla
v^+} dx + \int_{\Phi(Y^-)} p \cdot \overline{\nabla v^-} dx +
\int_{\Phi(\Gamma_0)} (n(x) \cdot p) \overline{(v^+ - v^-)}(x)
ds(x) \right).
\end{equation*}
For a fixed $\alpha>0$, we verify that $A_\alpha$ is bounded and
coercive, and $b_p$ is bounded. By the Lax--Milgram theorem, there
exists a unique $w_{p,\alpha} \in \mathcal{H}$ such that
\begin{equation}
A_\alpha(w_{p,\alpha}, \varphi) = b_p(\varphi), \quad \forall
\varphi \in \mathcal{H}. \label{eq:A=b}
\end{equation}
In fact, the solution satisfies \eqref{eq:regaux} in the
distributional sense $\Pb$-a.s.~in $\mathcal{O}$. Furthermore, the
following estimates are immediate:
\begin{equation}
\E \int_{Y^\pm} |\nabla \tilde{w}^\pm_{p,\alpha} |^2 \le C, \quad
\E \int_{\Gamma_0}
|\tilde{w}^+_{p,\alpha}-\tilde{w}^-_{p,\alpha}|^2 \le C, \quad \E
\int_{Y^\pm} |\tilde{w}^\pm_{p,\alpha}|^2 \le \frac{C}{\alpha}.
\label{eq:retabdd1}
\end{equation}
Apply the extension operators in Corollary \ref{cor:extRdm} and
Corollary \ref{cor:extPhiRdm}. We get the sequences
$\{\tilde{w}^{\rm ext}_{p,\alpha} = P\tilde{w}^+_{p,\alpha}\}
\subset H^1_{\rm loc}(\R^2)$ and $\{w^{\rm ext}_{p,\alpha} =
P_\gamma w^+_{p,\alpha}\}$. Further, $\{\tilde{w}^{\rm
ext}_{p,\alpha}\}$ are stationary. They satisfy that $w^{\rm
ext}_{p,\alpha} = \tilde{w}^{\rm ext}_{p,\alpha}\circ \Phi^{-1}$
and that
\begin{equation}
\E \int_{Y} |\nabla \tilde{w}^{\rm ext}_{p,\alpha} |^2 \le C,
\quad \E \int_{\Gamma_0} |\tilde{w}^{\rm
ext}_{p,\alpha}-\tilde{w}^-_{p,\alpha}|^2 \le C, \quad \E \int_{Y}
|\tilde{w}^{\rm ext}_{p,\alpha}|^2 \le \frac{C}{\alpha}.
\label{eq:retabdd}
\end{equation}
\smallskip

{\it Step 2: Converging subsequences as the regularization
parameter vanishes.} Thanks to the above estimates, there exists
some subsequence, still denoted by $\nabla \tilde{w}^{\rm
ext}_{p,\alpha}$, which converges weakly as $\alpha \downarrow 0$
to a function $\tilde{\eta}^{\rm ext}_p \in [L^2_{\rm loc}(\R^2,
L^2(\mathcal{O}))]^2$, where $\tilde{\eta}^{\rm ext}_p$ is
stationary. By a change of variable, we also have that $\nabla
w^{\rm ext}_{p,\alpha}$ converges in $[L^2_{\rm
loc}(\R^2,L^2(\mathcal{O}))]^2$ to $\eta^{\rm ext}_p$ and
\begin{equation}
\eta^{\rm ext}_p(y,\gamma) = \nabla_y \Psi(y,\gamma)
\tilde{\eta}^{\rm ext}_p(\tilde{y},\gamma), \label{eq:gxigxit}
\end{equation}
where $\Psi = \Phi^{-1}$ and $\tilde{y} = \Psi(y)$. Moreover, as
gradients, $\nabla_{\tilde{y}} \tilde{w}^{\rm ext}_{p,\alpha}$ and
$\nabla_y w^{\rm ext}_{p,\alpha}$ are curl free. This property is
preserved by their limits:
\begin{equation}
\partial_{y_i} (\eta^{\rm ext}_p)_j = \partial_{y_j} (\eta^{\rm ext}_p)_i, \quad \partial_{\tilde{y}_i} (\tilde{\eta}^{\rm ext}_p)_j = \partial_{\tilde{y}_j} (\tilde{\eta}^{\rm ext}_p)_i, \quad i,j=1,\cdots,d.
\end{equation}
That is to say, $\eta^{\rm ext}_p$ and $\tilde{\eta}^{\rm ext}_p$
are also gradient functions. Consequently, there exist $w_p^{\rm
ext}$ and $\tilde{w}^{\rm ext}_p$ such that $\eta^{\rm ext}_p =
\nabla_y w_p^{\rm ext}$ and $\tilde{\eta}^{\rm ext}_p =
\nabla_{\tilde{y}} \tilde{w}^{\rm ext}_p$. The relation
\eqref{eq:gxigxit} implies that $w^{\rm ext}_p(y) = \tilde{w}^{\rm
ext}_p(\Psi(y,\gamma),\gamma) + C_p(\gamma)$ where $C_p(\gamma)$
is a random constant. We hence re-define $\tilde{w}^{\rm ext}_p$
by adding to it the random variable $C_p$ so that $w^{\rm ext}_p =
\tilde{w}^{\rm ext}_p \circ \Psi$. By the same token, we have that
$\nabla \tilde{w}^-_{p,\alpha}$ and $\nabla w^-_{p,\alpha}$
converge (along the above subsequence) to $\tilde{\eta}^-_p \in
[L^2_{\rm loc}(\R_2^-)]^2$ and $\eta^-_p \in [L^2_{\rm
loc}(\Phi(\R_2^-))]^2$ respectively. Further, $\eta^-_p$ is
stationary; in addition, for some $\tilde{w}^-_p \in H^1_{\rm
loc}(\R_2^-)$ and $w^-_p \in H^1_{\rm loc}(\R_2^-)$ satisfying
that $w^-_p = \tilde{w}^-_p \circ \Psi$, we have $\tilde{\eta}^-_p
= \nabla \tilde{w}^-_p$ and $\eta^-_p = \nabla w^-_p$.

Repeating the above argument with the help of the second
inequality in \eqref{eq:retabdd}, one observes that
$\{\tilde{w}^{\rm ext}_{p,\alpha} - \tilde{w}^-_{p,\alpha}\}$
restricted to the interface $\GGamma$ converges to some
$\tilde{\zeta}_p \in L^2_{\rm loc}(\GGamma)$ and $\tilde{\zeta}_p$
is stationary. Similarly, by a change of variable, $w^{\rm
ext}_{p,\alpha} - w^-_{p,\alpha}$ converges to $\zeta_p =
\tilde{\zeta}_p \circ \Psi$ and $\zeta_p \in L^2_{\rm
loc}(\Phi(\GGamma))$.

Since $\tilde{w}^{\rm ext}_{p,\alpha}$ is an extension of
$\tilde{w}^+_{p,\alpha}$, the inequality \eqref{eq:extRdm} holds.
Also, since $\tilde{w}^{\rm ext}_{p,\alpha}$ is stationary, one
has $\E \int_Y \nabla \tilde{w}^{\rm ext}_{p,\alpha} d\tilde{y} =
0$. Passing to the limit, we get
\begin{equation}
\E \int_Y \nabla_{\tilde{y}} \tilde{w}^{\rm ext}_p(\tilde{y})
d\tilde{y} = 0, \quad \text{and} \quad \E\int_Y
|\nabla_{\tilde{y}} \tilde{w}^{\rm ext}_p|^s d\tilde{y} \le C
\E\int_{Y^+}  |\nabla_{\tilde{y}} \tilde{w}^{\rm ext}_p|^s
d\tilde{y}, \label{eq:meanzero}
\end{equation}
where $C$ depends on the same parameters as in \eqref{eq:extRdm}.
Similarly, we also have that
\begin{equation}
 \E\int_{\Phi(Y)} |\nabla_y w^{\rm ext}_p|^s dy
\le C \E\int_{\Phi(Y^+)}  |\nabla_y w^{\rm ext}_p|^s dy,
\end{equation}
where $C$ depends on the same parameters as in
\eqref{eq:extPhiRdm}. Here and above, $s\ge 1$ is some parameter
so that the right-hand sides are finite.
\smallskip

{\itshape Step 3: The limit solves the auxiliary problem.} Take
the limiting functions $w^{\rm ext}_p$ and $\tilde{w}^{\rm ext}_p$
from last step. Let $w^+_p$ be the restriction of $w^{\rm ext}_p$
to $\Phi(\R_2^+)$, $\tilde{w}^+_p$ be the restriction of
$\tilde{w}^{\rm ext}_p)$ to $\R_2^+$. Then by the above
construction and \eqref{eq:meanzero}, the last two equations of
\eqref{eq:auxiliary} are satisfied.

Let us verify that $(w^+_p, w^-_p)$ satisfies the equation and the
boundary conditions in \eqref{eq:auxiliary}. Recall the weak
formulation \eqref{eq:A=b} for the regularized equation
\eqref{eq:regaux}. Pass to the limit $\alpha \to 0$ along the
subsequence found above. In particular, we observe that
Cauchy--Schwarz and the last inequality of \eqref{eq:retabdd1}
imply that
\begin{equation*}
\left| \alpha\ \E \int_{\Phi(Y)} w_{p,\alpha} \overline{\varphi}\
dy \right| \le \sqrt{\alpha} \left(\E \int_{\Phi(Y)} \alpha
|w_{p,\alpha}|^2 dy\right)^{\frac 1 2} \left(\E \int_{\Phi(Y)}
|\varphi|^2 dy\right)^{\frac 1 2} \longrightarrow 0.
\end{equation*}
As a result, we obtain in the limit that for all $(\varphi^+,
\varphi^-) \in \mathcal{H}$,
\begin{equation}
\begin{aligned}
\E\left(\int_{\Phi(Y^+)} k_0(p+ \nabla w^+_p) \cdot \overline{\nabla \varphi^+} dx\right. &+ \left.\int_{\Phi(Y^-)}k_0 (p+ \nabla w^-_p) \cdot \overline{\nabla \varphi^-} dx\right) \hfill \\
&+ \frac{1}{\beta} \E\left(\int_{\Phi(\Gamma_0)} \zeta_p
\overline{(\varphi^+ - \varphi^-)} ds \right) = 0
\label{eq:proaux1}
\end{aligned}
\end{equation}
By letting $\varphi$ in $C^\infty_0(\Phi(Y^+)) \cap \mathcal{H}$
and $C^\infty_0(\Phi(Y^-)) \cap \mathcal{H}$ respectively, we see
that the first line of \eqref{eq:auxiliary} is satisfied in the
distributional sense. It is worth mentioning that
$C^\infty_0(\Phi(Y^\pm)) \cap \mathcal{H}$ is the set of test
functions in $\mathcal{H}$ (being stationary) and whose
restrictions to the unit cell is in the first space,
$C_0^\infty(\Phi(Y^\pm))$, but the function itself is not
compactly supported there. To verify the second line of
\eqref{eq:auxiliary}, we first take $\varphi \in
C^\infty_0(\Phi(Y)) \cap \mathcal{H}$ and apply integration by
parts to get
\begin{equation*}
\E \int_{\Phi(\Gamma_0)} k_0 (n \cdot \nabla w^+_p - n \cdot
\nabla w^-_p) \overline{\varphi} ds = 0, \quad\quad \forall
\varphi \in C^\infty_0(\Phi(Y)) \cap \mathcal{H}.
\end{equation*}
This implies the first boundary condition. For the second
condition, consider arbitrary $\varphi \in
C^\infty(\overline{\Phi(Y^-)}) \cap \mathcal{H}$ and let the test
function in \eqref{eq:proaux1} be $\varphi \chi_{\Phi(\R_2^-)}$.
By the first boundary condition and integration by parts formula,
we get
$$
\E \int_{\Phi(\Gamma_0)} \left[k_0 n \cdot (p + \nabla w^-_p) -
\frac{1}{\beta} \zeta_p\right] \overline{\varphi} ds = 0,
\quad\quad \forall \varphi \in C^\infty(\overline{\Phi(Y^-)}) \cap
\mathcal{H}.
$$
This implies that $\zeta_p = \beta k_0 n(x) \cdot(p+\nabla
w^+_p)$. It suffices to link $\zeta_p$ with $w^+_p - w^-_p$ on
$\Phi(\Gamma_0)$. Hence fix an arbitrary $\phi$ in
$C^\infty(\Phi(\Gamma_0)) \cap \mathcal{H}$ such that
$\int_{\Phi(\Gamma_0)} \phi ds = 0$ and $\phi\circ \Phi$ is
stationary. Then we can construct $\varphi = \varphi^+
\chi_{\Phi(Y^+)} + \varphi^- \chi_{\Phi(Y^-)}$ in $\mathcal{H}$ by
solving:
\begin{equation*}
\left\{
\begin{aligned}
\Delta \varphi^+ = 0 \text{ in } \Phi(Y^+), &\quad\text{and}\quad& \Delta \varphi^- = 0 \text{ in } \Phi(Y^-),\\
\varphi^+ = 0 \text{ on } \partial \Phi(Y,) &\quad\text{and}\quad&
n(x) \cdot \nabla \varphi^{\pm} = \phi \text{ on } \Phi(\Gamma_0).
\end{aligned}
\right.
\end{equation*}
Using $\varphi$ as test function in \eqref{eq:A=b} and by
integration by parts, one obtains
\begin{equation*}
\E \int_{\Phi(Y^+)} \nabla w^+_{p,\alpha} \cdot \nabla
\overline{\varphi^+} dx + \E \int_{\Phi(Y^-)} \nabla
w^-_{p,\alpha} \cdot \nabla \overline{\varphi^-} dx = \E
\int_{\Phi(\Gamma_0)} (-w^+_{p,\alpha} + w^-_{p,\alpha})
\overline{\phi} ds.
\end{equation*}
Pass to the limit and recall that $w^+_{p,\alpha} -
w^-_{p,\alpha}$ converges to $\zeta_p$; we obtain
$$
\E \int_{\Phi(Y^+)} \nabla w^+_p \cdot \nabla \overline{\varphi^+}
dx + \E \int_{\Phi(Y^-)} \nabla w^-_p \cdot \nabla
\overline{\varphi^-} dx = -\E \int_{\Phi(\Gamma_0)} \zeta_p
\overline{\phi} ds.
$$
Since $\varphi$ is constructed so that $\Delta \varphi^\pm = 0$ in
$\Phi(Y^\pm)$, we have also that
$$
\E \int_{\Phi(Y^+)} \nabla w^+_p \cdot \nabla \overline{\varphi^+}
dx + \E \int_{\Phi(Y^-)} \nabla w^-_p \cdot \nabla
\overline{\varphi^-} dx = - \E \int_{\Phi(\Gamma_0)} (w^+_p -
w^-_p) \overline{\phi} ds.
$$
It follows that
\begin{equation*}
\E \int_{\Phi(\Gamma_0)} (w^+_p - w^-_p - \zeta_p) \overline{\phi}
ds = 0, \quad \forall \phi \in C^\infty(\Phi(\Gamma_0)) \text{
s.t. } \int_{\Phi(\Gamma_0)} \phi ds = 0.
\end{equation*}
Note that both $\zeta_p$ and $w^+_p - w^+_p$ are stationary. The
above identity shows that $\zeta_p = w^+_p - w^-_p + C_2(\gamma)$
where $C_2(\gamma)$ is a random constant. Re-define $w^-_p$ by
subtracting from it the constant $C_2$; then the second boundary
condition in the second line of \eqref{eq:auxiliary} is satisfied.
Note that, by subtracting the same constant from $\tilde{w}^-_p$, the
change of variable $w^-_p = \tilde{w}^-_p \circ \Psi$ remains
valid. We summarize that $(w^+_p, w^-_p)$ obtained above provides
a solution to the auxiliary problem \eqref{eq:auxiliary}.
\smallskip

{\it Step 4: Uniqueness of the auxiliary problem}. Suppose
otherwise, then there exist $v^+_0$ and $v^-_0$ satisfying
\eqref{eq:auxiliary} with $p=0$. In addition, there is an
extension of $\tilde{v}^+_0$ denoted by $\tilde{v}^{\rm ext}_0$,
such that
\begin{equation}
\nabla \tilde{v}^{\rm ext}_0 \text{ is stationary}, \quad
\text{and} \quad \E \int_Y \nabla \tilde{v}^{\rm ext}_0 dx = 0.
\end{equation}

On the one hand, by the standard elliptic regularity theory, we
know that $v^+_0$ and $v^-_0$ are in $W^{1,\infty}_{\rm
loc}(\Phi(\R_2^+))$ and $W^{1,\infty}_{\rm loc}(\R_2^-)$. Then
this is true also for $\tilde{v}^+_0$ and $\tilde{v}^-_0$.
Consequently, we have that: $\nabla \tilde{v}^{\rm ext}_0$ is
stationary; $\E \|\nabla \tilde{v}^{\rm ext}_0\|^s_{L^s(Y)} \le C$
for some $s > 2$. These properties of $\tilde{v}^{\rm ext}_0$
imply that it grows sub-linearly, thanks to \cite[Lemma
A.5]{armstrong}.

Let us take the weak formulation of the equations satisfied by $(v^+_0,
v^-_0)$, and take this function itself as the test function.
Integrate over $\Phi(NY)$ for a large integer $N$. We get
\begin{equation*}
\begin{aligned}
\int_{\Phi(NY\cap \R_2^+)} k_0 |\nabla v^+_0|^2 dx &+ \int_{\Phi(NY\cap \R_2^-)} k_0 |\nabla v^-_0|^2 dx\\
&+ \beta^{-1} \int_{\Phi(NY\cap \GGamma)} |v^+_0 - v^-_0|^2 ds =
\int_{\partial \Phi(NY)} k_0 n \cdot \nabla v^+_0 \overline{v^+_0}
ds.
\end{aligned}
\end{equation*}
Since $v^+_0$ grows sub-linearly at infinity, for sufficiently
large $N$, one has $|v^+_0| = o(N)$. Consequently, the right-hand
side is of order $o(N^2)$. Take the real part of the left-hand
side and divided it by $N^2$, we have
\begin{equation*}
\frac{1}{N^2} \sum_{n \in \mathcal{I}(N)} \left[\int_{\Phi(Y_n^+)}
\sigma_0 |\nabla v^+_0|^2 dx + \int_{\Phi(Y^-_n)} \sigma_0 |\nabla
v^-_0|^2 dx + \Re \beta^{-1} \int_{\Phi(\Gamma_n)} |v^+_0 -
v^-_0|^2 ds\right] \longrightarrow 0,
\end{equation*}
where $\mathcal{I}(N)$ are the indices of cubes $\{Y_n \subset
NY\}$. By a change of variable with bounds \eqref{eq:Phic2} and
\eqref{eq:Phic3}, we also have that
 \begin{equation*}
\frac{1}{N^2} \sum_{n \in \mathcal{I}(N)} \left[\int_{Y_n^+}
\sigma_0 |\nabla \tilde{v}^+_0|^2 d\tilde{x} + \int_{Y^-_n}
\sigma_0 |\nabla \tilde{v}^-_0|^2 d\tilde{x} + \beta^{-1}
\int_{\Gamma_n} |\tilde{v}^+_0 - \tilde{v}^-_0|^2(\tilde{x})
ds(\tilde{x})\right] \longrightarrow 0.
\end{equation*}
Since the integrands above are stationary and the number of
elements in $\mathcal{I}(N)$ is also $N^2$, we can apply ergodic
theorem and conclude that
\begin{equation*}
\E \int_{Y} |\nabla \tilde{v}^+_0|^2 d\tilde{x} = \E \int_{Y^-}
|\nabla \tilde{v}^-_0|^2 d\tilde{x} = \E \int_{\Gamma_0}
|\tilde{v}^+_0 - \tilde{v}^-_0|^2(\tilde{x}) ds(\tilde{x}) = 0.
\end{equation*}
This implies that $\tilde{v}^+_0 = \tilde{v}^-_0 = C$ for some
constant. By the change of variables, $v^+_0 = v^-_0 = C$ as well,
proving the uniqueness.
\end{proof}


\subsection{Proof of the homogenization theorem} \label{sec:proof}

In this section, we prove the homogenization theorem using the
energy method, {\it i.e.}, the method of oscillating test
functions \cite{murat}.

\subsubsection{Oscillating test functions}

We first build the oscillating test functions using the solutions
to the auxiliary problem. Fix a vector $p \in \R^2$. Let
$(w^+_p,w^-_p)$ be the unique solution to the auxiliary problem
\eqref{eq:auxiliary}. In particular, $w^+_p$ has an extension
$w^{\rm ext}_p$. We define
\begin{equation}
\label{eq:osctest} \left\{
\begin{aligned}
w^\eps_{1p}(x,\gamma) =  x\cdot p + \eps w^{\rm ext}_p (\frac{x}{\eps},\gamma), &\quad&& x \in \R^2,\\
w^\eps_{2p}(x,\gamma) =  x\cdot p + \eps Qw^-_p
(\frac{x}{\eps},\gamma), &\quad&& x \in \R^2.
\end{aligned}
\right.
\end{equation}
Here and in the sequel, $Q$ denotes the trivial extension operator
which sets $Qf=0$ outside the domain of $f$. By scaling the
auxiliary problem, we verify that $(w^{\eps+}_p, w^{\eps-}_p)$,
where $w^{\eps+}_p$ is the restriction of $w^\eps_{1p}$ in
$\eps\Phi(\R_2^+)$ and $w^{\eps-}_p$ is the restriction of
$w^\eps_{2p}$ in $\eps\Phi(\R_2^-)$, satisfies
\begin{equation*}
\left\{
\begin{aligned}
\nabla \cdot k_0 \nabla w^{\eps+}_p = 0 &\quad\text{and}\quad&
\nabla \cdot k_0 \nabla w^{\eps-}_p = 0 \text{ in } \eps\Phi(\R_2^\pm),\\
k_0 n \cdot \nabla w^{\eps+}_p = k_0 n \cdot \nabla w^{\eps-}_p
&\quad\text{and}\quad& w^{\eps+}_p - w^{\eps-}_p = \eps \beta k_0
n\cdot \nabla w^{\eps-}_p \text{ on } \eps\Phi(\GGamma).
\end{aligned}
\right.
\end{equation*}
In particular, for any bounded open set $\mathscr{O} \subset \R^2$
and any test function $\varphi = (\varphi^+,\varphi^-)$ such that
$\varphi^\pm \in H^1_{\rm loc}(\mathscr{O}\cap \Phi(\R_2^\pm))$
and $\varphi$ supported in $\mathscr{O}$, we have that
\begin{equation}
\begin{aligned}
\int_{\mathscr{O}\cap \eps\Phi(\R_2^+)} k_0 \nabla w^{\eps+}_p \cdot \overline{\nabla \varphi^+} dx  &+ \int_{\mathscr{O}\cap \eps\Phi(\R_2^-)} k_0 \nabla w^{\eps-}_p \cdot \overline{\nabla \varphi^-} dx\\
&+ (\eps\beta)^{-1} \int_{\mathscr{O}\cap \eps\Phi(\GGamma)}
(w^{\eps-}_p - w^{\eps-}_p) \overline{(\varphi^+- \varphi^-)} ds =
0.
\end{aligned}
\label{eq:wep12weak}
\end{equation}

Define also the vector fields $\eta^{\eps\pm}_p = k_0 \nabla
w^{\eps\pm}_p$. We derive the following convergence results.

\begin{lem} \label{lem:osctest} Let $w^{\eps \pm}_p$ and the vector fields $\eta^{\eps\pm}_p$ be defined as above and let $\mathscr{O}
 \subset \R^2$ be a bounded open set. Then as $\eps \to 0$, we have the following:
\begin{equation}
w^\eps_{1 p} \rightarrow x\cdot p, \quad \text{ uniformly in }
\mathscr{O} \text{ a.s.~in } \mathcal{O}; \label{eq:limweext}
\end{equation}
\begin{equation}
w^\eps_{2p} \rightarrow x\cdot p, \quad \text{ in }
L^2(\mathscr{O}) \text{ a.s.~in } \mathcal{O}. \label{eq:limwem}
\end{equation}
\begin{equation}
Q\eta^{\eps\pm}_p \rightharpoonup \det\left(\E \int_Y \nabla
\Phi(x,\cdot) dx\right)^{-1} \E \int_{\Phi(Y^\pm)} k_0
\left(\nabla w^\pm_p(x,\cdot) + p\right)dx \text{ in }
[L^2(\mathscr{O})]^2 \text{ a.s.~in } \mathcal{O}.
\label{eq:limgwe}
\end{equation}
\end{lem}

\begin{proof} The first limit holds since $w^{\rm ext}_p(x)$ grows sub-linearly as $|x|$ tends to infinity, a fact we have already proved in Step 4 in the proof of Theorem \ref{thm:auxiliary}. Indeed, we have
\begin{equation*}
w^\eps_{1 p} - x\cdot p = |x|\left(|\eps^{-1}x|^{-1} w^{\rm
ext}_p(\eps^{-1}x)\right) \to 0 \text{ uniformly in } \mathscr{O}.
\end{equation*}

To prove the second convergence result, we write
$$
w^\eps_{2p} - x\cdot p = \eps \left(w^-_p(\frac{x}{\eps}) - w^{\rm
ext}_p(\frac{x}{\eps}) \right) \chi_{\eps\Phi(\R_2^-)} + \eps
w^{\rm ext}_p(\frac{x}{\eps}) \chi_{\eps\Phi(\R_2^-)}.
$$
The second item on the right converges uniformly in $\mathscr{O}$
to zero. Therefore, it suffices to prove that $J_\eps : = \|\eps
w^-_p(\eps^{-1}x) - \eps w^{\rm ext}_p(\eps^{-1}
x)\|_{L^2(\eps\Phi(\R_2^-) \cap \mathscr{O})}$ converges to zero.
Given $\mathscr{O}$ and $\eps$, we can find
$\mathcal{I}_\eps(\mathscr{O}) \subset \Z^2$ such that
$\mathscr{O} \subset \cup_{k \in \mathcal{I}_\eps} \eps \Phi(Y_n)$
and $|\mathcal{I}_\eps| \lesssim C(\mathscr{O}) \eps^{-d}$. Then
\begin{equation*}
\begin{aligned}
J_\eps &\le \sum_{n \in \mathcal{I}_\eps} \int_{\eps\Phi(Y^-_n)} \eps^2
 \left|w^{\rm ext}_p(\frac{x}{\eps}) - w^-_p(\frac{x}{\eps})\right|^2 dx = \eps^{2+d} \sum_{n
  \in \mathcal{I}_\eps} \int_{\Phi(Y^-_n)} \left|w^{\rm ext}_p(x) - w^-_p(x)\right|^2 dx\\
&\le C\eps^{2+d} \sum_{n \in \mathcal{I}_\eps} \int_{Y^-_n}
\left|\tilde{w}^{\rm ext}_p(\tilde{x}) -
\tilde{w}^-_p(\tilde{x})\right|^2 d\tilde{x}.
\end{aligned}
\end{equation*}
In the last inequality, we used the change of variable $\tilde{x}
= \Phi^{-1}(x)$ and the bounds \eqref{eq:Phic2} and
\eqref{eq:Phic3}. Using the estimate \eqref{eq:Ymest}, we have
\begin{equation*}
J_\eps \le C\eps^2 \left[\frac{1}{|\mathcal{I}_\eps|} \sum_{n \in
\mathcal{I}_\eps} \left(\int_{\Gamma_n}
\left|\tilde{w}^+_p(\tilde{x}) - \tilde{w}^-_p(\tilde{x})\right|^2
ds(\tilde{x}) + \int_{Y^-_n} \left| \nabla \tilde{w}^{\rm
ext}_p(\tilde{x}) - \nabla \tilde{w}^-_p(\tilde{x})\right|^2
d\tilde{x} \right)\right].
\end{equation*}
Note that the integrands above are stationary and the item inside
the bracket is ready for applying ergodic theorem. This item
converges to
\begin{equation*}
\E \int_{\Gamma_0} |\tilde{w}^+_p - \tilde{w}^-_p|^2(\tilde{x})
ds(\tilde{x}) + \E \int_{Y} |\nabla \tilde{w}^{\rm ext}_p - \nabla
\tilde{w}^-_p|^2 d\tilde{x},
\end{equation*}
which is bounded for example by \eqref{eq:retabdd1} and
\eqref{eq:retabdd}. Consequently, $J_\eps \to 0$, proving
\eqref{eq:limwem}.

Now we prove the last convergence result. Let
$\tilde{\eta}^{\pm}_p = (k_0 [p + (\nabla \Phi)^{-1}\nabla
\tilde{w}^\pm_p])\chi_{\Phi(\R_2^\pm)}$; then they are stationary
and the relation $\eta^{\eps\pm}_p = \tilde{\eta}^{\pm}_p
\left(\Phi^{-1} (\frac{x}{\eps},\gamma) \right)$ holds. By Lemma
2.2. of \cite{BLBL06}, we obtain \eqref{eq:limgwe}.
\end{proof}

\subsubsection{Proof of the homogenization theorem}

In this subsection we prove the homogenization theorem using
Tartar's energy method. There are two main steps. In the first
step, we use the energy estimates to extract converging
subsequences. In the second step, we identify the limit as a
solution to a homogenized equation which has unique solution.
\smallskip

\begin{proof}[Proof of Theorem \ref{thm:homog}] {\itshape Step 1: Extraction of converging subsequences.}  Let $(\uepsp,\uepsm)$ be the solution to the
 heterogeneous problem \eqref{eq:u_{epsilon}}. In particular, $\uepsp$ has an extension $u^{\rm ext} \in H^1(\Omega)$. Let the vector fields $\xieps^\pm$ be $k_0\nabla \ueps^\pm$. Then the estimates \eqref{eq:BddExt} and \eqref{eq:GradientEst} show that
\begin{equation*}
\|\ueps^{\rm ext}\|_{H^1(\Omega)} + \|Q\xiepsp\|_{[L^2(\Omega)]^2}
+ \|Q\xiepsp\|_{[L^2(\Omega)]^2} \le C.
\end{equation*}
Consequently, there exists a subsequence and functions $u_0 \in
H^1(\Omega)$ and $\xi_1, \xi_2 \in [L^2(\Omega)]^2$, such that
\begin{equation}
\begin{aligned}
&\ueps^{\rm ext} \rightharpoonup u_0 \text{ weakly in } H^1(\Omega), \quad &&\ueps^{\rm ext} \rightarrow u_0 \text{ strongly in }
L^2(\Omega);\\
&Q\xiepsp \rightharpoonup \xi_1 \text{ weakly in }
[L^2(\Omega)]^2, \quad &&Q\xiepsm \rightharpoonup \xi_2 \text{
weakly in } [L^2(\Omega)]^2.
\end{aligned}
\label{eq:usubseq}
\end{equation}
In the proof of Proposition \ref{prop:extlim}, we also proved that
\begin{equation}
\ueps^{\rm ext} \chi_{\eps}^- - Qu^-_\eps \rightarrow 0 \text{
strongly in } L^2(\Omega). \label{eq:usubseq1}
\end{equation}

Now fix an arbitrary test function $\varphi \in
C^\infty_0(\Omega)$. Take $(\varphi \chi_{\eps}^+,
\varphi\chi_{\eps}^-)$ as a test function in \eqref{eq:rpdeweak}.
Then the interface term disappears and we get
\begin{equation*}
\int_{\Omega} k_0 (Q\xiepsp) \cdot \overline{\nabla \varphi} dx +
\int_{\Omega} k_0 (Q\xiepsm) \cdot \overline{\nabla \varphi} dx =
0.
\end{equation*}
Passing to the limit $\eps \to 0$ along the subsequence above, one
finds
\begin{equation}
\int_\Omega (\xi_1 + \xi_2) \cdot \overline{\nabla \varphi} dx =
0, \quad \forall \varphi \in C^\infty_0(\Omega). \label{eq:xilim}
\end{equation}
Therefore, the limiting vector field $\xi_1 + \xi_2$ satisfies
that
\begin{equation}
\nabla \cdot (\xi_1 + \xi_2) = 0, \quad \text{ in }
\mathcal{D}'(\Omega), \label{eq:divxiEq0}
\end{equation}
where $\mathcal{D}'(\Omega)$ denotes the space of tempered
distributions on $\Omega$. Now for any $\phi \in C^\infty(\partial
\Omega)$, we may lift it to a smooth function $\varphi \in
C^\infty(\overline{\Omega})$ such that $\varphi = \phi$ on
$\partial \Omega$. Take $(\varphi \chi_{\eps}^+,
\varphi\chi_{\eps}^-)$ as the test function in \eqref{eq:rpdeweak}
and pass to the limit; we get
\begin{equation*}
\int_\Omega (\xi_1 + \xi_2) \cdot \overline{\nabla \varphi} dx =
\int_{\partial \Omega} g \overline{\phi}\ ds.
\end{equation*}
Since $\xi_1 + \xi_2 \in L^2(\Omega)$ and $\nabla \cdot (\xi_1 +
\xi_2) \in L^2(\Omega)$, the trace $n \cdot (\xi_1 + \xi_2)$ on
the boundary $\partial \Omega$ is well defined. Applying the
divergence theorem and \eqref{eq:divxiEq0} we get
\begin{equation*}
\int_{\partial \Omega} n \cdot (\xi_1 + \xi_2) \overline{\phi}\ ds
= \int_{\partial \Omega} g \overline{\phi}\ ds, \quad \forall \phi
\in C^\infty(\overline{\Omega}).
\end{equation*}
This shows that, $n \cdot (\xi_1 + \xi_2) = g$ at $\partial
\Omega$. Further, since the trace of $Q\xiepsm$ is zero for all
$\eps$, the same argument above shows that $n \cdot \xi_2 = 0$ at
$\partial \Omega$. We hence get
\begin{equation}
n \cdot \xi_1 = g \quad \text{ at } \partial \Omega.
\label{eq:xibc}
\end{equation}

{\itshape Step 2: Weak convergence of $Q\uepsm$.} We can write
$Q\uepsm$ as $\ueps^{\rm ext} \chi_{\eps}^- + (Q\uepsm -
\ueps^{\rm ext} \chi_{\eps}^-)$. Due to \eqref{eq:usubseq1} and
the fact that $\ueps^{\rm ext}$ converges strongly to $u_0$, we
only need to verify that $\chi_{\eps}^-$ converges weakly to
$\theta$. To this purpose, fix an arbitrary open set $K$ compactly
supported in $\Omega$, and observe that for sufficiently small
$\eps$, $K$ is compactly supported in $E_\eps$ defined in
\eqref{eq:KEdef}. Then we have
\begin{equation*}
\int_K \chi_{\Omega_\eps^-} dx = \int_{K\cap \eps \Phi(\R_2^-)} dx
= \int_{\eps \Phi^{-1} \left(\frac{K}{\eps}\right)} \chi_{\R_2^-}
(\frac{z}{\eps}) \det \nabla \Phi(\frac{z}{\eps},\gamma) dz.
\end{equation*}
In \cite{BLBL06,BLBL07}, it is shown that the characteristic
function $\eps \Phi^{-1}\left(\frac{K}{\eps}\right)$ converges
strongly in $L^1(\R^2)$ to that of the set $[\E \int_Y \nabla
\Phi(y,\cdot) dy]^{-1}K$. On the other hand, since the function
$\chi_{\R_2^-} {\rm det} \nabla \Phi$ is stationary, by ergodic
theorem, we have
\begin{equation*}
\chi_{\R_2^-}(\frac{z}{\eps}) \det \nabla
\Phi(\frac{z}{\eps},\gamma) \overset{*}{\rightharpoonup}
\E\int_{Y} \chi_{\R_2^-} \det \nabla \Phi(z,\gamma) dz = \theta
\varrho^{-1}, \quad \text{in } L^\infty(\R^2).
\end{equation*}
Consequently, we observe that for any open set $K$ compactly
supported in $\Omega$, we have
\begin{equation*}
\int \chi_{K} \chi_{\Omega_\eps^-} dx \rightarrow \theta
\varrho^{-1} \int_{[\E \int_Y \nabla \Phi(y,\cdot) dy]^{-1}K} dx =
\theta \varrho^{-1} \det \left(\E \int_Y \nabla \Phi(y,\cdot)
dy\right)^{-1} |K| = \theta |K|.
\end{equation*}
Here, we used the fact that $\det \left(\E \int_Y \nabla
\Phi(y,\cdot) dy\right) = \varrho^{-1}$, a fact also proved in
\cite{BLBL06,BLBL07}. Since linear combinations of characteristic
functions of compact sets in $\Omega$ are dense in $L^2(\Omega)$,
we get the desired result. The fact that $\theta < 1$ is easily
deduced from the assumption on $Y^-$ and the assumption
\eqref{eq:Phic4}, we omit the proof. This completes the proof of
item two of the theorem up to a subsequence.
\smallskip

{\itshape Step 3: Identifying the limit.} Fix an arbitrary test
function $\varphi \in C^\infty_0(\Omega)$. By the constructions of
$\Depsm, K_\eps$ and $E_\eps$ defined in \eqref{eq:Depsmdef} and
\eqref{eq:KEdef}, for sufficiently small $\eps$, the function
$\varphi$ is compactly supported in $E_\eps$.

Fix a $p \in \R^2$. Let $w^\eps_{1p}$ and $w^\eps_{2p}$ be as in
\eqref{eq:osctest}. In the weak formulation \eqref{eq:wep12weak}
of the equations satisfied by them, take $(\varphi
\overline{\uepsp}, \varphi \overline{\uepsm})$ as a test function;
we get
\begin{equation*}
\int_\Omega (Q\eta^{\eps+}_p) \cdot \nabla(\overline{\varphi}
\uepsp) dx + \int_\Omega (Q\eta^{\eps-}_p) \cdot
\nabla(\overline{\varphi} \uepsm) dx + \frac{1}{\eps\beta}
\int_{\interface} (w^\eps_{1p} - w^\eps_{2p})
\overline{\varphi}(\uepsp - \uepsm) ds = 0.
\end{equation*}
Similarly, in the weak formulation \eqref{eq:rpdeweak}, take
$(\varphi \overline{w^\eps_{1p}}, \varphi \overline{w^\eps_{2p}})$
as the test function; we get
\begin{equation*}
\int_\Omega (Q\xiepsp) \cdot \nabla(\overline{\varphi}
w^\eps_{1p}) dx + \int_\Omega (Q\xiepsm) \cdot
\nabla(\overline{\varphi} w^\eps_{2p}) dx + \frac{1}{\eps\beta}
\int_{\interface} (\uepsp - \uepsm)\overline{\varphi}(w^\eps_{1p}
- w^\eps_{2p}) ds = 0.
\end{equation*}
Note that the integrating domains in the first formula can be
taken as above because $\varphi$ is compactly supported in
$E_\eps$, which implies that $\eps\Phi(\GGamma) \cap \text{supp }
\varphi = \interface \cap \text{supp } \varphi$. Subtracting the
two formulas above and noticing in particular that the interface
terms cancel out, we get
\begin{equation*}
\begin{aligned}
\int_\Omega (Q\eta^{\eps+}_p) \cdot \overline{\nabla \varphi} \ueps^{\rm ext} dx + &\int_\Omega (Q\eta^{\eps-}_p) \cdot \overline{\nabla \varphi} \ueps^{\rm ext} dx
+ \int_\Omega (Q\eta^{\eps-}_p) \cdot \overline{\nabla \varphi} (Q\ueps^- - \ueps^{\rm ext}\chi_{\eps}^-) dx\\
- &\int_\Omega (Q\xiepsp) \cdot \overline{\nabla \varphi}
w^\eps_{1p} dx - \int_\Omega (Q\xiepsm) \cdot \overline{\nabla
\varphi} w^\eps_{2p} dx =0.
\end{aligned}
\end{equation*}
By the convergence results \eqref{eq:limgwe}, \eqref{eq:limweext},
\eqref{eq:limwem}, \eqref{eq:usubseq} and \eqref{eq:usubseq1}, we
observe that each integrand above is a product of a strong
converging term with a weak converging term. Therefore, we can
pass the above to the limit $\eps \to 0$ and get
\begin{equation*}
\int_\Omega (\eta_{1p} + \eta_{2p}) u_0 \cdot \overline{ \nabla
\varphi} dx = \int_\Omega (\xi_1 + \xi_2) (x\cdot p) \cdot
\overline{\nabla \varphi} dx,
\end{equation*}
where $\eta_{1p}$ (resp. $\eta_{2p}$) is defined as the right-hand
side of \eqref{eq:limgwe} with the "$+$" (resp. "$-$") sign. The
integral on the right can be written as
\begin{equation*}
\int_\Omega (\xi_1 + \xi_2) \cdot [\overline{\nabla (\varphi
x\cdot p )} - p \overline{\varphi}] dx = -\int_\Omega (\xi_1 +
\xi_2) \cdot p\ \overline{\varphi} dx,
\end{equation*}
where we have used \eqref{eq:xilim}. For the integral involving
$\eta_{1p} + \eta_{2p}$, we check that the derivation of
\eqref{eq:divxiEq0} works for $\eta_{1p}+\eta_{2p}$, which shows
that the integral can be written as
\begin{equation*}
-\int_\Omega [\nabla \cdot (\eta_{1p} + \eta_{2p}) u_0 +
(\eta_{1p} + \eta_{2p}) \nabla u_0]\overline{\varphi} dx =
-\int_\Omega [(\eta_{1p} + \eta_{2p}) \nabla u_0]
\overline{\varphi} dx.
\end{equation*}
Due to the linearity in $p$ of the auxiliary problem
\eqref{eq:auxiliary} in $p$, we verify that its solutions satisfy
that $w^\pm_p = \sum_{j=1}^2 w^\pm_{e_j} p_j$, where
$\{e_j\}_{j=1}^2$ is the Euclidean basis of $\R^2$. Consequently,
we have that
\begin{equation*}
\begin{aligned}
&e_i \cdot (\eta_{1p} + \eta_{2p}) = k_0 \det\left(\E \int_Y \nabla \Phi(x,\cdot) dx\right)^{-1} \\
&\quad\quad \times \sum_{j=1}^2 \left[\E \int_{\Phi(Y^+)}
\left(e_i \cdot \nabla w^+_{e_j} (x,\cdot) + \delta_{ij} \right)dx
+ \E \int_{\Phi(Y^-)} \left(e_i \cdot \nabla w^-_{e_j} (x,\cdot) +
\delta_{ij}\right) dx \right] p_j.
\end{aligned}
\end{equation*}
Recall the definition of the matrix $(K^*_{ij})$ in
\eqref{eq:Khomog}. We check that $\eta_{1p}+\eta_{2p} = (K^*)^t p$
where $(K^*)^t$ denotes the transpose of $K^*$. Combining the
above formulas, we finally obtain that
\begin{equation*}
\int_\Omega p^t (\xi_1 + \xi_2 - K^*\nabla u_0) \overline{\varphi}
dx = 0, \quad \forall \varphi \in C^\infty_0(\Omega), \ p \in
\R^2.
\end{equation*}
We hence conclude that $\xi_1 + \xi_2 = K^* \nabla u_0$, and by
\eqref{eq:divxiEq0} we verify the first line of \eqref{eq:hpde}.
Similarly, by \eqref{eq:xibc}, we verify the boundary condition in
\eqref{eq:hpde}. Finally, the facts that
\begin{equation*}
\int_\Omega Q\uepsp dx = 0, \quad\text{and}\quad Q\uepsm
\rightharpoonup \theta u_0 \text{ weakly in } L^2
\end{equation*}
indicate that
\begin{equation*}
\int_\Omega u_0 dx = \lim_{\eps \to 0} \int_\Omega (Q\uepsp +
Q\uepsm) dx = \lim_{\eps \to 0} \int_\Omega Q\uepsm dx = \theta
\int_\Omega u_0 dx.
\end{equation*}
Since $\theta < 1$, we see that the integral condition in
\eqref{eq:auxiliary} is satisfied. Therefore, the subsequence
obtained in step one satisfies the homogenized problem
\eqref{eq:hpde}. Since this problem has a unique solution, all
converging subsequences converge to the unique solution. Finally,
the whole sequence $(\uepsp,\uepsm)$ converges, completing the
proof.
\end{proof}

\subsection{Effective admittivity of a dilute suspension}
\label{sec:dilute}

In this subsection, we consider the case when the cells are
dilute. We aim to derive a formal first order asymptotic expansion
of the effective admittivity in terms of the volume fraction of
the dilute cells.

In the formula of the homogenized coefficient \eqref{eq:Khomog},
the integral term has the form
\begin{equation*}
J_{ij} = \E\int_{\Phi(Y^+)} e_j \cdot \nabla w^+_{e_i}(y,\cdot) \
dy + \E\int_{\Phi(Y^-)} e_j \cdot \nabla w^-_{e_i}(y,\cdot) \ dy.
\end{equation*}
Thanks to the ergodic theorem, $J_{ij}$ also takes the form
\begin{equation*}
J_{ij} = \lim_{N\to \infty} \frac{1}{N^2} \sum_{n\in
\mathcal{I}(N)} \left( \int_{\Phi(Y_n^+)} e_j \cdot \nabla
w^+_{e_i}(y,\cdot)\ dy + \int_{\Phi(Y_n^-)} e_j \cdot \nabla
w^-_{e_i}(y,\cdot)\ dy\right).
\end{equation*}
Here, $\mathcal{I}(N)$ is the indices for the cubes $\{Y_n\}$
inside the big cube $NY$. Now using integration by parts, we
simplify the above expression to
\begin{equation*}
J_{ij} = \lim_{N\to \infty} \frac{1}{N^2} \left( \int_{\partial
\Phi(NY)} n_j w^+_{e_i}(y,\cdot)\ ds(y) - \sum_{n\in
\mathcal{I}(N)} \int_{\Gamma_n} (w^+_{e_i} - w^-_{e_i})(y,\cdot)
n_j \ ds(y) \right).
\end{equation*}
Here, $n$ denotes the outer normal vector along the boundary of
$\Phi(NY)$ and $\Phi(Y_n^-)$, $n\in \mathcal{I}(N)$; $n_j = n\cdot
e_j$ denotes its $j$-th component. Note that the boundary terms at
$\{\partial \Phi(Y_n)\}_{n\in \mathcal{I}(N)} \cap \Phi(NY)$ are
cancelled because two adjacent cubes share the same outer normal
vector at their common boundary except for reversed signs.

Finally, we have seen that $w^+_{e_i}$ has sub-linear growth.
Since the surface $\Phi(NY)$ has volume of order $O(N)$, the
sub-linear growth indicates that the boundary integral at
$\partial \Phi(NY)$ is of order $o(N^2)$. Consequently, when
divided by $N^2$ this term goes to zero. By applying the ergodic
theorem again, we obtain that
\begin{equation}
J_{ij} = \lim_{N\to \infty} \frac{1}{N^2} \sum_{n\in
\mathcal{I}(N)} \int_{\Gamma_n} (w^-_{e_i} - w^+_{e_i})(y,\cdot)
n_j \ ds(y) = \E \int_{\partial \Phi(Y^-)} (w^-_{e_i} -
w^+_{e_i})(y,\cdot) n_j \ ds(y). \label{eq:Jijlim}
\end{equation}
In the next subsection, we investigate this integral further by
deriving a formal representation for the jump $w^+_{e_i} -
w^-_{e_i}$ in the case when the inclusions are dilute, {\it i.e.},
small and far away from each other.

To model the dilute suspension, we assume that the reference cell
$Y^-$ is of the form of $\rho B$, where $B$ is a domain of unit
length scale, and $\rho \ll 1$ denotes the small length scale of
the dilute inclusions. Due to the assumptions \eqref{eq:Phic2} and
\eqref{eq:Phic3}, the length scale of the cell $\Phi(Y^-)$ is
still of order $\rho$. Further, due to the assumption
\eqref{eq:Phic4}, the distance of the cell $\Phi(Y^-)$ from the
``boundary'' $\partial \Phi(Y)$ is of order one, which is much
larger than the size of the inclusion.

Since the distances between the inclusions are much larger than
their sizes, we may use the single inclusion approximation. That
is, $w^\pm_{e_i}$ can be approximated by the solutions to the
following interface problem:
\begin{equation*}
\left\{
\begin{aligned}
&\nabla \cdot k_0 \nabla w^\pm_{e_i} = 0 \text{ in } \Phi(Y^-) \text{ and } \R^2 \setminus \Phi(Y^-),\\
&\frac{\partial w^+_{e_i}}{\partial n} = \frac{\partial
w^-_{e_i}}{\partial n}, \text{ and } w^+_{e_i} - w^-_{e_i} = \rho
\beta k_0 (\frac{\partial w^-_{e_i}}{\partial n} + n \cdot e_i)
\text{ on } \Upsilon,\\
& w^+_{e_i} \rightarrow 0 \text{ at }  \infty.
\end{aligned}
\right.
\end{equation*}
Here, $\Upsilon$ denotes the boundary of the inclusion. Note that
the extra $\rho$ in the jump condition is due to the fact that the
length scale of the inclusion $\Phi(Y^-)$ is of order $\rho$.
Using double layer potentials, we represent $w^+_{e_i}$ and
$w^-_{e_i}$ as $\mathcal{D}_{\Upsilon}[\phi_i]$ restricted to
$\Phi(Y^-)$ and $\R^2\setminus \Phi(Y^-)$ respectively. Due to the
trace formula of $\mathcal{D}_{\Upsilon}$ and the jump conditions
above, the function $\phi_i$ is determined by
\begin{equation}
-\phi_i = \rho\beta k_0 (\frac{\partial \mathcal{D}_{\Upsilon}
[\phi_i]}{\partial n} + n_i). \label{eq:phiidef}
\end{equation}
Let us define the operator $\mathcal{L}_{\Upsilon}$ by
$\frac{\partial \mathcal{D}_{\Upsilon}}{\partial n}$, then we have
that
\begin{equation*}
w^+_{e_i} - w^-_{e_i} = - \phi_i = \rho\beta k_0 (I+\rho\beta k_0
\mathcal{L}_{\Upsilon})^{-1} [n_i], \quad \text{ on } \Upsilon.
\end{equation*}
As a consequence, we have also that
\begin{equation*}
J_{ij} \simeq - \rho \beta k_0 \E \int_{\Upsilon} (I+\rho\beta k_0
\mathcal{L}_{\Upsilon})^{-1} [n_i] n_j ds.
\end{equation*}
Let us define $\psi_i$ to be $-(I+\rho\beta k_0
\mathcal{L}_{\Upsilon})^{-1} [n_i]$, that is $\psi_i + \rho \beta
k_0 n \cdot \nabla \mathcal{D}_{\Upsilon}[\psi_i](x) = -n_i$.
Define the scaled function $\tilde{\psi}_i(\tilde{x}) =
\psi_i(\rho \tilde{x})$ on the scaled curve $\rho^{-1} \Upsilon$.
Using the homogeneity of the gradient of the Newtonian potential,
we verify that
\begin{equation*}
\mathcal{D}_{\Upsilon}[\psi_i](x) = \mathcal{D}_{\rho^{-1}
\Upsilon} [\tilde{\psi}_i] (\tilde{x}), \quad\text{and}\quad \rho
n \cdot \nabla \mathcal{D}_{\Upsilon}[\psi_i](x) = n \cdot \nabla
\mathcal{D}_{\rho^{-1} \Upsilon}[\tilde{\psi}_i](\tilde{x}),
\end{equation*}
where  $\tilde{x} = \rho^{-1} x$. This shows that $\tilde{\psi}_i
= -(I+\beta k_0 \mathcal{L}_{ \rho^{-1}\Upsilon})^{-1}[n_i]$.
Using the change of variable $y \to \rho \tilde{y}$ in the
previous integral representation of $J_{ij}$, we rewrite it as
\begin{equation*}
J_{ij} \simeq \rho\beta k_0 \E \int_{\rho^{-1}\Upsilon}
\psi_i(\rho \tilde{y}) n_j ds(\rho \tilde{y}) = \rho^2 \beta k_0
\E \int_{\rho^{-1} \Phi(\Gamma)} \tilde{\psi}_i n_j ds(\tilde{y}).
\end{equation*}
Finally, the approximation \eqref{eq:Kdilute} of the effective
permittivity for the dilute suspension holds, where $f = \varrho
\rho^2$ is the volume fraction where $\varrho$ accounts for the
averaged change of volume due to the random diffeomorphism; the
polarization matrix $M$ is defined by \eqref{eq:Mijdef} and is
associated to the deformed inclusion scaled to the unit length
scale. Note that the imaging approach developed in subsection
\ref{handside} can be applied here as well.

\section{Numerical simulations} \label{sect:numer}

We present in this section some numerical simulations to
illustrate the fact that the Debye relaxation times are
characteristics of the microstructure of the tissue.

We use for the different parameters the following realistic
values:
\begin{itemize}
\item the typical size of eukaryotes cells: $\rho \simeq10-100$
$\mu$m; \item the ratio between the membrane thickness and the
size of the cell: $\delta / \rho = 0.7 \cdot 10^{-3}$; \item the
conductivity of the medium and the cell: $\sigma_0 =0.5$
$\textrm{S.m}^{-1}$; \item the membrane conductivity: $\sigma_m =
10^{-8}$ $\textrm{S.m}^{-1}$; \item the permittivity of the medium
and the cell: $\epsilon_0 = 90 \times8.85\cdot10^{-12}$
$\textrm{F.m}^{-1}$; \item the membrane permittivity:
$\varepsilon_m =3.5 \times8.85\cdot10^{-12}$ $\textrm{F.m}^{-1}$;
\item the frequency: $\omega \in [10^4, 10^9]$ Hz.
\end{itemize}

Note that the assumptions of our model $\delta \ll \rho$ and
$\sigma_m \ll \sigma_0$ are verified.

We first want to retrieve the invariant properties of the Debye
relaxation times. We consider (Figure \ref{fig1}) an elliptic cell
(in green) that we translate (to obtain the red one), rotate (to
obtain the purple one) and scale (to obtain the dark blue one). We
compute the membrane polarization tensor, its imaginary part, and
the associated eigenvalues which are plotted as a function of the
frequency (Figure \ref{fig2}). The frequency is here represented
on a logarithmic scale. One can see that for the two eigenvalues
the maximum of the curves occurs at the same frequency, and hence
that the Debye relaxation times are identical for the four
elliptic cells. Note that the red and green curves are even
superposed; this comes from the fact that $M$ is invariant by
translation.

Next, we are interested in the effect of the shape of the cell on
the Debye relaxation times. We consider for this purpose, (Figure
\ref{fig3}) a circular cell (in green),  an elliptic cell (in red)
and a very elongated elliptic cell (in blue). We compute
similarly as in the preceding case, the polarization tensors
associated to the three cells, take their imaginary part and plot
the two eigenvalues of these imaginary parts with respect to the
frequency. As shown in Figure \ref{fig4}, the maxima occur at
different frequencies for the first and second eigenvalues. Hence,
we can distinguish with the Debye relaxation times between these
three shapes.

Finally, we study groups of one (in green), two (in blue) and
three cells (in red) in the unit period (Figure \ref{fig5}) and
the corresponding polarization tensors for the homogenized media.
The associated relaxation times are different in the three
configurations (Figure \ref{fig6}) and hence can be used to
differentiate tissues with different cell density or organization.

These simulations prove that the Debye relaxation times are
characteristics of the shape and organization of the cells. For a
given tissue, the idea is to obtain by spectroscopy the frequency
dependence spectrum of its effective admittivity. One then has
access to the membrane polarization tensor and the spectra of the
eigenvalues of its imaginary part. One compares the associated
Debye relaxation times to the known ones of healthy and cancerous
tissues at different levels. Then one would be able to know using
statical tools with which probability the imaged tissue is
cancerous and at which level.



\begin{figure}
\centering
\includegraphics[scale=0.4]{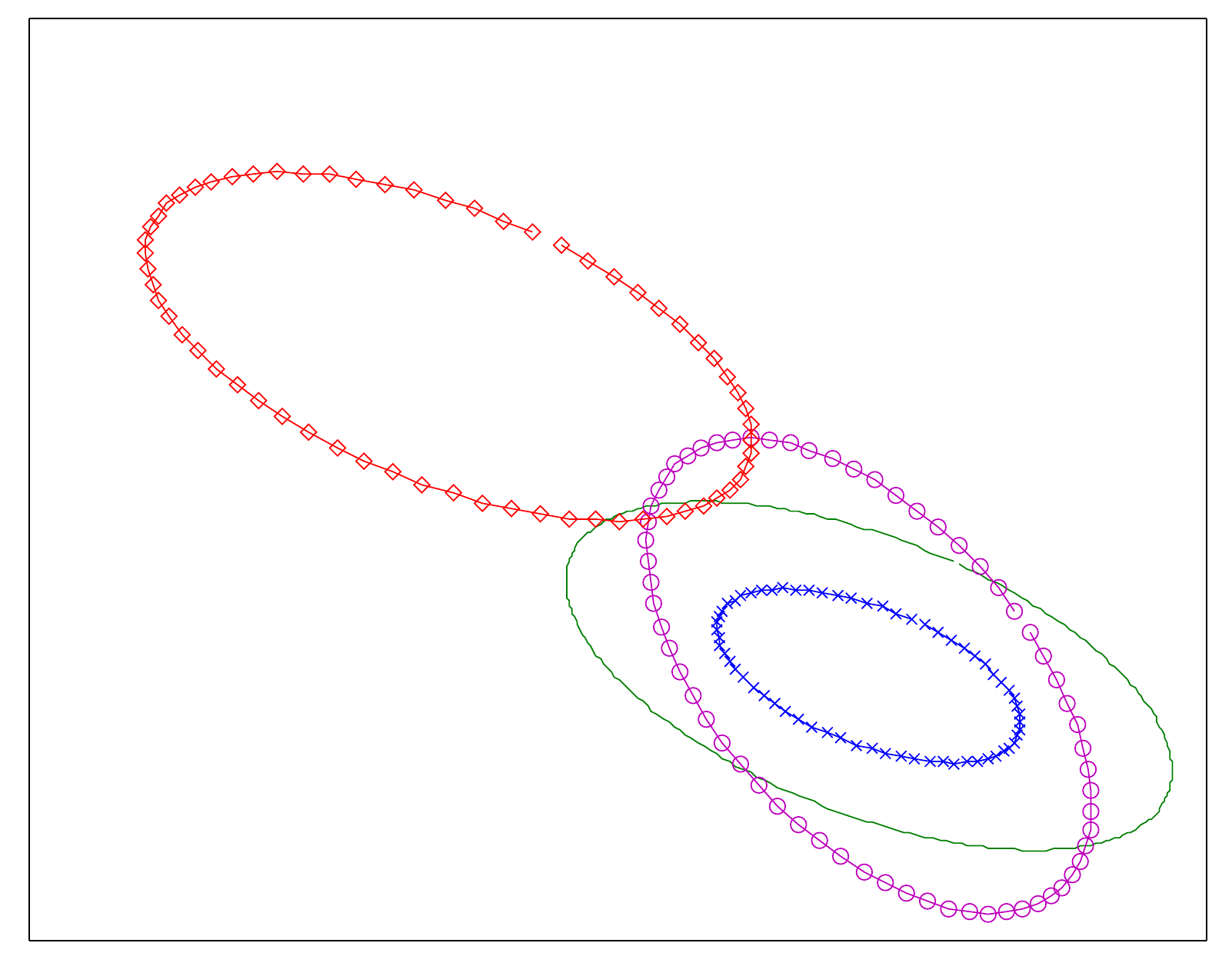}
 \caption{\it{An ellipse translated, rotated and scaled.} \label{fig1}}
\end{figure}

\begin{figure}
\centering
\includegraphics[scale=0.43]{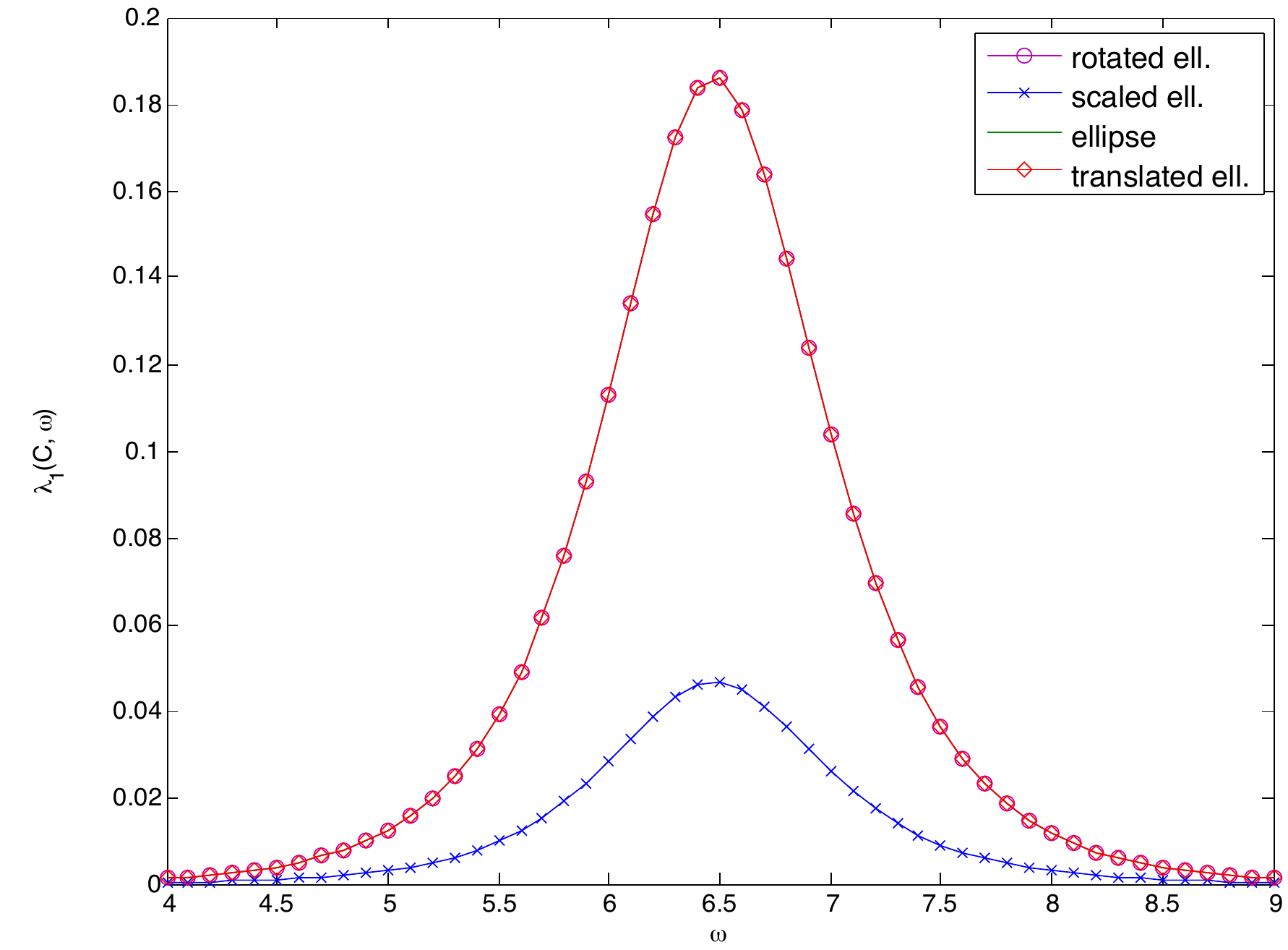}
\includegraphics[scale=0.452]{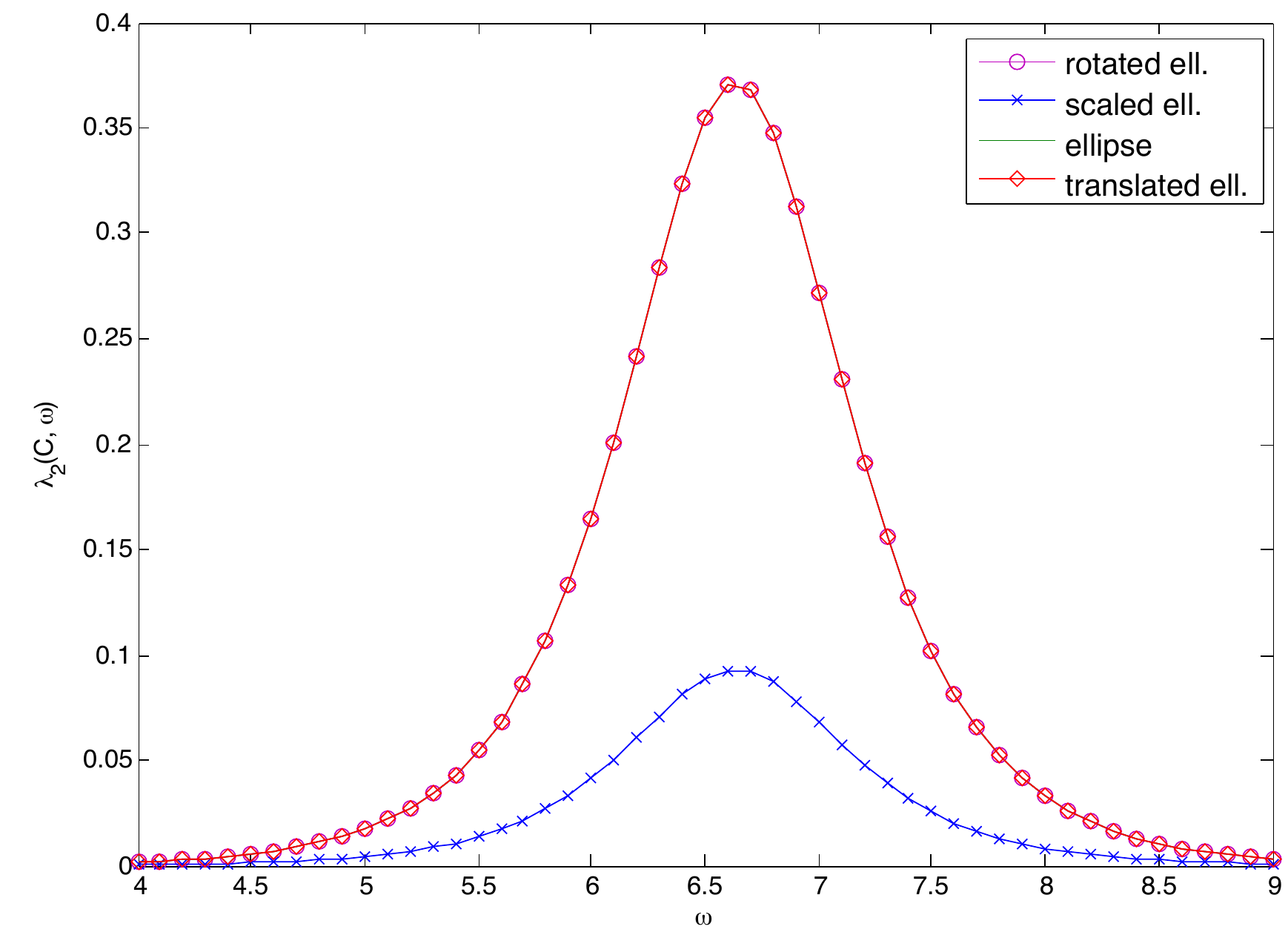}
 \caption{\it{Frequency dependence of the eigenvalues of $\Im M$ for the $4$ ellipses in Figure~\ref{fig1}.}\label{fig2}}
\end{figure}

\begin{figure}
\centering
\includegraphics[scale=0.4]{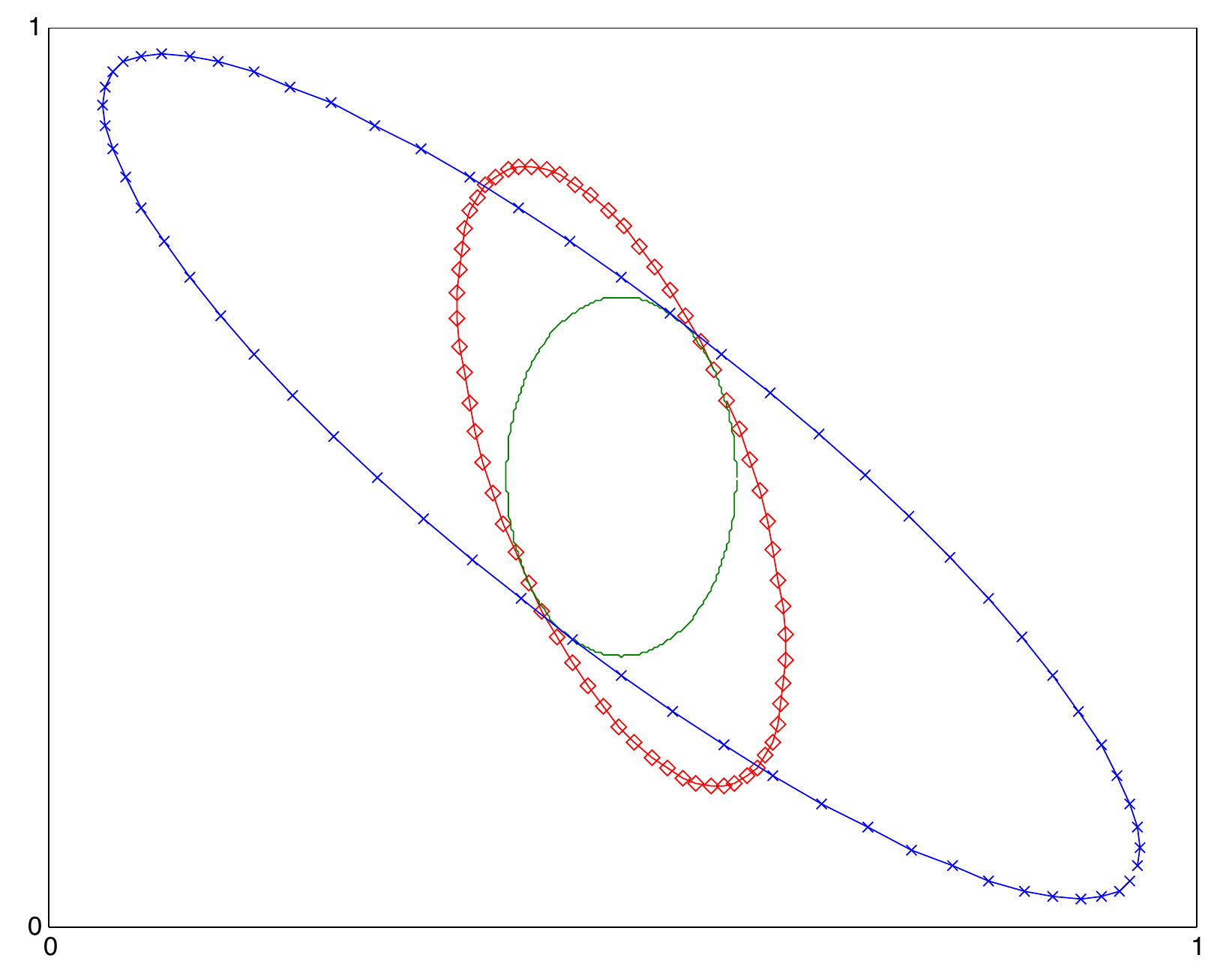}
 \caption{\it{A circle, an ellipse and a very elongated ellipse}.\label{fig3}}
\end{figure}

\begin{figure}
\centering
\includegraphics[scale=0.41]{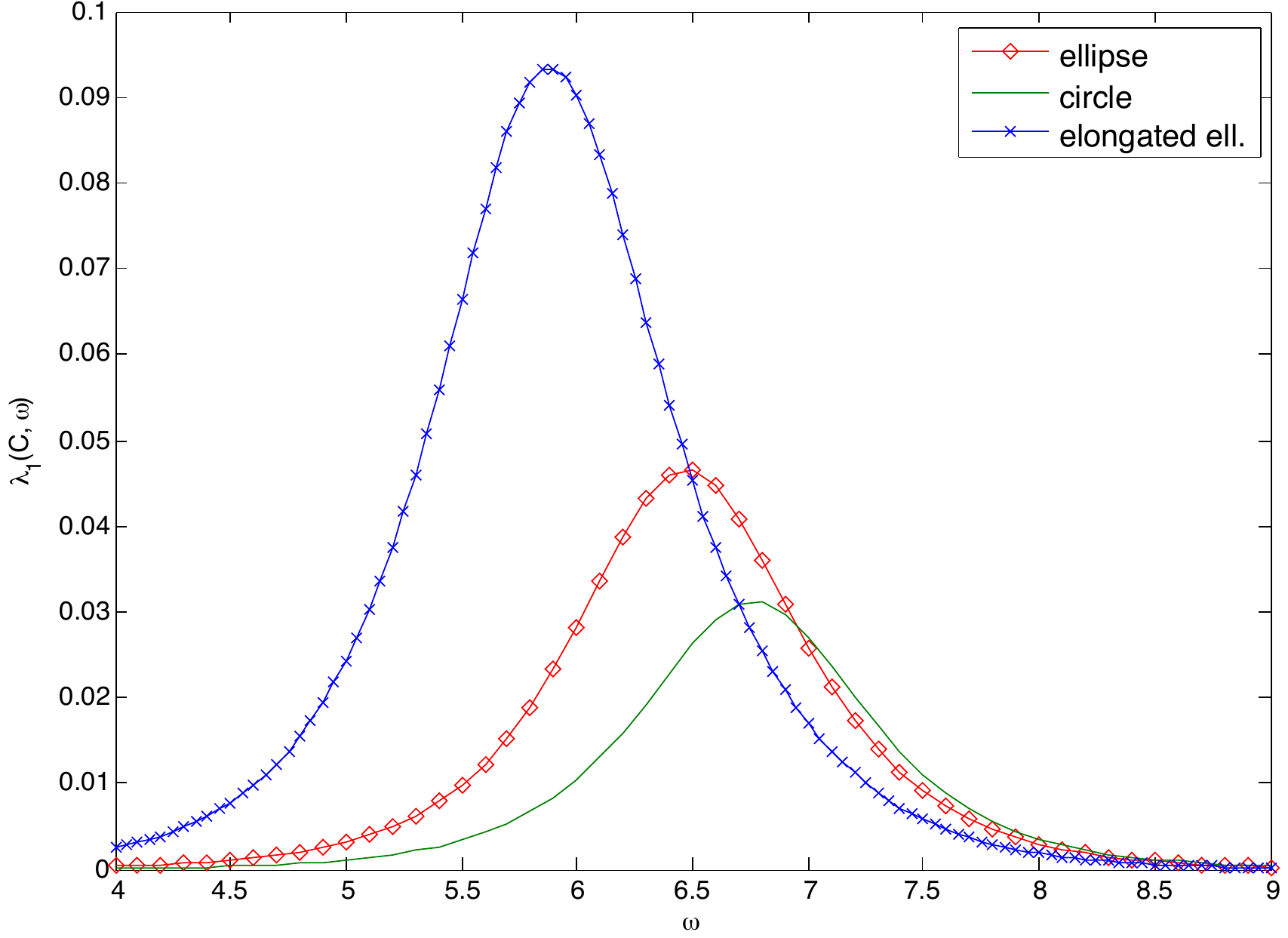}
\includegraphics[scale=0.41]{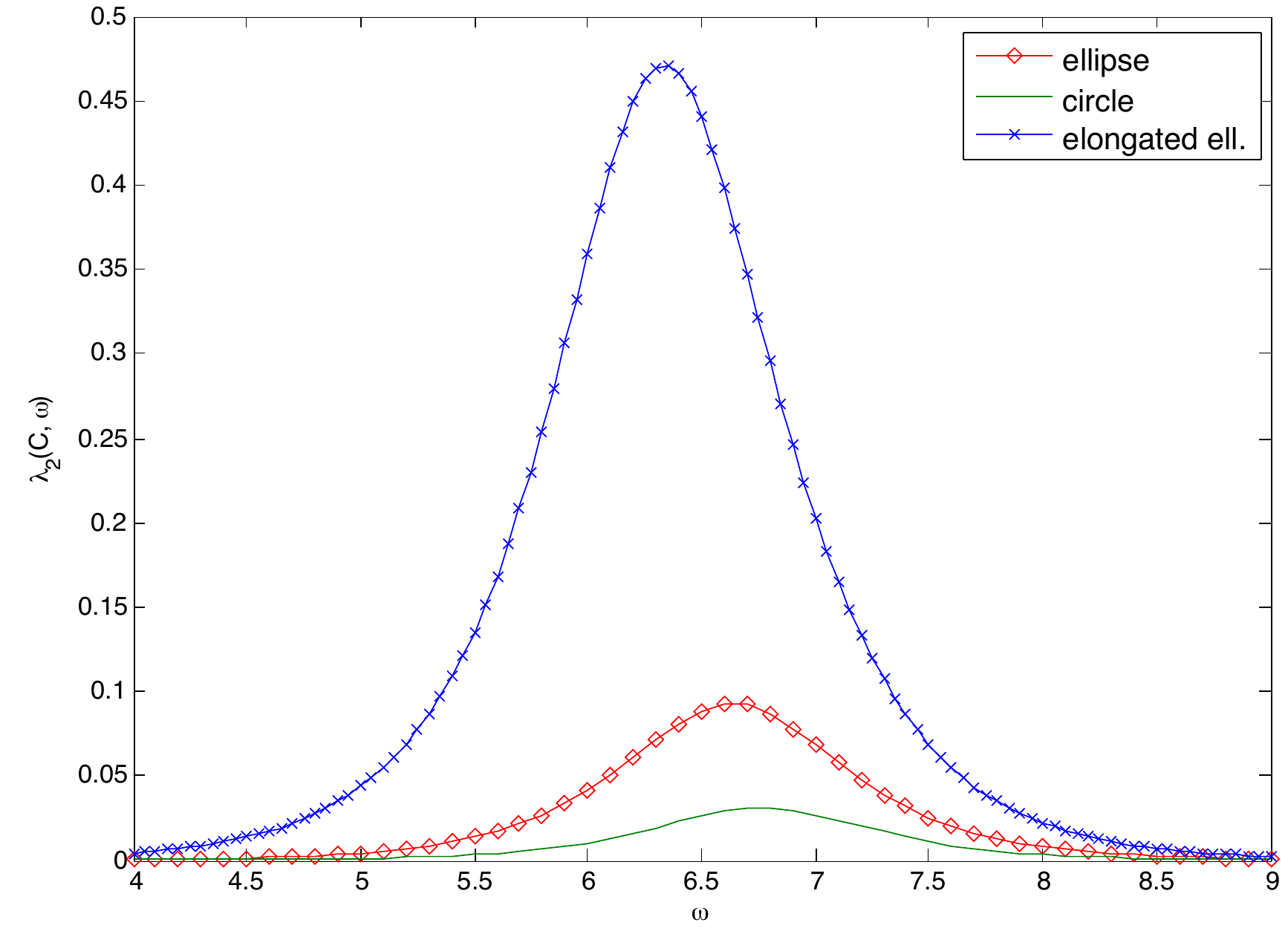}
 \caption{\it{Frequency dependence of the eigenvalues of $\Im M$ for the $3$ different cell shapes in Figure~\ref{fig3}.}\label{fig4}}
\end{figure}

\begin{figure}
\centering
\includegraphics[scale=0.283]{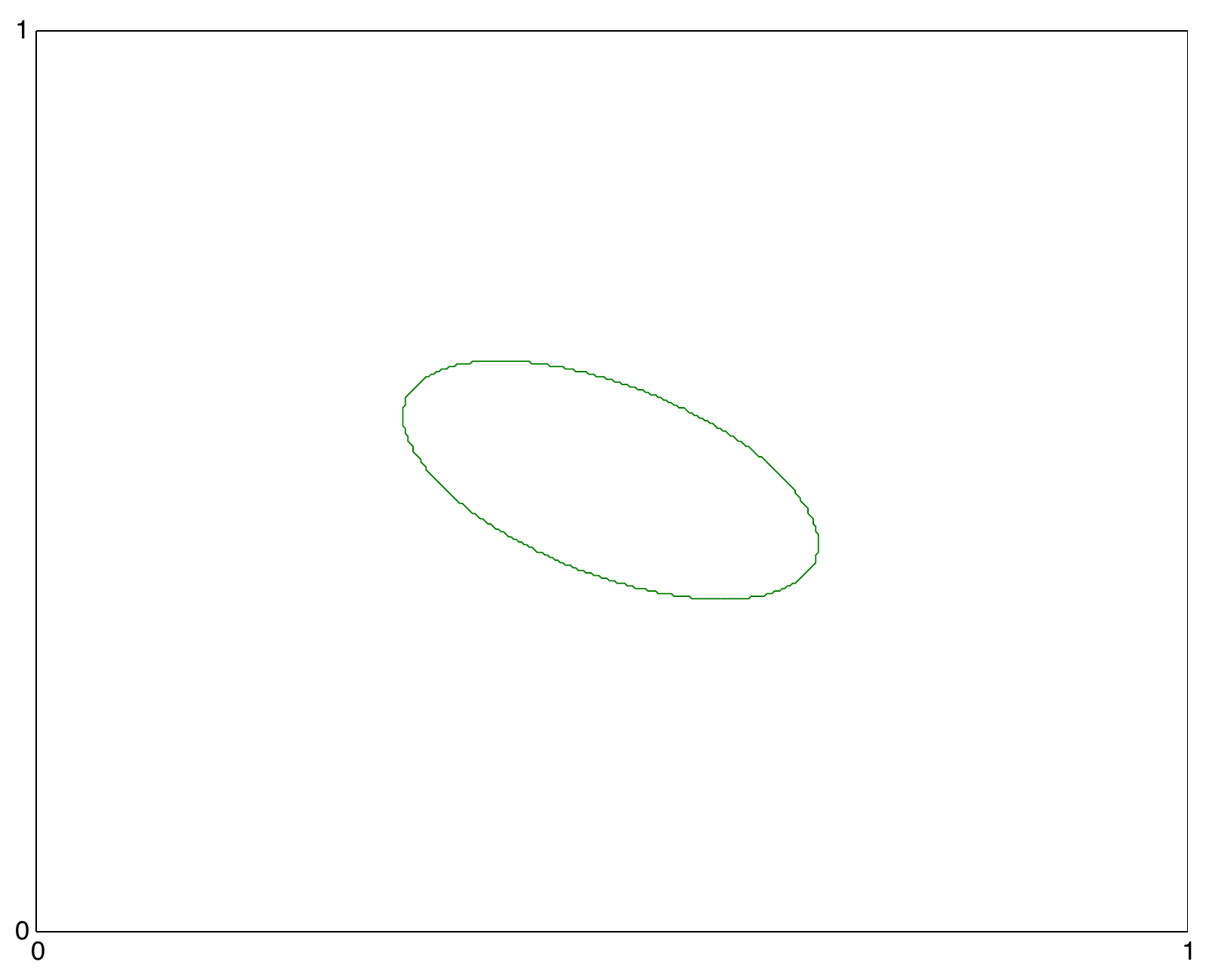}
\includegraphics[scale=0.3]{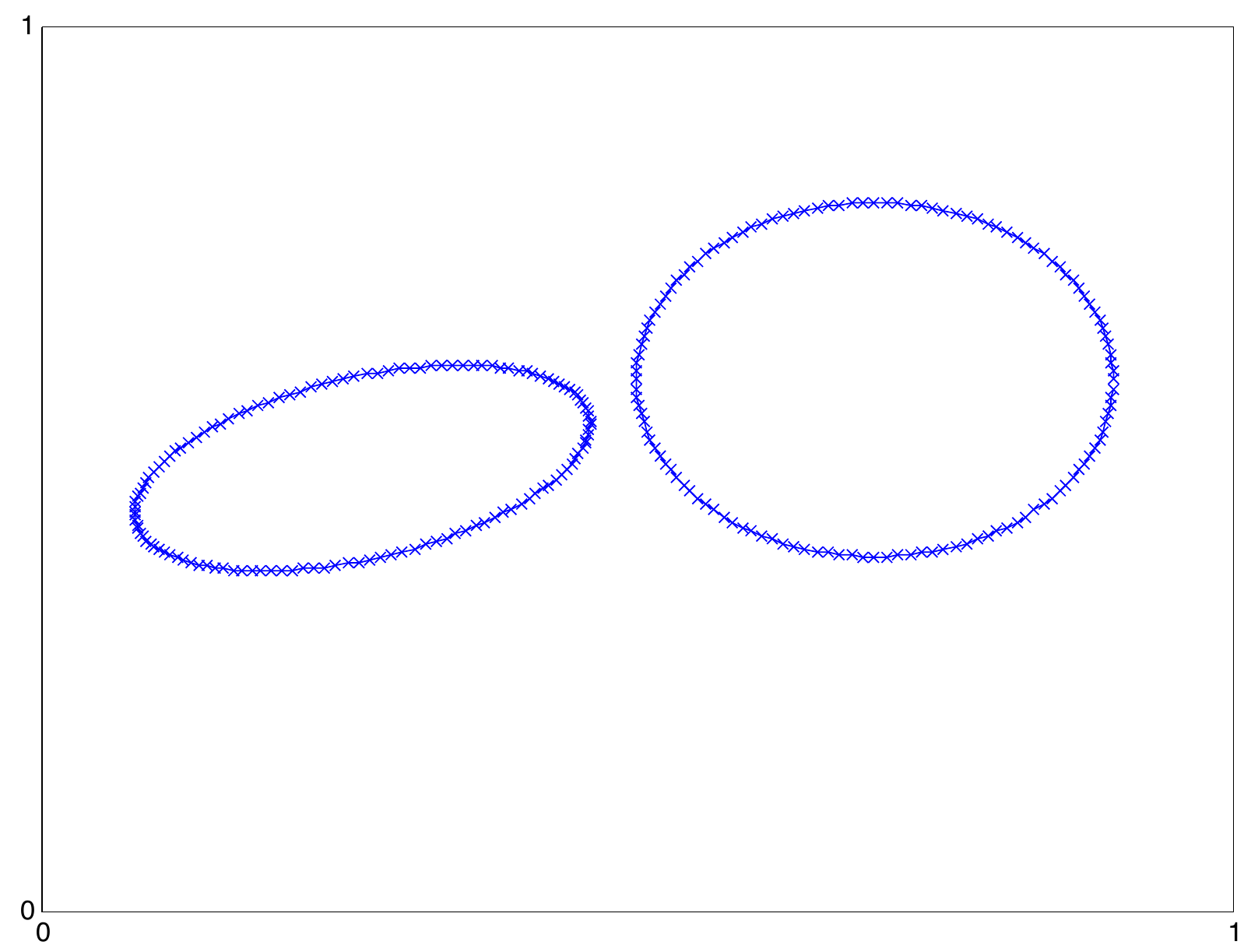}
\includegraphics[scale=0.3]{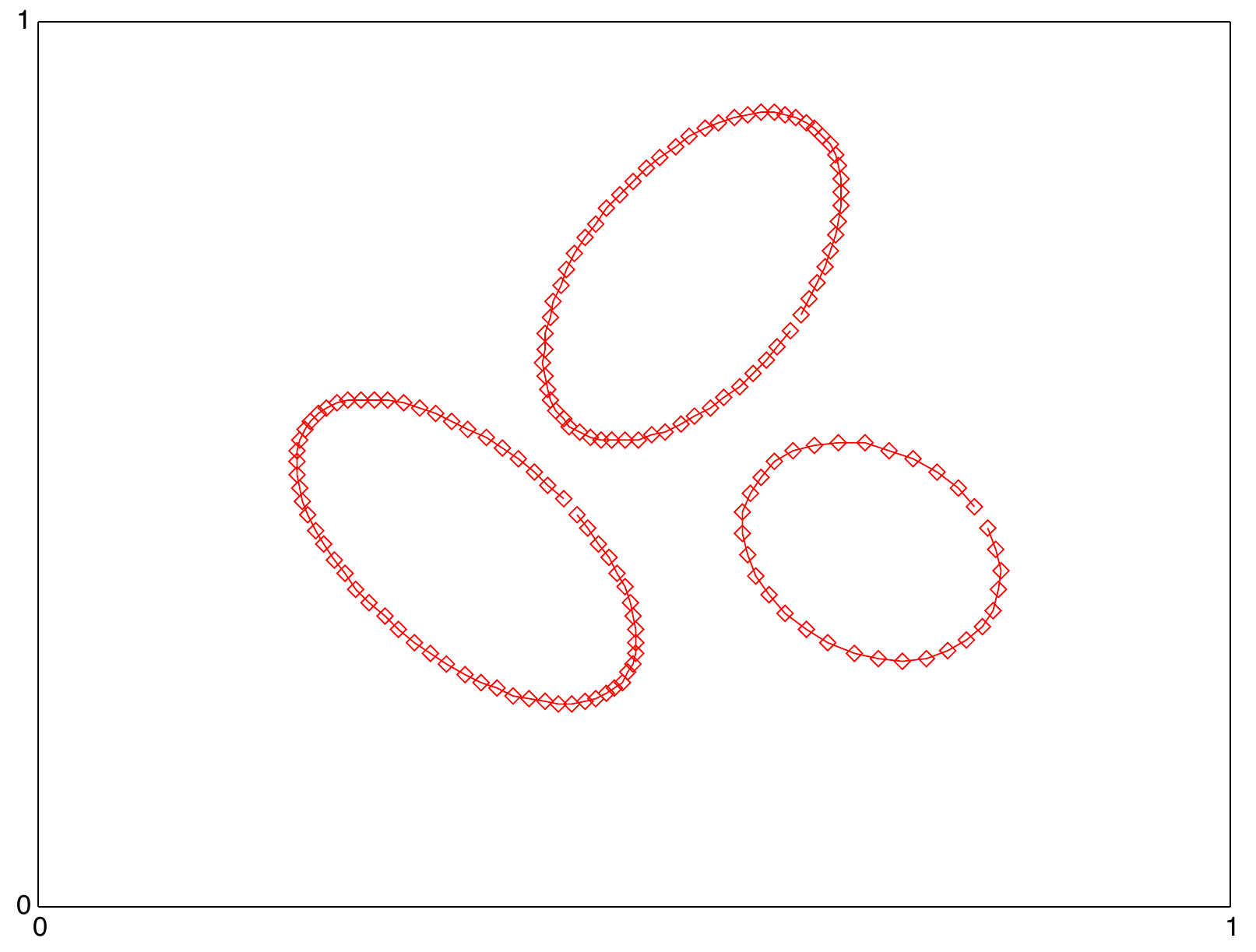}
 \caption{\it{Groups of one, two and three cells.} \label{fig5}}
\end{figure}

\begin{figure}
\centering
\includegraphics[scale=0.45]{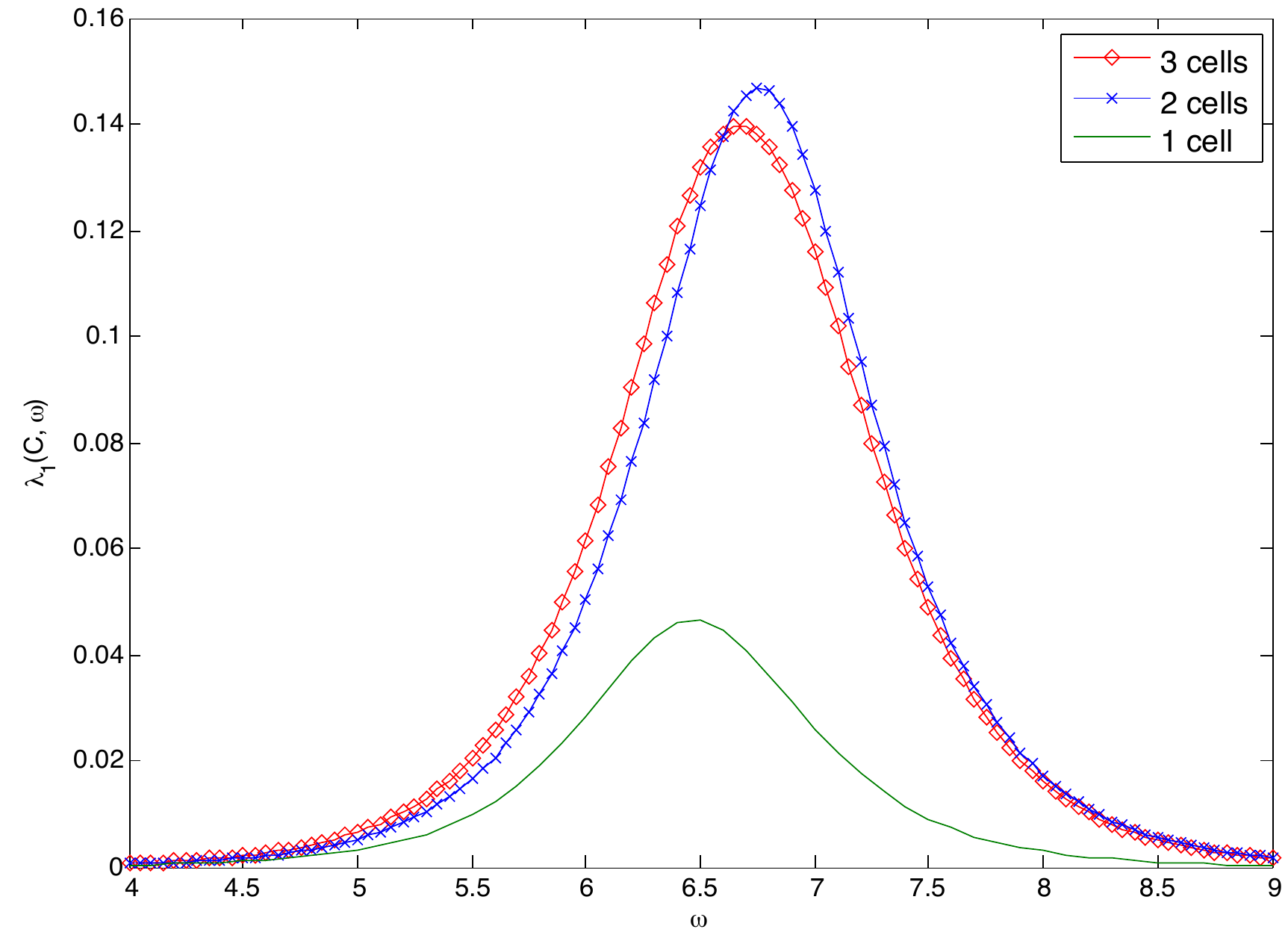}
\includegraphics[scale=0.45]{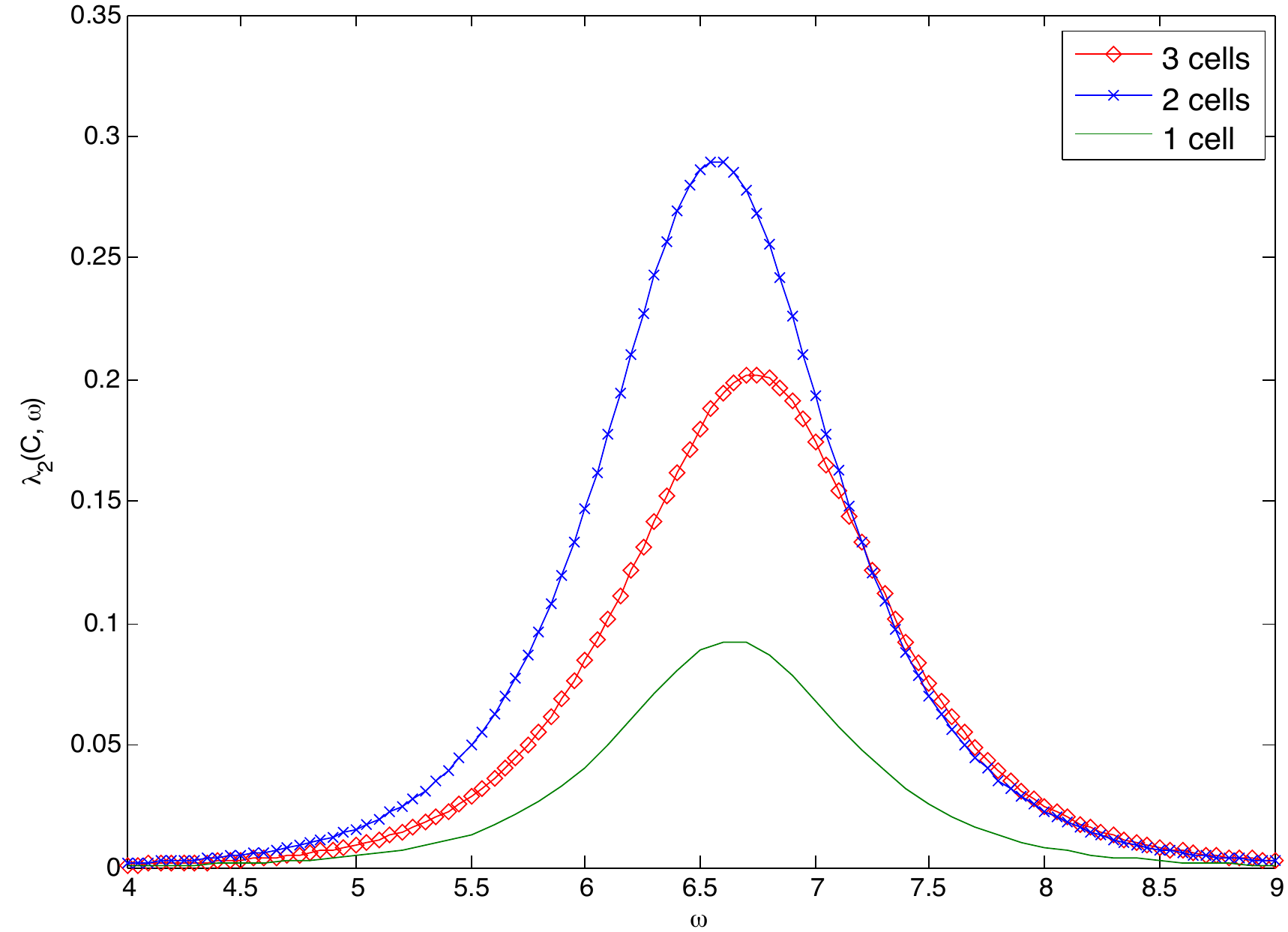}
 \caption{\it{Frequency dependence of the eigenvalues of $\Im M$ in the $3$ different cases.} \label{fig6}}
\end{figure}

\section{Concluding remarks}

In this paper we derived new formulas for the effective
admittivity of suspensions of cells and characterized their
dependance with respect to the frequency in terms of membrane
polarization tensors. We applied the formulas in the dilute case
to image suspensions of cells from electrical boundary
measurements. We presented numerical results to illustrate the use
of the Debye relaxation time in classifying microstructures. We
also developed a selective spectroscopic imaging approach. We
showed that specifying the pulse shape in terms of the relaxation
times of the dilute suspensions gives rise to selective imaging.

A challenging problem is to extend our results to elasticity
models of the cell. In \cite{mikyoung, pierre}, formulas for the
effective shear modulus and effective viscosity of dilute
suspensions of elastic inclusions were derived. On the other hand,
it was observed experimentally that the dependance of the
viscosity of a biological tissue with respect to the frequency
characterizes the microstructure \cite{tanter1,tanter2}. A
mathematical justification and modeling for this important finding
are under investigation and would be the subject of a forthcoming
paper.

\appendix

\section{Extension lemmas} \label{appendixa}

Due to the problem settings of this paper, we need to study
convergence properties of functions that are defined on the
multiple connected sets $\R_2^+$, $\Phi(\R_2^+)$ and $\eps
\Phi(\R+d^-)$. Extension operators becomes useful to treat such
functions.

Consider two open sets $U, V \subset \R^2$ with the relation $U
\subset V$, and two Sobolev spaces $W^{1,p}(U)$ and $W^{1,p}(V)$,
$p \in [1,\infty]$. What we call an {\itshape extension operator}
is a bounded linear map $P: W^{1,p}(U) \to W^{1,p}(V)$, such that
$Pu = u$ a.e. on $U$ for all $u \in W^{1,p}(U)$. In this section,
we introduce several extension operators of this kind that are
needed in the paper. They extend functions that are defined on
$Y^-$, $\R_2^+$, $\Phi(\R_2^+)$ and $\eps \Phi(\R_2^+)$ (hence
$\Depsp$) respectively.

Throughout this section, the short hand notion $\mean_A(f)$ for a
measurable set $A \subset \R^2$ with positive volume and a
function $f \in L^1(A)$ denotes the mean value of $f$ in $A$, that
is
\begin{equation}
\mean_A(f) = \frac{1}{|A|} \int_A f(x) dx. \label{eq:meandef}
\end{equation}
We start with an extension operator inside the unit cube $Y$.
Since $Y^-$ has smooth boundary, there exists an extension
operator $S: W^{1,p}(Y^+) \to W^{1,p}(Y)$ such that for all $f \in
W^{1,p}(Y^+)$ and $p \in [1,\infty)$,
\begin{equation}
\|Sf\|_{L^p(Y)} \le C\|f\|_{L^p(Y^+)}, \quad \|Sf\|_{W^{1,p}(Y)}
\le C\|f\|_{W^{1,p}(Y^+)}, \label{eq:extEvans}
\end{equation}
where $C$ only depends on $p$ and $Y^-$. Such an $S$ is given in
\cite[section 5.4]{Evans}, where the second estimate above is
given; the first estimate easily follows from their construction
as well. Cioranescu and Saint Paulin \cite{Ciopau} constructed
another extension operator which refines the second estimate
above. For the reader's convenience, we state and prove their
result in the following. Similar results can be found in
\cite{Jikov_book} as well.

\begin{thm}\label{thm:Yext} Let $Y,Y^+$ and $Y^-$ be as defined in section \ref{sec:setting}; in particular, $\partial Y^-$ is smooth. Then there exists an extension operator $P: H^1(Y^+) \to H^1(Y)$ satisfying that for any $f \in H^1(Y^+)$ and $p \in [1,\infty)$,
\begin{equation}
\|\nabla Pf\|_{L^p(Y)} \le C\|\nabla f\|_{L^p(Y^+)}, \quad
\|Pf\|_{L^p(Y)} \le C\|f\|_{L^p(Y^+)}, \label{eq:lem:Yext1}
\end{equation}
where $C$ only depends on the dimension and the set $Y^-$.
\end{thm}

\begin{proof} Recall the mean operator $\mean$ in \eqref{eq:meandef} and the extension operator $S$ in \eqref{eq:extEvans}. Given $f$, we define $Pf$ by
\begin{equation}
Pf = \mean_{Y^+}(f) + S(f-\mean_{Y^+}(f)). \label{eq:Pdef}
\end{equation}
Then by setting $\psi = f - \mean_{Y^+}(f)$, we have that
\begin{equation*}
\|\nabla Pf\|_{L^p(Y)} = \|\nabla S\psi\|_{L^p(Y)} \le
C\|\psi\|_{W^{1,p}(Y^+)} \le C\|\nabla \psi\|_{L^p(Y^+)} =
C\|\nabla f\|_{L^p(Y^+)}.
\end{equation*}
In the second inequality above, we used the Poincar\'e--Wirtinger
inequality for $\psi$ and the fact that $\psi$ is mean-zero on
$Y^+$. The $L^2$ bound of $Pf$ follows from the observation
\begin{equation*}
\|\mean_{Y^+} (f)\|_{L^p(Y)} \le
\left(\frac{|Y|}{|Y^+|}\right)^{\frac 1 p} \|f\|_{L^p(Y^+)}
\end{equation*}
and the $L^p$ estimate of $Sf$ in \eqref{eq:extEvans}. This
completes the proof.
\end{proof}

Apply the extension operator on each translated cubes in $\R_2^+$,
we get the following.

\begin{cor}\label{cor:extRdm} Recall the definition of $Y_n, Y_n^+$ and $Y_n^-$ in section \ref{sec:setting}.
Abuse notations and define
\begin{equation}
(Pu)|_{Y_n} = P(u|_{Y_n^+}), \quad n \in \Z^2, u \in W^{1,p}_{\rm
loc}(\R_2^+). \label{eq:extRdmdef}
\end{equation}
Then $P$ is an extension operator from $W^{1,p}_{\rm loc}(\R_2^+)$
to $W^{1,p}_{\rm loc}(\R^2)$. Further, with the same positive
constant $C$ in \eqref{eq:lem:Yext1} and for any $n\in \Z^2$, we
have
\begin{equation}
\|\nabla Pu\|_{L^p(Y_n)} \le C\|\nabla u\|_{L^p(Y_n^-)}, \quad
\|Pu\|_{L^p(Y_n)} \le C\|u\|_{L^p(Y_n^-)}. \label{eq:extRdm}
\end{equation}
\end{cor}

Given a diffeomorphism, the extension operator $P$ can be
transformed as follows. In the same manner, under the map of
scaling, the extension operator is naturally defined.

\begin{cor}\label{cor:extPhiRdm} Let $\Phi(\cdot,\gamma)$ be a random diffeomorphism satisfying \eqref{eq:Phic2} and \eqref{eq:Phic3}. Denote the inverse function $\Phi^{-1}$ by $\Psi$.
Define $P_{\gamma}$ as
\begin{equation}
P_\gamma u =  [P(u\circ \Phi)]\circ \Psi, \quad u \in W^{1,p}_{\rm
loc}(\Phi(\R_2^+)).
\end{equation}
Then $P_\gamma$ is an extension operator from $W^{1,p}_{\rm
loc}(\Phi(\R_2^+))$ to $W^{1,p}_{\rm loc}(\Phi(\R^2))$ which
satisfies that
\begin{equation}
\|\nabla P_\gamma u\|_{L^p(\Phi(Y_n))} \le C\|\nabla
u\|_{L^p(\Phi(Y_n^-))}, \quad \|P_\gamma u\|_{L^p(\Phi(Y_n))} \le
C\|u\|_{L^p(\Phi(Y_n^-))}, \label{eq:extPhiRdm}
\end{equation}
where the constant $C$ depends further on the constants in
\eqref{eq:Phic2} and \eqref{eq:Phic3}.
\end{cor}

\begin{cor}\label{cor:extPhieps} Let $\Phi(\cdot,\gamma)$ and $\Psi$ be as above. For each $\eps > 0$, define
$P^\eps_{\gamma}$ as follows: for any $u \in W^{1,p}_{\rm
loc}(\eps \Phi(\R_2^+))$, $P^\eps_{\gamma} u$ is defined on each
deformed and scaled cube $\eps \Phi(Y_n)$ by
\begin{equation}
P^\eps_\gamma u (x) =  \eps P \tilde{u}(\Psi(\frac{x}{\eps})),
\label{eq:Pwedef}
\end{equation}
where $\tilde{u} = \eps^{-1} u\circ \eps\Phi$ and $P$ is as in
\eqref{eq:extRdm}. Then $P^\eps_\gamma$ is an extension operator
from $W^{1,p}_{\rm loc}(\eps\Phi(\R_2^+))$ to $W^{1,p}_{\rm
loc}(\eps\Phi(\R^2))$ which satisfies that for any $n\in \Z^2$,
\begin{equation}
\|\nabla P^\eps_\gamma u\|_{L^p(\eps\Phi(Y_n))} \le C\|\nabla
u\|_{L^p(\eps\Phi(Y_n^-))}, \quad \|P^\eps_\gamma
u\|_{L^p(\eps\Phi(Y_n))} \le C\|u\|_{L^p(\eps\Phi(Y_n^-))},
\label{eq:extPhieps}
\end{equation}
where the constant $C$ depends on the same parameters as stated
below \eqref{eq:extPhiRdm}.
\end{cor}

\begin{proof}
We focus on proving \eqref{eq:extPhieps}. Under the change of
variable $x = \eps \Phi(y)$, we have
\begin{equation*}
\nabla_x P^\eps_\gamma u(x) = \nabla\Psi(\frac{x}{\eps}) \nabla_y
P\tilde{u} (\Phi^{-1}(\frac{x}{\eps}))= \nabla\Psi(\Phi(y))
\nabla_y P\tilde{u}(y),
\end{equation*}
On each deformed and scaled cube $\eps\Phi(Y_n)$, we calculate
\begin{equation*}
\begin{aligned}
\|\nabla P^\eps_\gamma u\|_{L^p(\eps \Phi(Y_n))}^p &= \int_{\eps
\Phi(Y_n)} |\nabla_x P^\eps_\gamma u(x)|^p dx = \int_{Y_n}
 |\nabla\Psi(\Phi(y)) \nabla_y P\tilde{u}(y)|^p \eps^2 \det(\nabla\Phi(y)) dy\\
&\le \eps^2 \int_{Y_n} |\nabla\Psi(\Phi(y))|^p |\nabla_y
P\tilde{u}(y)|^p \det(\nabla\Phi(y)) dy \le C\eps^2 \int_{Y_n}
|\nabla_y P \tilde{u}(y)|^p dy.
\end{aligned}
\end{equation*}
Here, we have used the Cauchy--Schwarz inequality and the bounds
\eqref{eq:Phic2}-\eqref{eq:Phic3} on the Jacobian matrix and its
determinant. Upon applying \eqref{eq:lem:Yext1}, we get
\begin{equation*}
\|\nabla P^\eps_\gamma u\|_{L^p(\eps \Phi(Y_n))}^p \le C\eps^2
\|\nabla_y \tilde{u}\|^p_{L^2(Y_n^+)}.
\end{equation*}
Since $\tilde{u}(y) = \frac{1}{\eps} u(\eps \Phi(y))$, we have
$\nabla_y \tilde{u}(y) = \nabla_y\Phi(y)\nabla_x u(\eps \Phi(y))$.
Change variables in the last integral and repeat the analysis
above to get
\begin{equation*}
\|\nabla_y \tilde{u}\|_{L^p(Y_n^+)}^p \le C\eps^{-d} \|\nabla_x
u\|_{L^p(\eps\Phi(Y_n^+))}^p.
\end{equation*}
Combining the above estimates, one finds some $C$ independent of
$\eps$ or $\gamma$ such that \eqref{eq:extPhieps} holds. Moreover,
the constant $C$ is uniform for all $\eps\Phi(Y_n)$. The $L^2$
estimate for $P^\eps_\gamma u$ is simpler and ignored. This
completes the proof.
\end{proof}

Finally, we define the extension operator from $W^{1,p}(\Depsp)$
to $W^{1,p}(\Omega)$. This is essentially the same operator in
Corollary \ref{cor:extPhieps}. Indeed, recall that $\Omega$ is
decomposed to the cushion $K_\eps$ and the cell containers
$E_\eps$; see \eqref{eq:KEdef}. We only need to apply
$P^\eps_\gamma$ in $E_\eps$.

\begin{thm}\label{thm:Dext} Let the domains $\Deps^{\pm}$, $K_\eps$ and $E_\eps$ be as defined in section \ref{sec:setting}. Let $\Phi(\cdot,\gamma)$ be a random diffeomorphism satisfying \eqref{eq:Phic2}-\eqref{eq:Phic4}. Define the linear operator $P^\eps_\gamma$ as follows: for $u \in W^{1,p}(\Depsp)$, let $P^\eps_\gamma u$ be given by \eqref{eq:Pwedef} in $E_\eps$, and let $P^\eps_\gamma u = u$ in $K_\eps$.
Then $P_\gamma^\eps$ is an extension operator from
$W^{1,p}(\Depsp)$ to $W^{1,p}(\Omega)$ and it satisfies
\begin{equation}
\label{eq:lem:Dext} \|\nabla P^\eps_\gamma u\|_{L^p(\Omega)} \le
C\|\nabla u\|_{L^p(\Depsp)}, \quad \|P^\eps_\gamma
u\|_{L^p(\Omega)} \le C\|u\|_{L^p(\Depsp)},
\end{equation}
where the constants $C$'s do not depend on $\eps$ or $\gamma$.
\end{thm}

\begin{proof} Since $P_\gamma^\eps$ leaves $u$ unchanged in $K_\eps$ and it satisfies the estimates \eqref{eq:extPhieps} uniformly in the cubes $E_\eps =
\cup_{n\in \mathcal{I}_\eps} \eps\Phi(Y_n)$, we have the
following:
\begin{equation*}
\begin{aligned}
\|\nabla P^\eps_\gamma u\|_{L^p(\Omega)}^p &= \|\nabla
f\|_{L^p(K_\eps)}^p + \sum_{n \in \mathcal{I}_\eps}
\|\nabla P^\eps_\gamma u\|_{L^p(\eps \Phi(Y_n))}^p\\
&\le \|\nabla u\|_{L^p(K_\eps)}^p + C\sum_{n\in \mathcal{I}_\eps}
\|\nabla u\|_{L^p(\eps \Phi(Y_n^+))}^p \le C\|\nabla
u\|_{L^p(\Depsp)}^p.
\end{aligned}
\end{equation*}
This completes the proof of the first estimate in
\eqref{eq:lem:Dext}. The second estimate follows in the same
manner, completing the proof.
\end{proof}

\section{Poincar\'e--Wirtinger inequality} \label{appendixb}
Our next goal is to derive a Poincar\'e--Wirtinger inequality for
functions in $H^1(\Depsp)$ with a constant independent of $\eps$
and $\gamma$. The following fact of the fluctuation of a function
is useful.
\begin{lem}\label{lem:meanineq} Let $X \subset \R^2$ be an open bounded domain with positive volume and $f \in L^1(X)$. Assume
 that $X_1 \subset X$ is a subset with positive volume, then we have
\begin{equation}
\|f - \mean_{X_1}(f)\|_{L^2(X_1)} \le \|f - \mean_X(f)\|_{L^2(X)}.
\label{eq:meanineq}
\end{equation}
\end{lem}
\begin{proof} To simplify notations, let $f_1$ be the restriction of $f$ on $X_1$, $m_1 = \mean_{X_1}(f_1)$ and $\theta_1 = |X_1|/|X|$. Similarly, let $f_2$ be the restriction of $f$ on $X_2 = X\setminus X_1$, $m_2 = \mean_{X_2}(f_2)$. Let $m = \mean_X(f)$. Then we have that
\begin{equation*}
f - m = \begin{cases}
f_1 - m_1 + (1-\theta)(m_1-m_2), & x \in X_1,\\
f_2 - m_2 + \theta(m_2 - m_1), & x \in X_2.
\end{cases}
\end{equation*}
Then basic computation plus the observation that $f_i - m_i$
integrates to zero on $X_i$ for $i=1,2$ yield the following:
\begin{equation*}
\|f-m\|^2_{L^2(X)} = \|f_1-m_1\|_{L^2(X_1)}^2  + \|f_2 -
m_2\|_{L^2(X_2)}^2 + (1-\theta)\theta |X| (m_2 - m_1)^2.
\end{equation*}
Since the items on the right-hand side are all non-negative, we
obtain \eqref{eq:meanineq}.
\end{proof}

\begin{cor}\label{cor:poincare} Assume the same conditions as in Theorem \ref{thm:Dext}. Then for any $u \in \Hq^1(\Depsp)$, we have that
\begin{equation}
\|u\|_{L^2(\Depsp)} \le C\|\nabla u\|_{L^2(\Depsp)}, \label{eq:PW}
\end{equation}
where the constant $C$ does not depend on $\eps$ or $\gamma$.
\end{cor}
\begin{proof} Thanks to Theorem \ref{thm:Dext}, we extend $u$ to $P^\eps_\gamma u$ which is in $H^1(\Omega)$. Use \eqref{eq:meanineq} and the fact that $\mean_{\Depsp}(u) = 0$ to get
$$
\|u\|_{L^2(\Depsp)} \le \|P^\eps_\gamma u - \mean_\Omega
(P^\eps_\gamma u)\|_{L^2(\Omega)}.
$$
Now apply the standard Poincar\'e--Wirtinger inequality for
functions in $H^1(\Omega)$, and then use \eqref{eq:lem:Dext}. We
get
$$
\|P^\eps_\gamma u - \mean_\Omega (P^\eps_\gamma u)\|_{L^2(\Omega)}
\le C\|\nabla P^\eps_\gamma u\|_{L^2(\Omega)} \le C\|\nabla
u\|_{L^2(\Depsp)}.
$$
The constant $C$ depends on $\Omega$ and the parameters stated in
Theorem \ref{thm:Dext} but not on $\eps$ or $\gamma$. The proof is
now complete.
\end{proof}

Another corollary of the extension lemma is that we have the
following uniform estimate when taking the trace of $u \in \Weps$
on the fixed boundary $\partial \Omega$.

\begin{cor}\label{cor:trace} Assume the same conditions as in Theorem \ref{thm:Dext}. Then there exists a constant $C$ depending
on $\Omega$ and the parameters as stated in Theorem \ref{thm:Dext}
but independent of $\eps$ and $\gamma$ such that
\begin{equation}
\|u\|_{H^{\frac 1 2}(\partial \Omega)} \le C\|\nabla
u\|_{L^2(\Depsp)}, \label{eq:Dtrace}
\end{equation}
for any $u \in H^1(\Depsp)$.
\end{cor}
\begin{proof} Thanks to Theorem \ref{thm:Dext} we extend $u$ to $P^\eps_\gamma u$ which is in $H^1(\Omega)$. The trace inequality
on $\Omega$ shows
\begin{equation}
\|P^\eps_\gamma u\|_{H^{\frac 1 2}(\partial \Omega)} \le C(\Omega)
\|P^\eps_\gamma u\|_{H^1(\Omega)}. \label{eq:cor:trace1}
\end{equation}
The desired estimate then follows from \eqref{eq:lem:Dext} and
\eqref{eq:PW}.
\end{proof}

\section{Equivalence of the two norms on $\Weps$} \label{appendixc}

In this section, we prove Proposition \ref{prop:equiv} which
establishes the equivalence between the two norms on $\Weps$. We
essentially follow \cite{Monsur} where the periodic case was
considered. The random deformation setting requires certain
modification. The details of such modifications are provided here
for the reader's convenience.

The first inequality of the proposition is proved by the following
lemma together with the Poincar\'e--Wirtinger inequality
\eqref{eq:PW}:
\begin{lem}\label{lem:L2Gamma}
There exists a constant $C$ independent of $\eps$ or $\gamma$,
such that
\begin{equation}
\|v^\pm\|_{L^2(\interface)}^2 \le
C(\eps^{-1}\|v^\pm\|_{L^2(\Deps^\pm)}^2 + \eps \|\nabla
v^\pm\|_{L^2(\Deps^\pm)}^2) \label{eq:L2Gamma}
\end{equation}
for any $v^+ \in H^1(\Depsp)$ and $v^- \in H^1(\Depsm)$.
\end{lem}
\begin{proof} According to the set-up, the interface $\interface$ consists of $\eps\Phi(\Gamma_i)$ where $i=1,\cdots, N(\eps)$ are the labels for the deformed cubes $\{\eps \Phi(Y_i)\}$
inside $\Omega$ and $\Gamma_i$ are the corresponding unit scale
interfaces.

Let us consider the case of $v^+ \in H^1(\Depsp)$; the other case
is proved in the same manner. Denote by $v_i$ the restriction of
$v^+$ on the deformed cube $\eps \Phi(Y_i)$. We lift this function
to $\tilde{v}_i(y) = v_i(\eps\Phi(y))$ which is now defined on
$Y_i^+$. For this function, we have the trace inequality
\begin{equation}
\|\tilde{v}_i\|_{L^2(\Gamma_i)}^2 \le
C(\|\tilde{v}_i\|_{L^2(Y_i^+)}^2 + \|\nabla
\tilde{v}_i\|_{L^2(Y_i^+)}^2). \label{eq:tracei}
\end{equation}
Note that this constant depends on the reference shape $Y^-$ but
is uniform in $i$.

On the other hand, because for any $\gamma \in \mathcal{O}$, the
diffeomorphism $\Phi$ satisfies \eqref{eq:Phic2} and
\eqref{eq:Phic3}, the Lebesgue measures $ds(x)$ on the curve
$\eps\Phi(\Gamma_i)$ and $ds(y)$ on $\Gamma_i$, which are related
by the change of variable $x = \eps \Phi(y)$, satisfy
\begin{equation*}
C_1  ds(x) \le \eps ds(y) \le C_2 ds(x)
\end{equation*}
for some constant $C_{1,2}$ which depend only on the constants in
the assumptions and $Y^-$ but uniform in $\eps$ and $\gamma$.

Consequently, we have
\begin{equation*}
\|v^+\|_{L^2(\interface)}^2 = \sum_{i=1}^{N(\eps)}
\int_{\eps\Phi(\Gamma_i)} |v_i(x)|^2 ds(x) \le C\eps
\sum_{i=1}^{N(\eps)} \int_{\Gamma_i} |\tilde{v}_i(y)|^2 ds(y).
\end{equation*}
Apply \eqref{eq:tracei} and change the variable back; use again
$dx \sim \eps^2 dy$ and $\nabla_y \tilde{v}_i = \eps \nabla_x v_i$
to get
\begin{equation*}
\begin{aligned}
\|v^+\|_{L^2(\interface)}^2 \le &\ C\eps \sum_{i=1}^{N(\eps)} \int_{Y_i^+} |\tilde{v}_i(y)|^2 + |\nabla_y \tilde{v}(y)|^2 dy\\
\le &\ C\eps^{-1}\sum_{i=1}^{N(\eps)} \int_{\eps\Phi(Y_i^+)}
|v_i(x)|^2 + \eps^2 |\nabla v(x)|^2 dx
\end{aligned}
\end{equation*}
This completes the proof of \eqref{eq:L2Gamma}.
\end{proof}

The other inequality in \eqref{eq:equiv} is implied by the
following lemma:
\begin{lem}\label{lem:L2fW} There exists a constant $C>0$ independent of $\eps$ or $\gamma$ such that
\begin{equation}
\|v\|_{L^2(\Depsm)} \le C\left(\sqrt{\eps} \|v\|_{L^2(\interface)}
+ \eps \|\nabla v\|_{L^2(\Depsm)}\right) \label{eq:L2fW}
\end{equation}
for all $v\in H^1(\Depsm)$.
\end{lem}
\begin{proof} We first observe that on the reference cube $Y$ with reference cell $Y^-$, we have that
\begin{equation}
\|v\|_{L^2(Y^-)}^2 \le C\left(\|v\|_{L^2(\Gamma_0)}^2 + \|\nabla
v\|_{L^2(Y^-)}^2 \right), \label{eq:Ymest}
\end{equation}
for any $v \in H^1(Y^-)$ where $C$ only depends on $Y^-$ and the
dimension. Indeed, suppose otherwise, we could find a sequence
$\{v_n\} \subset H^1(Y^-)$ such that $\|v_n\|_{L^2(Y^-)} \equiv 1$
but
\begin{equation*}
\|v_n\|_{L^2(\Gamma_0)} + \|\nabla v_n\|_{L^2(Y^-)}
\longrightarrow 0, \quad \text{as } n \to \infty.
\end{equation*}
Then since $\|v_n\|_{H^1}$ is uniformly bounded, there exists a
subsequence, still denoted as $\{v_n\}$, and a function $v\in
H^1(Y^-)$ such that
\begin{equation*}
v_n \rightharpoonup v \text{ weakly in } H^1(Y^-), \quad \nabla
v_n \rightharpoonup \nabla v \text{ weakly in } L^2(Y^-).
\end{equation*}
Consequently, $\|\nabla v\|_{L^2} \le \liminf \|\nabla v_n\|_{L^2}
= 0$, which implies that $v = C$ for some constant. Moreover,
since the embedding $H^1(Y^-) \hookrightarrow L^2(\Gamma_0)$ is
compact, the convergence $v_n \rightarrow v$ holds strongly in
$L^2(\Gamma_0)$ and $\|v\|_{L^2(\Gamma)} \le \lim
\|v_n\|_{L^2(\Gamma_0)} = 0$. Consequently $v \equiv 0$. On the
other hand, $v_n \rightarrow v$ holds strongly in $L^2(Y^-)$ and
hence $\|v\|_{L^2(Y^-)} = \lim\|v_n\|_{L^2(Y^-)} = 1$. This
contradicts with the fact that $v \equiv 0$.

To prove \eqref{eq:L2fW}, we lift functions in $\eps \Phi(Y^-_i)$
to functions in $Y^-_i$ as in the proof of the previous lemma, and
use the scaling relations of the measures: $dx \sim \eps^2 dy$ and
$ds(x) \sim \eps ds(y)$. We calculate
\begin{equation*}
\|v\|_{L^2(\Depsm)}^2 = \sum_{i=1}^{N(\eps)}
\int_{\eps\Phi(Y^-_i)} |v|^2 dx \le C\eps^2 \int_{Y^-}
|\tilde{v}|^2 dy \le C\eps^2 \sum_{i=1}^{N(\eps)} \int_{\Gamma_i}
|\tilde{v}|^2 ds + \int_{Y^-_i} |\nabla \tilde{v}|^2 dy
\end{equation*}
where in the last inequality we used \eqref{eq:Ymest}. Change the
variables back to get
\begin{equation*}
\|v\|_{L^2(\Depsm)}^2 \le C\eps^2 \sum_{i=1}^{N(\eps)}
\int_{\eps\Phi(\Gamma_i)} \eps^{-d+1}|v|^2 ds +
\int_{\eps\Phi(Y^-_i)} \eps^{-d+2}|\nabla v|^2 dy.
\end{equation*}
Note that we used again $\nabla_y \tilde{v} = \eps\nabla_x v$. The
above inequality is precisely \eqref{eq:L2fW}.
\end{proof}

\begin{proof}[Proof of Proposition \ref{prop:equiv}] To prove the first inequality, we apply Lemma \ref{lem:L2Gamma} to get
\begin{equation*}
\begin{aligned}
\eps\|u^+ - u^-\|_{L^2(\interface)}^2 \le&\ 2(\eps\|u^+\|_{L^2(\interface)}^2 +\|u^-\|_{L^2(\interface)}^2)\\
\le&\ C(\|u^+\|^2_{L^2(\Depsp)} + \|u^-\|^2_{L^2(\Depsm)} +
\eps^2\|\nabla u^+\|_{L^2(\Depsp)}^2 + \eps^2\|\nabla
u^+\|_{L^2(\Depsp)}^2).
\end{aligned}
\end{equation*}
Only the first term in \eqref{eq:PW} does not show in
$\|\cdot\|_{\Hq^1\times H^1}$, but it is controlled by $\|\nabla
u^+\|_{L^2(\Depsp)}$ uniformly in $\eps$ and $\gamma$ thanks to
\eqref{eq:PW}.

For the second inequality, we only need to control
$\|u^-\|_{L^2(\Depsm)}$. We apply Lemma \ref{lem:L2fW} and the
triangle inequality:
\begin{equation*}
\|u^-\|_{L^2(\Depsm)}^2 \le C\left(
\eps\|u^+\|_{L^2(\interface)}^2 + \eps\|u^+ -
u^-\|_{L^2(\interface)}^2 + \eps^2 \|\nabla u^-\|_{L^2(\Depsm)}^2
\right).
\end{equation*}
Only the first term does not appear in $\|\cdot\|_{\Weps}$, but
using Lemma \ref{lem:L2Gamma} and \eqref{eq:PW} we can bound it by
\begin{equation*}
\eps\|u^+\|_{L^2(\interface)}^2 \le C(\|u^+\|_{L^2(\Depsp)}^2 +
\eps^2 \|\nabla u^+\|_{L^2(\Depsp)}^2) \le C\|\nabla
u^+\|_{L^2(\Depsp)}^2.
\end{equation*}
This completes the proof.
\end{proof}

\section{Technical lemma}

\begin{lem}\label{lem:theta}
Let $\varphi_1$ be a function in $\mathcal{D}(\Omega,
C^{\infty}_{\sharp}(Y^+))\times \mathcal{D}(\Omega,
C^{\infty}_{\sharp}(Y^-))$. There exists at least one function
$\theta$ in $(\mathcal{D}(\Omega, H^1_{\sharp}(Y^+)) \times
\mathcal{D}(\Omega, H^1_{\sharp}(Y^-)))^2$ solution of the
following problem:
\begin{equation}\label{eq:theta}
\left \{
\begin{array}{ll}
\vspace{0.3cm} \displaystyle- \nabla_y \cdot \theta^+(x, y) =0 & \textrm{in} \, \,Y^+ ,\\
\vspace{0.3cm} \displaystyle -\nabla_y \cdot \theta^-(x,y) =0 & \textrm{in} \, \,Y^- ,\\
\vspace{0.3cm} \theta^+(x, y) \cdot n = \theta^-(x, y) \cdot n & \textrm{on} \, \, \Gamma ,\\
\vspace{0.3cm}\displaystyle \theta^+(x, y) \cdot n= \varphi_1^+(x,y) - \varphi_1^-(x,y) & \textrm{on} \, \Gamma ,\\
y\longmapsto \theta(x,y) Y-\textrm{periodic}. &
\end{array}
\right .
\end{equation}
\end{lem}
\begin{proof}
We look for a solution under the form $\theta = \nabla_y \eta$. We
hence introduce the following variational problem:
\begin{equation*}
\left \{
\begin{array}{l}
\textrm{Find}\, \eta \in (H^1_{\sharp}(Y^+)/\mathbb{C})\times (H^1_{\sharp}(Y^-)/\mathbb{C})\, \textrm{such that}\\
\displaystyle \int_{Y^+} \nabla \eta^+(y) \cdot \overline{\psi}^+(y) dy +  \int_{Y^-} \nabla \eta^-(y) \cdot \overline{\psi}^-(y) dy \\
\hspace{6.5cm}\displaystyle=\frac{1}{\beta k_0} \int_{\Gamma}(\varphi_1^+-\varphi^-_1)(\overline{\psi}^+ - \overline{\psi}^-)(y) ds(y),\\
\textrm{for all}\, \psi \in (H^1_{\sharp}(Y^+)/\mathbb{C})\times
(H^1_{\sharp}(Y^-)/\mathbb{C}),
\end{array}
\right.
\end{equation*}
for a fixed $x \in \Omega$. Lax-Milgram theorem gives us existence
and uniqueness of such an $\eta$.  Since $\varphi_1 \in
\mathcal{D}(\Omega, C^{\infty}_{\sharp}(Y^+)) \times
\mathcal{D}(\Omega, C^{\infty}_{\sharp}(Y^-))$, there exists at
least one function $\theta \in (\mathcal{D}(\Omega,
H^1_{\sharp}(Y^+)) \times \mathcal{D}(\Omega,
H^1_{\sharp}(Y^-))^2$ solution of (\ref{eq:theta}). Note that we
do not have uniqueness of such a solution.
\end{proof}

\end{document}